\documentclass[a4paper,11pt]{article}
\usepackage{MyStyle}

\usepackage{leftindex}

\usepackage{algorithm}
\floatname{algorithm}{Method}
\usepackage{algorithmic}

\newcolumntype{C}{>{$}c<{$}}
\usepackage{planarforest}
\newcommand{\forestA}{
\tikz[planar forest ] {

\node [b] at (0.0, 0.0) {  } 
;
}}

\newcommand{\forestB}{
\tikz[planar forest ] {

\node [lc] at (0.0, 0.0) { 1 } 
;
\node [lc] at (1.0, 0.0) { 1 } 
;
}}

\newcommand{\forestC}{
\tikz[planar forest ] {

\node [b] at (0.0, 0.0) {  } 
child {node [b] at (0.0, 1.0) {  }  
}
;
}}

\newcommand{\forestD}{
\tikz[planar forest ] {

\node [b] at (0.0, 0.0) {  } 
child {node [lc] at (-0.5, 1.0) { 1 }  
}
child {node [lc] at (0.5, 1.0) { 1 }  
}
;
}}

\newcommand{\forestE}{
\tikz[planar forest ] {

\node [b] at (0.0, 0.0) {  } 
;
\node [b] at (1.0, 0.0) {  } 
;
}}

\newcommand{\forestF}{
\tikz[planar forest ] {

\node [b] at (0.0, 0.0) {  } 
child {node [lc] at (0.0, 1.0) { 1 }  
}
;
\node [lc] at (1.0, 0.0) { 1 } 
;
}}

\newcommand{\forestG}{
\tikz[planar forest ] {

\node [lc] at (0.0, 0.0) { 1 } 
;
\node [b] at (1.0, 0.0) {  } 
child {node [lc] at (0.0, 1.0) { 1 }  
}
;
}}

\newcommand{\forestH}{
\tikz[planar forest ] {

\node [b] at (0.0, 0.0) {  } 
;
\node [lc] at (1.0, 0.0) { 1 } 
;
\node [lc] at (2.0, 0.0) { 1 } 
;
}}

\newcommand{\forestI}{
\tikz[planar forest ] {

\node [lc] at (0.0, 0.0) { 1 } 
;
\node [b] at (1.0, 0.0) {  } 
;
\node [lc] at (2.0, 0.0) { 1 } 
;
}}

\newcommand{\forestJ}{
\tikz[planar forest ] {

\node [lc] at (0.0, 0.0) { 1 } 
;
\node [lc] at (1.0, 0.0) { 1 } 
;
\node [b] at (2.0, 0.0) {  } 
;
}}

\newcommand{\forestK}{
\tikz[planar forest ] {

\node [lc] at (0.0, 0.0) { 2 } 
;
\node [lc] at (1.0, 0.0) { 2 } 
;
\node [lc] at (2.0, 0.0) { 1 } 
;
\node [lc] at (3.0, 0.0) { 1 } 
;
}}

\newcommand{\forestL}{
\tikz[planar forest ] {

\node [lc] at (0.0, 0.0) { 2 } 
;
\node [lc] at (1.0, 0.0) { 1 } 
;
\node [lc] at (2.0, 0.0) { 2 } 
;
\node [lc] at (3.0, 0.0) { 1 } 
;
}}

\newcommand{\forestM}{
\tikz[planar forest ] {

\node [lc] at (0.0, 0.0) { 1 } 
;
\node [lc] at (1.0, 0.0) { 2 } 
;
\node [lc] at (2.0, 0.0) { 2 } 
;
\node [lc] at (3.0, 0.0) { 1 } 
;
}}

\newcommand{\forestN}{
\tikz[planar forest ] {

\node [b] at (0.0, 0.0) {  } 
;
}}

\newcommand{\forestO}{
\tikz[planar forest ] {

\node [b] at (0.0, 0.0) {  } 
;
\node [b] at (1.0, 0.0) {  } 
;
}}

\newcommand{\forestP}{
\tikz[planar forest ] {

\node [b] at (0.0, 0.0) {  } 
;
}}

\newcommand{\forestQ}{
\tikz[planar forest ] {

\node [lc] at (0.0, 0.0) { 1 } 
;
\node [lc] at (1.0, 0.0) { 1 } 
;
}}

\newcommand{\forestR}{
\tikz[planar forest ] {

\node [b] at (0.0, 0.0) {  } 
;
\node [lc] at (1.0, 0.0) { 1 } 
;
\node [lc] at (2.0, 0.0) { 1 } 
;
}}

\newcommand{\forestS}{
\tikz[planar forest ] {

\node [lc] at (0.0, 0.0) { 1 } 
;
\node [lc] at (1.0, 0.0) { 1 } 
;
\node [b] at (2.0, 0.0) {  } 
;
}}

\newcommand{\forestT}{
\tikz[planar forest ] {

\node [lc] at (0.0, 0.0) { 1 } 
;
\node [b] at (1.0, 0.0) {  } 
;
\node [lc] at (2.0, 0.0) { 1 } 
;
}}

\newcommand{\forestU}{
\tikz[planar forest ] {

\node [lc] at (0.0, 0.0) { 1 } 
;
\node [lc] at (1.0, 0.0) { 1 } 
;
}}

\newcommand{\forestV}{
\tikz[planar forest ] {

\node [lc] at (0.0, 0.0) { 2 } 
;
\node [lc] at (1.0, 0.0) { 2 } 
;
\node [lc] at (2.0, 0.0) { 1 } 
;
\node [lc] at (3.0, 0.0) { 1 } 
;
}}

\newcommand{\forestW}{
\tikz[planar forest ] {

\node [lc] at (0.0, 0.0) { 2 } 
;
\node [lc] at (1.0, 0.0) { 1 } 
;
\node [lc] at (2.0, 0.0) { 2 } 
;
\node [lc] at (3.0, 0.0) { 1 } 
;
}}

\newcommand{\forestX}{
\tikz[planar forest ] {

\node [lc] at (0.0, 0.0) { 1 } 
;
\node [lc] at (1.0, 0.0) { 2 } 
;
\node [lc] at (2.0, 0.0) { 2 } 
;
\node [lc] at (3.0, 0.0) { 1 } 
;
}}

\newcommand{\forestY}{
\tikz[planar forest ] {

\node [b] at (0.0, 0.0) {  } 
child {node [w,label={ [label distance=-1mm]90:{ \scriptsize 1 } }] at (0.0, 1.0) {  }  
}
;
\node [b] at (1.0, 0.0) {  } 
;
}}

\newcommand{\forestAB}{
\tikz[planar forest ] {

\node [b] at (0.0, 0.0) {  } 
child {node [w,label={ [label distance=-1mm]90:{ \scriptsize {1} } }] at (0.0, 1.0) {  }  
}
;
\node [b] at (1.0, 0.0) {  } 
;
}}

\newcommand{\forestBB}{
\tikz[planar forest ] {

\node [b,label={ [label distance=-1mm]0:{ \scriptsize j } }] at (0.0, 0.0) {  } 
child {node [w,label={ [label distance=-1mm]90:{ \scriptsize {1i} } }] at (0.0, 1.0) {  }  
}
;
\node [b,label={ [label distance=-1mm]0:{ \scriptsize k } }] at (1.0, 0.0) {  } 
;
}}

\newcommand{\forestCB}{
\tikz[planar forest ] {

\node [b] at (0.0, 0.0) {  } 
child {node [w,label={ [label distance=-1mm]90:{ \scriptsize {1} } }] at (0.0, 1.0) {  }  
}
;
\node [b] at (1.0, 0.0) {  } 
;
}}

\newcommand{\forestDB}{
\tikz[planar forest ] {

\node [b] at (0.0, 0.0) {  } 
child {node [w,label={ [label distance=-1mm]90:{ \scriptsize 1 } }] at (0.0, 1.0) {  }  
}
;
}}

\newcommand{\forestEB}{
\tikz[planar forest ] {

\node [b] at (0.0, 0.0) {  } 
;
\node [w,label={ [label distance=-1mm]90:{ \scriptsize 2 } }] at (1.0, 0.0) {  } 
;
}}

\newcommand{\forestFB}{
\tikz[planar forest ] {

\node [b] at (0.0, 0.0) {  } 
child {node [w,label={ [label distance=-1mm]90:{ \scriptsize 1 } }] at (0.0, 1.0) {  }  
}
;
\node [b] at (1.0, 0.0) {  } 
;
\node [w,label={ [label distance=-1mm]90:{ \scriptsize 2 } }] at (2.0, 0.0) {  } 
;
}}

\newcommand{\forestGB}{
\tikz[planar forest ] {

\node [b] at (0.0, 0.0) {  } 
;
\node [b] at (1.0, 0.0) {  } 
child {node [w,label={ [label distance=-1mm]90:{ \scriptsize 1 } }] at (0.0, 1.0) {  }  
}
;
\node [w,label={ [label distance=-1mm]90:{ \scriptsize 2 } }] at (2.0, 0.0) {  } 
;
}}

\newcommand{\forestHB}{
\tikz[planar forest ] {

\node [b] at (0.0, 0.0) {  } 
;
\node [w,label={ [label distance=-1mm]90:{ \scriptsize 2 } }] at (1.0, 0.0) {  } 
;
\node [b] at (2.0, 0.0) {  } 
child {node [w,label={ [label distance=-1mm]90:{ \scriptsize 1 } }] at (0.0, 1.0) {  }  
}
;
}}

\newcommand{\forestIB}{
\tikz[planar forest ] {

\node [b] at (0.0, 0.0) {  } 
child {node [w,label={ [label distance=-1mm]90:{ \scriptsize 1 } }] at (0.0, 1.0) {  }  
}
;
\node [w,label={ [label distance=-1mm]90:{ \scriptsize 1 } }] at (1.0, 0.0) {  } 
;
}}

\newcommand{\forestJB}{
\tikz[planar forest ] {

\node [b] at (0.0, 0.0) {  } 
;
\node [w,label={ [label distance=-1mm]90:{ \scriptsize 2 } }] at (1.0, 0.0) {  } 
;
}}

\newcommand{\forestKB}{
\tikz[planar forest ] {

\node [b] at (0.0, 0.0) {  } 
child {node [w,label={ [label distance=-1mm]90:{ \scriptsize 1 } }] at (0.0, 1.0) {  }  
}
;
\node [w,label={ [label distance=-1mm]90:{ \scriptsize 1 } }] at (1.0, 0.0) {  } 
;
\node [b] at (2.0, 0.0) {  } 
;
\node [w,label={ [label distance=-1mm]90:{ \scriptsize 2 } }] at (3.0, 0.0) {  } 
;
}}

\newcommand{\forestLB}{
\tikz[planar forest ] {

\node [b] at (0.0, 0.0) {  } 
child {node [w,label={ [label distance=-1mm]90:{ \scriptsize 1 } }] at (0.0, 1.0) {  }  
}
;
\node [b] at (1.0, 0.0) {  } 
;
\node [w,label={ [label distance=-1mm]90:{ \scriptsize 1 } }] at (2.0, 0.0) {  } 
;
\node [w,label={ [label distance=-1mm]90:{ \scriptsize 2 } }] at (3.0, 0.0) {  } 
;
}}

\newcommand{\forestMB}{
\tikz[planar forest ] {

\node [b] at (0.0, 0.0) {  } 
child {node [w,label={ [label distance=-1mm]90:{ \scriptsize 1 } }] at (0.0, 1.0) {  }  
}
;
\node [b] at (1.0, 0.0) {  } 
;
\node [w,label={ [label distance=-1mm]90:{ \scriptsize 2 } }] at (2.0, 0.0) {  } 
;
\node [w,label={ [label distance=-1mm]90:{ \scriptsize 1 } }] at (3.0, 0.0) {  } 
;
}}

\newcommand{\forestNB}{
\tikz[planar forest ] {

\node [b] at (0.0, 0.0) {  } 
;
\node [b] at (1.0, 0.0) {  } 
child {node [w,label={ [label distance=-1mm]90:{ \scriptsize 1 } }] at (0.0, 1.0) {  }  
}
;
\node [w,label={ [label distance=-1mm]90:{ \scriptsize 1 } }] at (2.0, 0.0) {  } 
;
\node [w,label={ [label distance=-1mm]90:{ \scriptsize 2 } }] at (3.0, 0.0) {  } 
;
}}

\newcommand{\forestOB}{
\tikz[planar forest ] {

\node [b] at (0.0, 0.0) {  } 
;
\node [b] at (1.0, 0.0) {  } 
child {node [w,label={ [label distance=-1mm]90:{ \scriptsize 1 } }] at (0.0, 1.0) {  }  
}
;
\node [w,label={ [label distance=-1mm]90:{ \scriptsize 2 } }] at (2.0, 0.0) {  } 
;
\node [w,label={ [label distance=-1mm]90:{ \scriptsize 1 } }] at (3.0, 0.0) {  } 
;
}}

\newcommand{\forestPB}{
\tikz[planar forest ] {

\node [b] at (0.0, 0.0) {  } 
;
\node [w,label={ [label distance=-1mm]90:{ \scriptsize 2 } }] at (1.0, 0.0) {  } 
;
\node [b] at (2.0, 0.0) {  } 
child {node [w,label={ [label distance=-1mm]90:{ \scriptsize 1 } }] at (0.0, 1.0) {  }  
}
;
\node [w,label={ [label distance=-1mm]90:{ \scriptsize 1 } }] at (3.0, 0.0) {  } 
;
}}

\newcommand{\forestQB}{
\tikz[planar forest ] {

\node [b] at (0.0, 0.0) {  } 
;
}}

\newcommand{\forestRB}{
\tikz[planar forest ] {

\node [b] at (0.0, 0.0) {  } 
child {node [b] at (0.0, 1.0) {  }  
}
;
}}

\newcommand{\forestSB}{
\tikz[planar forest ] {

\node [b] at (0.0, 0.0) {  } 
child {node [lc] at (-0.5, 1.0) { 1 }  
}
child {node [lc] at (0.5, 1.0) { 1 }  
}
;
}}

\newcommand{\forestTB}{
\tikz[planar forest ] {

\node [b] at (0.0, 0.0) {  } 
child {node [b] at (-0.5, 1.0) {  }  
}
child {node [b] at (0.5, 1.0) {  }  
}
;
}}

\newcommand{\forestUB}{
\tikz[planar forest ] {

\node [b] at (0.0, 0.0) {  } 
child {node [b] at (0.0, 1.0) {  }  
child {node [b] at (0.0, 1.0) {  }  
}
}
;
}}

\newcommand{\forestVB}{
\tikz[planar forest ] {

\node [b] at (0.0, 0.0) {  } 
child {node [b] at (-1.0, 1.0) {  }  
}
child {node [lc] at (0.0, 1.0) { 1 }  
}
child {node [lc] at (1.0, 1.0) { 1 }  
}
;
}}

\newcommand{\forestWB}{
\tikz[planar forest ] {

\node [b] at (0.0, 0.0) {  } 
child {node [lc] at (-1.0, 1.0) { 1 }  
}
child {node [b] at (0.0, 1.0) {  }  
}
child {node [lc] at (1.0, 1.0) { 1 }  
}
;
}}

\newcommand{\forestXB}{
\tikz[planar forest ] {

\node [b] at (0.0, 0.0) {  } 
child {node [lc] at (-1.0, 1.0) { 1 }  
}
child {node [lc] at (0.0, 1.0) { 1 }  
}
child {node [b] at (1.0, 1.0) {  }  
}
;
}}

\newcommand{\forestYB}{
\tikz[planar forest ] {

\node [b] at (0.0, 0.0) {  } 
child {node [b] at (-0.5, 1.0) {  }  
child {node [lc] at (0.0, 1.0) { 1 }  
}
}
child {node [lc] at (0.5, 1.0) { 1 }  
}
;
}}

\newcommand{\forestAC}{
\tikz[planar forest ] {

\node [b] at (0.0, 0.0) {  } 
child {node [lc] at (-0.5, 1.0) { 1 }  
}
child {node [b] at (0.5, 1.0) {  }  
child {node [lc] at (0.0, 1.0) { 1 }  
}
}
;
}}

\newcommand{\forestBC}{
\tikz[planar forest ] {

\node [b] at (0.0, 0.0) {  } 
child {node [b] at (0.0, 1.0) {  }  
child {node [lc] at (-0.5, 1.0) { 1 }  
}
child {node [lc] at (0.5, 1.0) { 1 }  
}
}
;
}}

\newcommand{\forestCC}{
\tikz[planar forest ] {

\node [b] at (0.0, 0.0) {  } 
child {node [lc] at (-1.5, 1.0) { 1 }  
}
child {node [lc] at (-0.5, 1.0) { 1 }  
}
child {node [lc] at (0.5, 1.0) { 2 }  
}
child {node [lc] at (1.5, 1.0) { 2 }  
}
;
}}

\newcommand{\forestDC}{
\tikz[planar forest ] {

\node [b] at (0.0, 0.0) {  } 
child {node [lc] at (-1.5, 1.0) { 1 }  
}
child {node [lc] at (-0.5, 1.0) { 2 }  
}
child {node [lc] at (0.5, 1.0) { 2 }  
}
child {node [lc] at (1.5, 1.0) { 1 }  
}
;
}}

\newcommand{\forestEC}{
\tikz[planar forest ] {

\node [b] at (0.0, 0.0) {  } 
child {node [lc] at (-1.5, 1.0) { 1 }  
}
child {node [lc] at (-0.5, 1.0) { 2 }  
}
child {node [lc] at (0.5, 1.0) { 1 }  
}
child {node [lc] at (1.5, 1.0) { 2 }  
}
;
}}

\newcommand{\forestFC}{
\tikz[planar forest ] {

\node [b] at (0.0, 0.0) {  } 
child {node [b] at (-0.5, 1.0) {  }  
child {node [lc] at (-0.5, 1.0) { 1 }  
}
child {node [b] at (0.5, 1.0) {  }  
child {node [lc] at (-0.5, 1.0) { 2 }  
}
child {node [lc] at (0.5, 1.0) { 2 }  
}
}
}
child {node [lc] at (0.5, 1.0) { 1 }  
}
;
}}

\newcommand{\forestGC}{
\tikz[planar forest ] {

\node [b] at (0.0, 0.0) {  } 
child {node [b] at (-0.5, 1.0) {  }  
child {node [lc] at (-0.5, 1.0) { 3 }  
}
child {node [b] at (0.5, 1.0) {  }  
child {node [lc] at (-0.5, 1.0) { 1 }  
}
child {node [lc] at (0.5, 1.0) { 1 }  
}
}
}
child {node [lc] at (0.5, 1.0) { 3 }  
}
;
}}

\newcommand{\forestHC}{
\tikz[planar forest ] {

\node [lc] at (0.0, 0.0) { 1 } 
;
\node [lc] at (1.0, 0.0) { 1 } 
;
}}

\newcommand{\forestIC}{
\tikz[planar forest ] {

\node [b] at (0.0, 0.0) {  } 
child {node [lc] at (-0.5, 1.0) { 1 }  
}
child {node [b] at (0.5, 1.0) {  }  
child {node [lc] at (-0.5, 1.0) { 3 }  
}
child {node [lc] at (0.5, 1.0) { 3 }  
}
}
;
\node [lc] at (1.25, 0.0) { 1 } 
;
\node [b] at (2.5, 0.0) {  } 
child {node [b] at (-0.5, 1.0) {  }  
child {node [lc] at (0.0, 1.0) { 2 }  
}
}
child {node [b] at (0.5, 1.0) {  }  
child {node [lc] at (0.0, 1.0) { 2 }  
}
}
;
}}

\newcommand{\forestJC}{
\tikz[planar forest ] {

\node [b] at (0.0, 0.0) {  } 
child {node [w,label={ [label distance=-1mm]90:{ \scriptsize i } }] at (-0.5, 1.0) {  }  
}
child {node [b] at (0.5, 1.0) {  }  
child {node [w,label={ [label distance=-1mm]90:{ \scriptsize k } }] at (-0.5, 1.0) {  }  
}
child {node [w,label={ [label distance=-1mm]90:{ \scriptsize k } }] at (0.5, 1.0) {  }  
}
}
;
\node [w,label={ [label distance=-1mm]90:{ \scriptsize i } }] at (1.25, 0.0) {  } 
;
\node [b] at (2.5, 0.0) {  } 
child {node [b] at (-0.5, 1.0) {  }  
child {node [w,label={ [label distance=-1mm]90:{ \scriptsize j } }] at (0.0, 1.0) {  }  
}
}
child {node [b] at (0.5, 1.0) {  }  
child {node [w,label={ [label distance=-1mm]90:{ \scriptsize j } }] at (0.0, 1.0) {  }  
}
}
;
}}

\newcommand{\forestKC}{
\tikz[planar forest ] {

\node [b] at (0.0, 0.0) {  } 
child {node [lc] at (-0.5, 1.0) { 1 }  
}
child {node [b] at (0.5, 1.0) {  }  
child {node [lc] at (-0.5, 1.0) { 3 }  
}
child {node [lc] at (0.5, 1.0) { 3 }  
}
}
;
\node [lc] at (1.25, 0.0) { 2 } 
;
\node [b] at (2.5, 0.0) {  } 
child {node [b] at (-0.5, 1.0) {  }  
child {node [lc] at (0.0, 1.0) { 2 }  
}
}
child {node [b] at (0.5, 1.0) {  }  
child {node [lc] at (0.0, 1.0) { 1 }  
}
}
;
}}

\newcommand{\forestLC}{
\tikz[planar forest ] {

\node [b] at (0.0, 0.0) {  } 
child {node [w,label={ [label distance=-1mm]90:{ \scriptsize i } }] at (-0.5, 1.0) {  }  
}
child {node [b] at (0.5, 1.0) {  }  
child {node [w,label={ [label distance=-1mm]90:{ \scriptsize k } }] at (-0.5, 1.0) {  }  
}
child {node [w,label={ [label distance=-1mm]90:{ \scriptsize k } }] at (0.5, 1.0) {  }  
}
}
;
\node [w,label={ [label distance=-1mm]90:{ \scriptsize j } }] at (1.25, 0.0) {  } 
;
\node [b] at (2.5, 0.0) {  } 
child {node [b] at (-0.5, 1.0) {  }  
child {node [w,label={ [label distance=-1mm]90:{ \scriptsize j } }] at (0.0, 1.0) {  }  
}
}
child {node [b] at (0.5, 1.0) {  }  
child {node [w,label={ [label distance=-1mm]90:{ \scriptsize i } }] at (0.0, 1.0) {  }  
}
}
;
}}

\newcommand{\forestMC}{
\tikz[planar forest ] {

\node [b] at (0.0, 0.0) {  } 
child {node [lc] at (-0.5, 1.0) { 1 }  
}
child {node [b] at (0.5, 1.0) {  }  
child {node [lc] at (-0.5, 1.0) { 3 }  
}
child {node [lc] at (0.5, 1.0) { 3 }  
}
}
;
\node [lc] at (1.25, 0.0) { 2 } 
;
\node [b] at (2.5, 0.0) {  } 
child {node [b] at (-0.5, 1.0) {  }  
child {node [lc] at (0.0, 1.0) { 1 }  
}
}
child {node [b] at (0.5, 1.0) {  }  
child {node [lc] at (0.0, 1.0) { 2 }  
}
}
;
}}

\newcommand{\forestNC}{
\tikz[planar forest ] {

\node [b] at (0.0, 0.0) {  } 
child {node [w,label={ [label distance=-1mm]90:{ \scriptsize i } }] at (-0.5, 1.0) {  }  
}
child {node [b] at (0.5, 1.0) {  }  
child {node [w,label={ [label distance=-1mm]90:{ \scriptsize k } }] at (-0.5, 1.0) {  }  
}
child {node [w,label={ [label distance=-1mm]90:{ \scriptsize k } }] at (0.5, 1.0) {  }  
}
}
;
\node [w,label={ [label distance=-1mm]90:{ \scriptsize j } }] at (1.25, 0.0) {  } 
;
\node [b] at (2.5, 0.0) {  } 
child {node [b] at (-0.5, 1.0) {  }  
child {node [w,label={ [label distance=-1mm]90:{ \scriptsize i } }] at (0.0, 1.0) {  }  
}
}
child {node [b] at (0.5, 1.0) {  }  
child {node [w,label={ [label distance=-1mm]90:{ \scriptsize j } }] at (0.0, 1.0) {  }  
}
}
;
}}

\newcommand{\forestOC}{
\tikz[planar forest ] {

\node [b] at (0.0, 0.0) {  } 
child {node [w,label={ [label distance=-1mm]90:{ \scriptsize 1 } }] at (-0.5, 1.0) {  }  
}
child {node [b] at (0.5, 1.0) {  }  
child {node [w,label={ [label distance=-1mm]90:{ \scriptsize 2 } }] at (-0.5, 1.0) {  }  
}
child {node [w,label={ [label distance=-1mm]90:{ \scriptsize 2 } }] at (0.5, 1.0) {  }  
}
}
;
\node [w,label={ [label distance=-1mm]90:{ \scriptsize 1 } }] at (1.25, 0.0) {  } 
;
\node [b] at (2.5, 0.0) {  } 
child {node [b] at (-0.5, 1.0) {  }  
child {node [w,label={ [label distance=-1mm]90:{ \scriptsize 1 } }] at (0.0, 1.0) {  }  
}
}
child {node [b] at (0.5, 1.0) {  }  
child {node [w,label={ [label distance=-1mm]90:{ \scriptsize 1 } }] at (0.0, 1.0) {  }  
}
}
;
}}

\newcommand{\forestPC}{
\tikz[planar forest ] {

\node [b] at (0.0, 0.0) {  } 
child {node [lc] at (-0.5, 1.0) { 1 }  
}
child {node [b] at (0.5, 1.0) {  }  
child {node [lc] at (-0.5, 1.0) { 3 }  
}
child {node [lc] at (0.5, 1.0) { 3 }  
}
}
;
\node [lc] at (1.25, 0.0) { 1 } 
;
\node [b] at (2.5, 0.0) {  } 
child {node [b] at (-0.5, 1.0) {  }  
child {node [lc] at (0.0, 1.0) { 2 }  
}
}
child {node [b] at (0.5, 1.0) {  }  
child {node [lc] at (0.0, 1.0) { 2 }  
}
}
;
}}

\newcommand{\forestQC}{
\tikz[planar forest ] {

\node [b] at (0.0, 0.0) {  } 
child {node [lc] at (-0.5, 1.0) { 1 }  
}
child {node [b] at (0.5, 1.0) {  }  
child {node [lc] at (-0.5, 1.0) { 3 }  
}
child {node [lc] at (0.5, 1.0) { 3 }  
}
}
;
\node [lc] at (1.25, 0.0) { 2 } 
;
\node [b] at (2.5, 0.0) {  } 
child {node [b] at (-0.5, 1.0) {  }  
child {node [lc] at (0.0, 1.0) { 2 }  
}
}
child {node [b] at (0.5, 1.0) {  }  
child {node [lc] at (0.0, 1.0) { 1 }  
}
}
;
}}

\newcommand{\forestRC}{
\tikz[planar forest ] {

\node [b] at (0.0, 0.0) {  } 
child {node [lc] at (-0.5, 1.0) { 1 }  
}
child {node [b] at (0.5, 1.0) {  }  
child {node [lc] at (-0.5, 1.0) { 3 }  
}
child {node [lc] at (0.5, 1.0) { 3 }  
}
}
;
\node [lc] at (1.25, 0.0) { 2 } 
;
\node [b] at (2.5, 0.0) {  } 
child {node [b] at (-0.5, 1.0) {  }  
child {node [lc] at (0.0, 1.0) { 1 }  
}
}
child {node [b] at (0.5, 1.0) {  }  
child {node [lc] at (0.0, 1.0) { 2 }  
}
}
;
}}

\newcommand{\forestSC}{
\tikz[planar forest ] {

\node [b] at (0.0, 0.0) {  } 
;
}}

\newcommand{\forestTC}{
\tikz[planar forest ] {

\node [w] at (0.0, 0.0) {  } 
;
}}

\newcommand{\forestUC}{
\tikz[planar forest ] {

\node [w] at (0.0, 0.0) {  } 
;
}}

\newcommand{\forestVC}{
\tikz[planar forest ] {

\node [b] at (0.0, 0.0) {  } 
;
}}

\newcommand{\forestWC}{
\tikz[planar forest ] {

\node [w] at (0.0, 0.0) {  } 
;
\node [w] at (1.0, 0.0) {  } 
;
}}

\newcommand{\forestXC}{
\tikz[planar forest ] {

\node [b] at (0.0, 0.0) {  } 
;
\node [w] at (1.0, 0.0) {  } 
;
}}

\newcommand{\forestYC}{
\tikz[planar forest ] {

\node [w] at (0.0, 0.0) {  } 
;
\node [b] at (1.0, 0.0) {  } 
;
}}

\newcommand{\forestAD}{
\tikz[planar forest ] {

\node [w] at (0.0, 0.0) {  } 
;
\node [w] at (1.0, 0.0) {  } 
;
\node [w] at (2.0, 0.0) {  } 
;
}}

\newcommand{\forestBD}{
\tikz[planar forest ] {

\node [b] at (0.0, 0.0) {  } 
;
\node [b] at (1.0, 0.0) {  } 
;
}}

\newcommand{\forestCD}{
\tikz[planar forest ] {

\node [b] at (0.0, 0.0) {  } 
;
\node [w] at (1.0, 0.0) {  } 
;
\node [w] at (2.0, 0.0) {  } 
;
}}

\newcommand{\forestDD}{
\tikz[planar forest ] {

\node [w] at (0.0, 0.0) {  } 
;
\node [b] at (1.0, 0.0) {  } 
;
\node [w] at (2.0, 0.0) {  } 
;
}}

\newcommand{\forestED}{
\tikz[planar forest ] {

\node [w] at (0.0, 0.0) {  } 
;
\node [w] at (1.0, 0.0) {  } 
;
\node [b] at (2.0, 0.0) {  } 
;
}}

\newcommand{\forestFD}{
\tikz[planar forest ] {

\node [w] at (0.0, 0.0) {  } 
;
\node [w] at (1.0, 0.0) {  } 
;
\node [w] at (2.0, 0.0) {  } 
;
\node [w] at (3.0, 0.0) {  } 
;
}}

\newcommand{\forestGD}{
\tikz[planar forest ] {

\node [b] at (0.0, 0.0) {  } 
;
}}

\newcommand{\forestHD}{
\tikz[planar forest ] {

\node [lc] at (0.0, 0.0) { 1 } 
;
\node [lc] at (1.0, 0.0) { 1 } 
;
}}

\newcommand{\forestID}{
\tikz[planar forest ] {

\node [b] at (0.0, 0.0) {  } 
;
\node [b] at (1.0, 0.0) {  } 
;
}}

\newcommand{\forestJD}{
\tikz[planar forest ] {

\node [b] at (0.0, 0.0) {  } 
;
\node [lc] at (1.0, 0.0) { 1 } 
;
\node [lc] at (2.0, 0.0) { 1 } 
;
}}

\newcommand{\forestKD}{
\tikz[planar forest ] {

\node [lc] at (0.0, 0.0) { 1 } 
;
\node [b] at (1.0, 0.0) {  } 
;
\node [lc] at (2.0, 0.0) { 1 } 
;
}}

\newcommand{\forestLD}{
\tikz[planar forest ] {

\node [lc] at (0.0, 0.0) { 1 } 
;
\node [lc] at (1.0, 0.0) { 1 } 
;
\node [b] at (2.0, 0.0) {  } 
;
}}

\newcommand{\forestMD}{
\tikz[planar forest ] {

\node [lc] at (0.0, 0.0) { 2 } 
;
\node [lc] at (1.0, 0.0) { 2 } 
;
\node [lc] at (2.0, 0.0) { 1 } 
;
\node [lc] at (3.0, 0.0) { 1 } 
;
}}

\newcommand{\forestND}{
\tikz[planar forest ] {

\node [lc] at (0.0, 0.0) { 2 } 
;
\node [lc] at (1.0, 0.0) { 1 } 
;
\node [lc] at (2.0, 0.0) { 2 } 
;
\node [lc] at (3.0, 0.0) { 1 } 
;
}}

\newcommand{\forestOD}{
\tikz[planar forest ] {

\node [lc] at (0.0, 0.0) { 1 } 
;
\node [lc] at (1.0, 0.0) { 2 } 
;
\node [lc] at (2.0, 0.0) { 2 } 
;
\node [lc] at (3.0, 0.0) { 1 } 
;
}}

\newcommand{\forestPD}{
\tikz[planar forest ] {

\node [b] at (0.0, 0.0) {  } 
;
}}

\newcommand{\forestQD}{
\tikz[planar forest ] {

\node [lc] at (0.0, 0.0) { 1 } 
;
\node [lc] at (1.0, 0.0) { 1 } 
;
}}

\newcommand{\forestRD}{
\tikz[planar forest ] {

\node [b] at (0.0, 0.0) {  } 
;
}}

\newcommand{\forestSD}{
\tikz[planar forest ] {

\node [lc] at (0.0, 0.0) { 1 } 
;
\node [lc] at (1.0, 0.0) { 1 } 
;
}}

\newcommand{\forestTD}{
\tikz[planar forest ] {

\node [b] at (0.0, 0.0) {  } 
;
}}

\newcommand{\forestUD}{
\tikz[planar forest ] {

\node [lc] at (0.0, 0.0) { 1 } 
;
\node [lc] at (1.0, 0.0) { 1 } 
;
}}

\newcommand{\forestVD}{
\tikz[planar forest ] {

\node [b] at (0.0, 0.0) {  } 
;
}}

\newcommand{\forestWD}{
\tikz[planar forest ] {

\node [lc] at (0.0, 0.0) { 1 } 
;
\node [lc] at (1.0, 0.0) { 1 } 
;
}}

\newcommand{\forestXD}{
\tikz[planar forest ] {

\node [b] at (0.0, 0.0) {  } 
child {node [b] at (0.0, 1.0) {  }  
}
;
}}

\newcommand{\forestYD}{
\tikz[planar forest ] {

\node [b] at (0.0, 0.0) {  } 
child {node [lc] at (-0.5, 1.0) { 1 }  
}
child {node [lc] at (0.5, 1.0) { 1 }  
}
;
}}

\newcommand{\forestAE}{
\tikz[planar forest ] {

\node [b] at (0.0, 0.0) {  } 
;
\node [b] at (1.0, 0.0) {  } 
;
}}

\newcommand{\forestBE}{
\tikz[planar forest ] {

\node [lc] at (0.0, 0.0) { 1 } 
;
\node [b] at (1.0, 0.0) {  } 
child {node [lc] at (0.0, 1.0) { 1 }  
}
;
}}

\newcommand{\forestCE}{
\tikz[planar forest ] {

\node [b] at (0.0, 0.0) {  } 
;
\node [lc] at (1.0, 0.0) { 1 } 
;
\node [lc] at (2.0, 0.0) { 1 } 
;
}}

\newcommand{\forestDE}{
\tikz[planar forest ] {

\node [lc] at (0.0, 0.0) { 1 } 
;
\node [lc] at (1.0, 0.0) { 1 } 
;
\node [b] at (2.0, 0.0) {  } 
;
}}

\newcommand{\forestEE}{
\tikz[planar forest ] {

\node [lc] at (0.0, 0.0) { 2 } 
;
\node [lc] at (1.0, 0.0) { 2 } 
;
\node [lc] at (2.0, 0.0) { 1 } 
;
\node [lc] at (3.0, 0.0) { 1 } 
;
}}

\newcommand{\forestFE}{
\tikz[planar forest ] {

\node [b] at (0.0, 0.0) {  } 
;
\node [b] at (1.0, 0.0) {  } 
child {node [lc] at (-0.5, 1.0) { 1 }  
}
child {node [b] at (0.5, 1.0) {  }  
child {node [lc] at (0.0, 1.0) { 1 }  
}
}
;
}}

\newcommand{\forestGE}{
\tikz[planar forest ] {

\node [b,label={ [label distance=-1mm]180:{ \scriptsize 2 } }] at (0.0, 0.0) {  } 
;
\node [b,label={ [label distance=-1mm]0:{ \scriptsize 8 } }] at (1.0, 0.0) {  } 
child {node [lc,label={ [label distance=-1mm]180:{ \scriptsize 3 } }] at (-0.5, 1.0) { 1 }  
}
child {node [b,label={ [label distance=-1mm]0:{ \scriptsize 7 } }] at (0.5, 1.0) {  }  
child {node [lc,label={ [label distance=-1mm]0:{ \scriptsize 2 } }] at (0.0, 1.0) { 1 }  
}
}
;
}}

\newcommand{\forestHE}{
\tikz[planar forest ] {

\node [w] at (0.0, 0.0) {  } 
;
}}

\newcommand{\forestIE}{
\tikz[planar forest ] {

\node [w] at (0.0, 0.0) {  } 
;
}}

\newcommand{\forestJE}{
\tikz[planar forest ] {

\node [w] at (0.0, 0.0) {  } 
;
}}

\newcommand{\forestKE}{
\tikz[planar forest ] {

\node [w] at (0.0, 0.0) {  } 
;
}}

\newcommand{\forestLE}{
\tikz[planar forest ] {

\node [w] at (0.0, 0.0) {  } 
;
}}

\newcommand{\forestME}{
\tikz[planar forest ] {

\node [w] at (0.0, 0.0) {  } 
;
}}

\newcommand{\forestNE}{
\tikz[planar forest ] {

\node [w] at (0.0, 0.0) {  } 
;
}}

\newcommand{\forestOE}{
\tikz[planar forest ] {

\node [w] at (0.0, 0.0) {  } 
;
}}

\newcommand{\forestPE}{
\tikz[planar forest ] {

\node [b,label={ [label distance=-1mm]270:{ \scriptsize 3 } }] at (0.0, 0.0) {  } 
child {node [b,label={ [label distance=-1mm]180:{ \scriptsize 8 } }] at (-1.0, 1.0) {  }  
child {node [lc,label={ [label distance=-1mm]90:{ \scriptsize 9 } }] at (0.0, 1.0) { 2 }  
}
}
child {node [lc,label={ [label distance=-1mm]90:{ \scriptsize 7 } }] at (0.0, 1.0) { 1 }  
}
child {node [lc,label={ [label distance=-1mm]90:{ \scriptsize 6 } }] at (1.0, 1.0) { 1 }  
}
;
\node [lc,label={ [label distance=-1mm]270:{ \scriptsize 2 } }] at (1.25, 0.0) { 2 } 
;
\node [b,label={ [label distance=-1mm]270:{ \scriptsize 1 } }] at (2.5, 0.0) {  } 
child {node [b,label={ [label distance=-1mm]90:{ \scriptsize 5 } }] at (-0.5, 1.0) {  }  
}
child {node [b,label={ [label distance=-1mm]90:{ \scriptsize 4 } }] at (0.5, 1.0) {  }  
}
;
}}

\newcommand{\forestQE}{
\tikz[planar forest ] {

\node [b,label={ [label distance=-1mm]90:{ \scriptsize d } }] at (0.0, 0.0) {  } 
;
}}

\newcommand{\forestRE}{
\tikz[planar forest ] {

\node [b,label={ [label distance=-1mm]270:{ \scriptsize a } }] at (0.0, 0.0) {  } 
;
\node [b,label={ [label distance=-1mm]270:{ \scriptsize b } }] at (1.0, 0.0) {  } 
child {node [b] at (0.0, 1.0) {  }  
}
;
}}

\newcommand{\forestSE}{
\tikz[planar forest ] {

\node [b,label={ [label distance=-1mm]270:{ \scriptsize c } }] at (0.0, 0.0) {  } 
child {node [b] at (-0.5, 1.0) {  }  
}
child {node [b] at (0.5, 1.0) {  }  
}
;
\node [b,label={ [label distance=-1mm]270:{ \scriptsize d } }] at (1.5, 0.0) {  } 
child {node [b] at (0.0, 1.0) {  }  
child {node [b] at (0.0, 1.0) {  }  
}
}
;
}}

\newcommand{\forestTE}{
\tikz[planar forest ] {

\node [b,label={ [label distance=-1mm]270:{ \scriptsize c } }] at (0.0, 0.0) {  } 
child {node [b] at (-0.5, 1.0) {  }  
}
child {node [b] at (0.5, 1.0) {  }  
}
;
\node [b,label={ [label distance=-1mm]270:{ \scriptsize a } }] at (1.0, 0.0) {  } 
;
\node [b,label={ [label distance=-1mm]270:{ \scriptsize b } }] at (2.0, 0.0) {  } 
child {node [b] at (0.0, 1.0) {  }  
}
;
\node [b,label={ [label distance=-1mm]270:{ \scriptsize d } }] at (3.0, 0.0) {  } 
child {node [b] at (0.0, 1.0) {  }  
child {node [b] at (0.0, 1.0) {  }  
}
}
;
}}

\newcommand{\forestUE}{
\tikz[planar forest ] {

\node [b] at (0.0, 0.0) {  } 
;
\node [b] at (1.0, 0.0) {  } 
child {node [b] at (0.0, 1.0) {  }  
}
;
}}

\newcommand{\forestVE}{
\tikz[planar forest ] {

\node [b] at (0.0, 0.0) {  } 
child {node [b] at (-0.5, 1.0) {  }  
}
child {node [b] at (0.5, 1.0) {  }  
}
;
\node [b] at (1.5, 0.0) {  } 
child {node [b] at (0.0, 1.0) {  }  
child {node [b] at (0.0, 1.0) {  }  
}
}
;
}}

\newcommand{\forestWE}{
\tikz[planar forest ] {

\node [b] at (0.0, 0.0) {  } 
child {node [b] at (-0.5, 1.0) {  }  
}
child {node [b] at (0.5, 1.0) {  }  
}
;
\node [b] at (1.0, 0.0) {  } 
;
\node [b] at (2.0, 0.0) {  } 
child {node [b] at (0.0, 1.0) {  }  
}
;
\node [b] at (3.0, 0.0) {  } 
child {node [b] at (0.0, 1.0) {  }  
child {node [b] at (0.0, 1.0) {  }  
}
}
;
}}

\newcommand{\forestXE}{
\tikz[planar forest ] {

\node [b] at (0.0, 0.0) {  } 
child {node [b] at (0.0, 1.0) {  }  
}
;
}}

\newcommand{\forestYE}{
\tikz[planar forest ] {

\node [b] at (0.0, 0.0) {  } 
child {node [b] at (-0.5, 1.0) {  }  
}
child {node [b] at (0.5, 1.0) {  }  
}
;
}}

\newcommand{\forestAF}{
\tikz[planar forest ] {

\node [b] at (0.0, 0.0) {  } 
;
\node [b] at (1.0, 0.0) {  } 
child {node [b] at (0.0, 1.0) {  }  
child {node [b] at (0.0, 1.0) {  }  
}
}
;
}}

\newcommand{\forestBF}{
\tikz[planar forest ] {

\node [b] at (0.0, 0.0) {  } 
child {node [b] at (-0.5, 1.0) {  }  
}
child {node [b] at (0.5, 1.0) {  }  
}
;
\node [b] at (1.0, 0.0) {  } 
;
}}

\newcommand{\forestCF}{
\tikz[planar forest ] {

\node [b] at (0.0, 0.0) {  } 
child {node [b] at (0.0, 1.0) {  }  
}
;
}}

\newcommand{\forestDF}{
\tikz[planar forest ] {

\node [b] at (0.0, 0.0) {  } 
child {node [b] at (0.0, 1.0) {  }  
child {node [b] at (0.0, 1.0) {  }  
}
}
;
}}

\newcommand{\forestEF}{
\tikz[planar forest ] {

\node [b] at (0.0, 0.0) {  } 
;
}}

\newcommand{\forestFF}{
\tikz[planar forest ] {

\node [b] at (0.0, 0.0) {  } 
child {node [b] at (-0.5, 1.0) {  }  
}
child {node [b] at (0.5, 1.0) {  }  
}
;
}}

\newcommand{\forestGF}{
\tikz[planar forest ] {

\node [b] at (0.0, 0.0) {  } 
child {node [b] at (0.0, 1.0) {  }  
}
;
\node [b] at (1.0, 0.0) {  } 
child {node [b] at (0.0, 1.0) {  }  
child {node [b] at (0.0, 1.0) {  }  
}
}
;
}}

\newcommand{\forestHF}{
\tikz[planar forest ] {

\node [b] at (0.0, 0.0) {  } 
child {node [b] at (-0.5, 1.0) {  }  
}
child {node [b] at (0.5, 1.0) {  }  
}
;
\node [b] at (1.5, 0.0) {  } 
child {node [b] at (0.0, 1.0) {  }  
}
;
}}

\newcommand{\forestIF}{
\tikz[planar forest ] {

\node [b] at (0.0, 0.0) {  } 
;
}}

\newcommand{\forestJF}{
\tikz[planar forest ] {

\node [b] at (0.0, 0.0) {  } 
child {node [b] at (0.0, 1.0) {  }  
child {node [b] at (0.0, 1.0) {  }  
}
}
;
}}

\newcommand{\forestKF}{
\tikz[planar forest ] {

\node [b] at (0.0, 0.0) {  } 
;
\node [b] at (1.0, 0.0) {  } 
child {node [b] at (0.0, 1.0) {  }  
}
;
}}

\newcommand{\forestLF}{
\tikz[planar forest ] {

\node [b] at (0.0, 0.0) {  } 
child {node [b] at (-0.5, 1.0) {  }  
}
child {node [b] at (0.5, 1.0) {  }  
}
;
\node [b] at (1.5, 0.0) {  } 
child {node [b] at (0.0, 1.0) {  }  
child {node [b] at (0.0, 1.0) {  }  
}
}
;
}}

\newcommand{\forestMF}{
\tikz[planar forest ] {

\node [b] at (0.0, 0.0) {  } 
;
}}

\newcommand{\forestNF}{
\tikz[planar forest ] {

\node [lc] at (0.0, 0.0) { d } 
;
}}

\newcommand{\forestOF}{
\tikz[planar forest ] {

\node [b] at (0.0, 0.0) {  } 
child {node [lc] at (-0.5, 1.0) { 1 }  
}
child {node [lc] at (0.5, 1.0) { 1 }  
}
;
}}

\newcommand{\forestPF}{
\tikz[planar forest ] {

\node [b] at (0.0, 0.0) {  } 
child {node [lc] at (-0.5, 1.0) { i }  
}
child {node [lc] at (0.5, 1.0) { i }  
}
;
}}

\newcommand{\forestQF}{
\tikz[planar forest ] {

\node [b] at (0.0, 0.0) {  } 
child {node [b] at (-0.5, 1.0) {  }  
child {node [lc] at (-0.5, 1.0) { 1 }  
}
child {node [b] at (0.5, 1.0) {  }  
child {node [lc] at (-0.5, 1.0) { 2 }  
}
child {node [lc] at (0.5, 1.0) { 2 }  
}
}
}
child {node [lc] at (0.5, 1.0) { 1 }  
}
;
}}

\newcommand{\forestRF}{
\tikz[planar forest ] {

\node [b] at (0.0, 0.0) {  } 
child {node [b] at (-0.5, 1.0) {  }  
child {node [lc] at (-0.5, 1.0) { i }  
}
child {node [b] at (0.5, 1.0) {  }  
child {node [lc] at (-0.5, 1.0) { j }  
}
child {node [lc] at (0.5, 1.0) { j }  
}
}
}
child {node [lc] at (0.5, 1.0) { i }  
}
;
}}

\newcommand{\forestSF}{
\tikz[planar forest ] {

\node [b] at (0.0, 0.0) {  } 
child {node [lc] at (0.0, 1.0) { 1 }  
}
;
\node [lc] at (1.0, 0.0) { 1 } 
;
}}

\newcommand{\forestTF}{
\tikz[planar forest ] {

\node [b] at (0.0, 0.0) {  } 
child {node [b] at (0.0, 1.0) {  }  
}
;
}}

\newcommand{\forestUF}{
\tikz[planar forest ] {

\node [b] at (0.0, 0.0) {  } 
child {node [lc] at (0.0, 1.0) { 1 }  
}
;
\node [lc] at (1.0, 0.0) { 1 } 
;
\node [b] at (2.0, 0.0) {  } 
child {node [b] at (0.0, 1.0) {  }  
}
;
}}

\newcommand{\forestVF}{
\tikz[planar forest ] {

\node [b] at (0.0, 0.0) {  } 
child {node [lc] at (0.0, 1.0) { 1 }  
}
;
\node [b] at (1.0, 0.0) {  } 
child {node [b] at (0.0, 1.0) {  }  
}
;
\node [lc] at (2.0, 0.0) { 1 } 
;
}}

\newcommand{\forestWF}{
\tikz[planar forest ] {

\node [b] at (0.0, 0.0) {  } 
child {node [b] at (0.0, 1.0) {  }  
}
;
\node [b] at (1.0, 0.0) {  } 
child {node [lc] at (0.0, 1.0) { 1 }  
}
;
\node [lc] at (2.0, 0.0) { 1 } 
;
}}

\newcommand{\forestXF}{
\tikz[planar forest ] {

\node [lc] at (0.0, 0.0) { 1 } 
;
\node [lc] at (1.0, 0.0) { 1 } 
;
}}

\newcommand{\forestYF}{
\tikz[planar forest ] {

\node [b] at (0.0, 0.0) {  } 
child {node [b] at (0.0, 1.0) {  }  
}
;
\node [b] at (1.5, 0.0) {  } 
child {node [lc] at (-0.5, 1.0) { 1 }  
}
child {node [lc] at (0.5, 1.0) { 1 }  
}
;
}}

\newcommand{\forestAG}{
\tikz[planar forest ] {

\node [lc] at (0.0, 0.0) { 1 } 
;
\node [lc] at (1.0, 0.0) { 1 } 
;
\node [b] at (2.0, 0.0) {  } 
child {node [b] at (0.0, 1.0) {  }  
}
;
\node [b] at (3.5, 0.0) {  } 
child {node [lc] at (-0.5, 1.0) { 2 }  
}
child {node [lc] at (0.5, 1.0) { 2 }  
}
;
}}

\newcommand{\forestBG}{
\tikz[planar forest ] {

\node [lc] at (0.0, 0.0) { 1 } 
;
\node [b] at (1.0, 0.0) {  } 
child {node [b] at (0.0, 1.0) {  }  
}
;
\node [lc] at (2.0, 0.0) { 1 } 
;
\node [b] at (3.0, 0.0) {  } 
child {node [lc] at (-0.5, 1.0) { 2 }  
}
child {node [lc] at (0.5, 1.0) { 2 }  
}
;
}}

\newcommand{\forestCG}{
\tikz[planar forest ] {

\node [lc] at (0.0, 0.0) { 1 } 
;
\node [b] at (1.0, 0.0) {  } 
child {node [b] at (0.0, 1.0) {  }  
}
;
\node [b] at (2.5, 0.0) {  } 
child {node [lc] at (-0.5, 1.0) { 2 }  
}
child {node [lc] at (0.5, 1.0) { 2 }  
}
;
\node [lc] at (3.5, 0.0) { 1 } 
;
}}

\newcommand{\forestDG}{
\tikz[planar forest ] {

\node [b] at (0.0, 0.0) {  } 
child {node [b] at (0.0, 1.0) {  }  
}
;
\node [lc] at (1.0, 0.0) { 1 } 
;
\node [lc] at (2.0, 0.0) { 1 } 
;
\node [b] at (3.0, 0.0) {  } 
child {node [lc] at (-0.5, 1.0) { 2 }  
}
child {node [lc] at (0.5, 1.0) { 2 }  
}
;
}}

\newcommand{\forestEG}{
\tikz[planar forest ] {

\node [b] at (0.0, 0.0) {  } 
child {node [b] at (0.0, 1.0) {  }  
}
;
\node [lc] at (1.0, 0.0) { 1 } 
;
\node [b] at (2.0, 0.0) {  } 
child {node [lc] at (-0.5, 1.0) { 2 }  
}
child {node [lc] at (0.5, 1.0) { 2 }  
}
;
\node [lc] at (3.0, 0.0) { 1 } 
;
}}

\newcommand{\forestFG}{
\tikz[planar forest ] {

\node [b] at (0.0, 0.0) {  } 
child {node [b] at (0.0, 1.0) {  }  
}
;
\node [b] at (1.5, 0.0) {  } 
child {node [lc] at (-0.5, 1.0) { 2 }  
}
child {node [lc] at (0.5, 1.0) { 2 }  
}
;
\node [lc] at (2.5, 0.0) { 1 } 
;
\node [lc] at (3.5, 0.0) { 1 } 
;
}}

\newcommand{\forestGG}{
\tikz[planar forest ] {

\node [b] at (0.0, 0.0) {  } 
;
}}

\newcommand{\forestHG}{
\tikz[planar forest ] {

\node [lc] at (0.0, 0.0) { 1 } 
;
\node [lc] at (1.0, 0.0) { 1 } 
;
}}

\newcommand{\forestIG}{
\tikz[planar forest ] {

\node [b] at (0.0, 0.0) {  } 
child {node [b] at (0.0, 1.0) {  }  
}
;
}}

\newcommand{\forestJG}{
\tikz[planar forest ] {

\node [b] at (0.0, 0.0) {  } 
child {node [lc] at (-0.5, 1.0) { 1 }  
}
child {node [lc] at (0.5, 1.0) { 1 }  
}
;
}}

\newcommand{\forestKG}{
\tikz[planar forest ] {

\node [b] at (0.0, 0.0) {  } 
child {node [lc] at (0.0, 1.0) { 1 }  
}
;
\node [lc] at (1.0, 0.0) { 1 } 
;
}}

\newcommand{\forestLG}{
\tikz[planar forest ] {

\node [lc] at (0.0, 0.0) { 1 } 
;
\node [b] at (1.0, 0.0) {  } 
child {node [lc] at (0.0, 1.0) { 1 }  
}
;
}}

\newcommand{\forestMG}{
\tikz[planar forest ] {

\node [lc] at (0.0, 0.0) { 1 } 
;
\node [b] at (1.0, 0.0) {  } 
;
\node [lc] at (2.0, 0.0) { 1 } 
;
}}

\newcommand{\forestNG}{
\tikz[planar forest ] {

\node [lc] at (0.0, 0.0) { 2 } 
;
\node [lc] at (1.0, 0.0) { 1 } 
;
\node [lc] at (2.0, 0.0) { 2 } 
;
\node [lc] at (3.0, 0.0) { 1 } 
;
}}

\newcommand{\forestOG}{
\tikz[planar forest ] {

\node [lc] at (0.0, 0.0) { 1 } 
;
\node [lc] at (1.0, 0.0) { 2 } 
;
\node [lc] at (2.0, 0.0) { 2 } 
;
\node [lc] at (3.0, 0.0) { 1 } 
;
}}

\newcommand{\forestPG}{
\tikz[planar forest ] {

\node [b] at (0.0, 0.0) {  } 
;
}}

\newcommand{\forestQG}{
\tikz[planar forest ] {

\node [lc] at (0.0, 0.0) { 1 } 
;
\node [lc] at (1.0, 0.0) { 1 } 
;
}}

\newcommand{\forestRG}{
\tikz[planar forest ] {

\node [lc,label={ [label distance=-1mm]270:{ \scriptsize a } }] at (0.0, 0.0) { 1 } 
;
\node [b,label={ [label distance=-1mm]270:{ \scriptsize b } }] at (1.0, 0.0) {  } 
child {node [lc] at (0.0, 1.0) { 1 }  
}
;
}}

\newcommand{\forestSG}{
\tikz[planar forest ] {

\node [b,label={ [label distance=-1mm]270:{ \scriptsize c } }] at (0.0, 0.0) {  } 
child {node [lc] at (-0.5, 1.0) { 2 }  
}
child {node [lc] at (0.5, 1.0) { 2 }  
}
;
}}

\newcommand{\forestTG}{
\tikz[planar forest ] {

\node [b,label={ [label distance=-1mm]270:{ \scriptsize c } }] at (0.0, 0.0) {  } 
child {node [lc] at (-1.0, 1.0) { 1 }  
}
child {node [lc] at (0.0, 1.0) { 2 }  
}
child {node [lc] at (1.0, 1.0) { 2 }  
}
;
\node [b,label={ [label distance=-1mm]270:{ \scriptsize b } }] at (2.0, 0.0) {  } 
child {node [lc] at (0.0, 1.0) { 1 }  
}
;
}}

\newcommand{\forestUG}{
\tikz[planar forest ] {

\node [lc,label={ [label distance=-1mm]270:{ \scriptsize a } }] at (0.0, 0.0) { 1 } 
;
\node [b,label={ [label distance=-1mm]270:{ \scriptsize c } }] at (1.0, 0.0) {  } 
child {node [b] at (-1.0, 1.0) {  }  
child {node [lc] at (0.0, 1.0) { 1 }  
}
}
child {node [lc] at (0.0, 1.0) { 2 }  
}
child {node [lc] at (1.0, 1.0) { 2 }  
}
;
}}

\newcommand{\forestVG}{
\tikz[planar forest ] {

\node [b,label={ [label distance=-1mm]270:{ \scriptsize c } }] at (0.0, 0.0) {  } 
child {node [lc] at (-1.5, 1.0) { 1 }  
}
child {node [b] at (-0.5, 1.0) {  }  
child {node [lc] at (0.0, 1.0) { 1 }  
}
}
child {node [lc] at (0.5, 1.0) { 2 }  
}
child {node [lc] at (1.5, 1.0) { 2 }  
}
;
}}

\newcommand{\forestWG}{
\tikz[planar forest ] {

\node [lc] at (0.0, 0.0) { 1 } 
;
\node [lc] at (1.0, 0.0) { 1 } 
;
}}

\newcommand{\forestXG}{
\tikz[planar forest ] {

\node [lc] at (0.0, 0.0) { 1 } 
;
\node [lc] at (1.0, 0.0) { 1 } 
;
}}

\newcommand{\forestYG}{
\tikz[planar forest ] {

\node [lc] at (0.0, 0.0) { 1 } 
;
\node [lc] at (1.0, 0.0) { 1 } 
;
}}

\newcommand{\forestAH}{
\tikz[planar forest ] {

\node [b] at (0.0, 0.0) {  } 
child {node [lc] at (-0.5, 1.0) { 1 }  
}
child {node [lc] at (0.5, 1.0) { 1 }  
}
;
}}

\newcommand{\forestBH}{
\tikz[planar forest ] {

\node [b] at (0.0, 0.0) {  } 
child {node [lc] at (-0.5, 1.0) { 1 }  
}
child {node [lc] at (0.5, 1.0) { 1 }  
}
;
}}

\newcommand{\forestCH}{
\tikz[planar forest ] {

\node [lc] at (0.0, 0.0) { 1 } 
;
\node [lc] at (1.0, 0.0) { 1 } 
;
}}

\newcommand{\forestDH}{
\tikz[planar forest ] {

\node [b] at (0.0, 0.0) {  } 
;
}}

\newcommand{\forestEH}{
\tikz[planar forest ] {

\node [b] at (0.0, 0.0) {  } 
child {node [lc] at (-0.5, 1.0) { 1 }  
}
child {node [lc] at (0.5, 1.0) { 1 }  
}
;
}}

\newcommand{\forestFH}{
\tikz[planar forest ] {

\node [b] at (0.0, 0.0) {  } 
child {node [lc] at (0.0, 1.0) { 1 }  
}
;
\node [b] at (1.0, 0.0) {  } 
child {node [b] at (0.0, 1.0) {  }  
child {node [lc] at (0.0, 1.0) { 1 }  
}
}
;
}}

\newcommand{\forestGH}{
\tikz[planar forest ] {

\node [b] at (0.0, 0.0) {  } 
child {node [lc] at (0.0, 1.0) { 1 }  
}
;
\node [b] at (1.0, 0.0) {  } 
child {node [b] at (0.0, 1.0) {  }  
child {node [lc] at (0.0, 1.0) { 1 }  
}
}
;
}}

\newcommand{\forestHH}{
\tikz[planar forest ] {

\node [lc] at (0.0, 0.0) { 1 } 
;
\node [lc] at (1.0, 0.0) { 1 } 
;
}}

\newcommand{\forestIH}{
\tikz[planar forest ] {

\node [b] at (0.0, 0.0) {  } 
;
\node [b] at (1.0, 0.0) {  } 
child {node [b] at (0.0, 1.0) {  }  
}
;
}}

\newcommand{\forestJH}{
\tikz[planar forest ] {

\node [lc] at (0.0, 0.0) { 1 } 
;
\node [b] at (1.0, 0.0) {  } 
child {node [lc] at (0.0, 1.0) { 1 }  
}
;
}}

\newcommand{\forestKH}{
\tikz[planar forest ] {

\node [b] at (0.0, 0.0) {  } 
child {node [lc] at (0.0, 1.0) { 1 }  
}
;
\node [lc] at (1.0, 0.0) { 1 } 
;
}}

\newcommand{\forestLH}{
\tikz[planar forest ] {

\node [b] at (0.0, 0.0) {  } 
;
\node [b] at (1.0, 0.0) {  } 
;
}}

\newcommand{\forestMH}{
\tikz[planar forest ] {

\node [lc] at (0.0, 0.0) { 1 } 
;
\node [b] at (1.0, 0.0) {  } 
child {node [lc] at (0.0, 1.0) { 1 }  
}
;
}}

\newcommand{\forestNH}{
\tikz[planar forest ] {

\node [b] at (0.0, 0.0) {  } 
child {node [lc] at (0.0, 1.0) { 1 }  
}
;
\node [lc] at (1.0, 0.0) { 1 } 
;
}}

\newcommand{\forestOH}{
\tikz[planar forest ] {

\node [b] at (0.0, 0.0) {  } 
child {node [b] at (0.0, 1.0) {  }  
}
;
}}

\newcommand{\forestPH}{
\tikz[planar forest ] {

\node [lc] at (0.0, 0.0) { 1 } 
;
\node [b] at (1.0, 0.0) {  } 
child {node [b] at (0.0, 1.0) {  }  
child {node [lc] at (0.0, 1.0) { 1 }  
}
}
;
}}

\newcommand{\forestQH}{
\tikz[planar forest ] {

\node [b] at (0.0, 0.0) {  } 
child {node [b] at (0.0, 1.0) {  }  
child {node [lc] at (0.0, 1.0) { 1 }  
}
}
;
\node [lc] at (1.0, 0.0) { 1 } 
;
}}

\newcommand{\forestRH}{
\tikz[planar forest ] {

\node [b] at (0.0, 0.0) {  } 
;
}}

\newcommand{\forestSH}{
\tikz[planar forest ] {

\node [b] at (0.0, 0.0) {  } 
child {node [lc] at (0.0, 1.0) { 1 }  
}
;
\node [b] at (1.0, 0.0) {  } 
child {node [lc] at (0.0, 1.0) { 1 }  
}
;
}}

\newcommand{\forestTH}{
\tikz[planar forest ] {

\node [b] at (0.0, 0.0) {  } 
;
}}

\newcommand{\forestUH}{
\tikz[planar forest ] {

\node [b] at (0.0, 0.0) {  } 
child {node [lc] at (0.0, 1.0) { 1 }  
}
;
\node [b] at (1.0, 0.0) {  } 
child {node [b] at (0.0, 1.0) {  }  
child {node [lc] at (0.0, 1.0) { 1 }  
}
}
;
}}

\newcommand{\forestVH}{
\tikz[planar forest ] {

\node [b] at (0.0, 0.0) {  } 
;
}}

\newcommand{\forestWH}{
\tikz[planar forest ] {

\node [lc] at (0.0, 0.0) { d } 
;
}}

\newcommand{\forestXH}{
\tikz[planar forest ] {

\node [b] at (0.0, 0.0) {  } 
;
}}

\newcommand{\forestYH}{
\tikz[planar forest ] {

\node [lc] at (0.0, 0.0) { d } 
;
}}

\newcommand{\forestAI}{
\tikz[planar forest ] {

\node [b] at (0.0, 0.0) {  } 
child {node [b] at (0.0, 1.0) {  }  
child {node [b] at (0.0, 1.0) {  }  
}
}
;
}}

\newcommand{\forestBI}{
\tikz[planar forest ] {

\node [b] at (0.0, 0.0) {  } 
child {node [b] at (-0.5, 1.0) {  }  
}
child {node [b] at (0.5, 1.0) {  }  
}
;
}}

\newcommand{\forestCI}{
\tikz[planar forest ] {

\node [b] at (0.0, 0.0) {  } 
child {node [b] at (0.0, 1.0) {  }  
child {node [lc] at (-0.5, 1.0) { 1 }  
}
child {node [lc] at (0.5, 1.0) { 1 }  
}
}
;
}}

\newcommand{\forestDI}{
\tikz[planar forest ] {

\node [b] at (0.0, 0.0) {  } 
child {node [b] at (-0.5, 1.0) {  }  
child {node [lc] at (0.0, 1.0) { 1 }  
}
}
child {node [lc] at (0.5, 1.0) { 1 }  
}
;
}}

\newcommand{\forestEI}{
\tikz[planar forest ] {

\node [b] at (0.0, 0.0) {  } 
child {node [lc] at (-0.5, 1.0) { 1 }  
}
child {node [b] at (0.5, 1.0) {  }  
child {node [lc] at (0.0, 1.0) { 1 }  
}
}
;
}}

\newcommand{\forestFI}{
\tikz[planar forest ] {

\node [b] at (0.0, 0.0) {  } 
child {node [b] at (-1.0, 1.0) {  }  
}
child {node [lc] at (0.0, 1.0) { 1 }  
}
child {node [lc] at (1.0, 1.0) { 1 }  
}
;
}}

\newcommand{\forestGI}{
\tikz[planar forest ] {

\node [b] at (0.0, 0.0) {  } 
child {node [lc] at (-1.0, 1.0) { 1 }  
}
child {node [b] at (0.0, 1.0) {  }  
}
child {node [lc] at (1.0, 1.0) { 1 }  
}
;
}}

\newcommand{\forestHI}{
\tikz[planar forest ] {

\node [b] at (0.0, 0.0) {  } 
child {node [lc] at (-1.0, 1.0) { 1 }  
}
child {node [lc] at (0.0, 1.0) { 1 }  
}
child {node [b] at (1.0, 1.0) {  }  
}
;
}}

\newcommand{\forestII}{
\tikz[planar forest ] {

\node [b] at (0.0, 0.0) {  } 
child {node [lc] at (-1.5, 1.0) { 2 }  
}
child {node [lc] at (-0.5, 1.0) { 2 }  
}
child {node [lc] at (0.5, 1.0) { 1 }  
}
child {node [lc] at (1.5, 1.0) { 1 }  
}
;
}}

\newcommand{\forestJI}{
\tikz[planar forest ] {

\node [b] at (0.0, 0.0) {  } 
child {node [lc] at (-1.5, 1.0) { 2 }  
}
child {node [lc] at (-0.5, 1.0) { 1 }  
}
child {node [lc] at (0.5, 1.0) { 2 }  
}
child {node [lc] at (1.5, 1.0) { 1 }  
}
;
}}

\newcommand{\forestKI}{
\tikz[planar forest ] {

\node [b] at (0.0, 0.0) {  } 
child {node [lc] at (-1.5, 1.0) { 1 }  
}
child {node [lc] at (-0.5, 1.0) { 2 }  
}
child {node [lc] at (0.5, 1.0) { 2 }  
}
child {node [lc] at (1.5, 1.0) { 1 }  
}
;
}}

\newcommand{\forestLI}{
\tikz[planar forest ] {

\node [b] at (0.0, 0.0) {  } 
child {node [b] at (0.0, 1.0) {  }  
}
;
\node [b] at (1.0, 0.0) {  } 
;
}}

\newcommand{\forestMI}{
\tikz[planar forest ] {

\node [b] at (0.0, 0.0) {  } 
;
\node [b] at (1.0, 0.0) {  } 
child {node [b] at (0.0, 1.0) {  }  
}
;
}}

\newcommand{\forestNI}{
\tikz[planar forest ] {

\node [b] at (0.0, 0.0) {  } 
child {node [lc] at (-0.5, 1.0) { 1 }  
}
child {node [lc] at (0.5, 1.0) { 1 }  
}
;
\node [b] at (1.0, 0.0) {  } 
;
}}

\newcommand{\forestOI}{
\tikz[planar forest ] {

\node [b] at (0.0, 0.0) {  } 
;
\node [b] at (1.0, 0.0) {  } 
child {node [lc] at (-0.5, 1.0) { 1 }  
}
child {node [lc] at (0.5, 1.0) { 1 }  
}
;
}}

\newcommand{\forestPI}{
\tikz[planar forest ] {

\node [b] at (0.0, 0.0) {  } 
child {node [lc] at (0.0, 1.0) { 1 }  
}
;
\node [b] at (1.0, 0.0) {  } 
child {node [lc] at (0.0, 1.0) { 1 }  
}
;
}}

\newcommand{\forestQI}{
\tikz[planar forest ] {

\node [b] at (0.0, 0.0) {  } 
child {node [b] at (0.0, 1.0) {  }  
child {node [lc] at (0.0, 1.0) { 1 }  
}
}
;
\node [lc] at (1.0, 0.0) { 1 } 
;
}}

\newcommand{\forestRI}{
\tikz[planar forest ] {

\node [lc] at (0.0, 0.0) { 1 } 
;
\node [b] at (1.0, 0.0) {  } 
child {node [b] at (0.0, 1.0) {  }  
child {node [lc] at (0.0, 1.0) { 1 }  
}
}
;
}}

\newcommand{\forestSI}{
\tikz[planar forest ] {

\node [b] at (0.0, 0.0) {  } 
child {node [b] at (-0.5, 1.0) {  }  
}
child {node [lc] at (0.5, 1.0) { 1 }  
}
;
\node [lc] at (1.0, 0.0) { 1 } 
;
}}

\newcommand{\forestTI}{
\tikz[planar forest ] {

\node [b] at (0.0, 0.0) {  } 
child {node [lc] at (-0.5, 1.0) { 1 }  
}
child {node [b] at (0.5, 1.0) {  }  
}
;
\node [lc] at (1.0, 0.0) { 1 } 
;
}}

\newcommand{\forestUI}{
\tikz[planar forest ] {

\node [lc] at (0.0, 0.0) { 1 } 
;
\node [b] at (1.0, 0.0) {  } 
child {node [b] at (-0.5, 1.0) {  }  
}
child {node [lc] at (0.5, 1.0) { 1 }  
}
;
}}

\newcommand{\forestVI}{
\tikz[planar forest ] {

\node [lc] at (0.0, 0.0) { 1 } 
;
\node [b] at (1.0, 0.0) {  } 
child {node [lc] at (-0.5, 1.0) { 1 }  
}
child {node [b] at (0.5, 1.0) {  }  
}
;
}}

\newcommand{\forestWI}{
\tikz[planar forest ] {

\node [b] at (0.0, 0.0) {  } 
child {node [lc] at (-1.0, 1.0) { 2 }  
}
child {node [lc] at (0.0, 1.0) { 2 }  
}
child {node [lc] at (1.0, 1.0) { 1 }  
}
;
\node [lc] at (1.0, 0.0) { 1 } 
;
}}

\newcommand{\forestXI}{
\tikz[planar forest ] {

\node [b] at (0.0, 0.0) {  } 
child {node [lc] at (-1.0, 1.0) { 2 }  
}
child {node [lc] at (0.0, 1.0) { 1 }  
}
child {node [lc] at (1.0, 1.0) { 2 }  
}
;
\node [lc] at (1.0, 0.0) { 1 } 
;
}}

\newcommand{\forestYI}{
\tikz[planar forest ] {

\node [b] at (0.0, 0.0) {  } 
child {node [lc] at (-1.0, 1.0) { 1 }  
}
child {node [lc] at (0.0, 1.0) { 2 }  
}
child {node [lc] at (1.0, 1.0) { 2 }  
}
;
\node [lc] at (1.0, 0.0) { 1 } 
;
}}

\newcommand{\forestAJ}{
\tikz[planar forest ] {

\node [lc] at (0.0, 0.0) { 1 } 
;
\node [b] at (1.0, 0.0) {  } 
child {node [lc] at (-1.0, 1.0) { 2 }  
}
child {node [lc] at (0.0, 1.0) { 2 }  
}
child {node [lc] at (1.0, 1.0) { 1 }  
}
;
}}

\newcommand{\forestBJ}{
\tikz[planar forest ] {

\node [lc] at (0.0, 0.0) { 1 } 
;
\node [b] at (1.0, 0.0) {  } 
child {node [lc] at (-1.0, 1.0) { 2 }  
}
child {node [lc] at (0.0, 1.0) { 1 }  
}
child {node [lc] at (1.0, 1.0) { 2 }  
}
;
}}

\newcommand{\forestCJ}{
\tikz[planar forest ] {

\node [lc] at (0.0, 0.0) { 1 } 
;
\node [b] at (1.0, 0.0) {  } 
child {node [lc] at (-1.0, 1.0) { 1 }  
}
child {node [lc] at (0.0, 1.0) { 2 }  
}
child {node [lc] at (1.0, 1.0) { 2 }  
}
;
}}

\newcommand{\forestDJ}{
\tikz[planar forest ] {

\node [b] at (0.0, 0.0) {  } 
;
\node [b] at (1.0, 0.0) {  } 
;
\node [b] at (2.0, 0.0) {  } 
;
}}

\newcommand{\forestEJ}{
\tikz[planar forest ] {

\node [b] at (0.0, 0.0) {  } 
child {node [b] at (0.0, 1.0) {  }  
}
;
\node [lc] at (1.0, 0.0) { 1 } 
;
\node [lc] at (2.0, 0.0) { 1 } 
;
}}

\newcommand{\forestFJ}{
\tikz[planar forest ] {

\node [lc] at (0.0, 0.0) { 1 } 
;
\node [b] at (1.0, 0.0) {  } 
child {node [b] at (0.0, 1.0) {  }  
}
;
\node [lc] at (2.0, 0.0) { 1 } 
;
}}

\newcommand{\forestGJ}{
\tikz[planar forest ] {

\node [lc] at (0.0, 0.0) { 1 } 
;
\node [lc] at (1.0, 0.0) { 1 } 
;
\node [b] at (2.0, 0.0) {  } 
child {node [b] at (0.0, 1.0) {  }  
}
;
}}

\newcommand{\forestHJ}{
\tikz[planar forest ] {

\node [b] at (0.0, 0.0) {  } 
child {node [lc] at (0.0, 1.0) { 1 }  
}
;
\node [b] at (1.0, 0.0) {  } 
;
\node [lc] at (2.0, 0.0) { 1 } 
;
}}

\newcommand{\forestIJ}{
\tikz[planar forest ] {

\node [b] at (0.0, 0.0) {  } 
child {node [lc] at (0.0, 1.0) { 1 }  
}
;
\node [lc] at (1.0, 0.0) { 1 } 
;
\node [b] at (2.0, 0.0) {  } 
;
}}

\newcommand{\forestJJ}{
\tikz[planar forest ] {

\node [b] at (0.0, 0.0) {  } 
;
\node [b] at (1.0, 0.0) {  } 
child {node [lc] at (0.0, 1.0) { 1 }  
}
;
\node [lc] at (2.0, 0.0) { 1 } 
;
}}

\newcommand{\forestKJ}{
\tikz[planar forest ] {

\node [b] at (0.0, 0.0) {  } 
;
\node [lc] at (1.0, 0.0) { 1 } 
;
\node [b] at (2.0, 0.0) {  } 
child {node [lc] at (0.0, 1.0) { 1 }  
}
;
}}

\newcommand{\forestLJ}{
\tikz[planar forest ] {

\node [lc] at (0.0, 0.0) { 1 } 
;
\node [b] at (1.0, 0.0) {  } 
child {node [lc] at (0.0, 1.0) { 1 }  
}
;
\node [b] at (2.0, 0.0) {  } 
;
}}

\newcommand{\forestMJ}{
\tikz[planar forest ] {

\node [lc] at (0.0, 0.0) { 1 } 
;
\node [b] at (1.0, 0.0) {  } 
;
\node [b] at (2.0, 0.0) {  } 
child {node [lc] at (0.0, 1.0) { 1 }  
}
;
}}

\newcommand{\forestNJ}{
\tikz[planar forest ] {

\node [b] at (0.0, 0.0) {  } 
child {node [lc] at (-0.5, 1.0) { 2 }  
}
child {node [lc] at (0.5, 1.0) { 2 }  
}
;
\node [lc] at (1.0, 0.0) { 1 } 
;
\node [lc] at (2.0, 0.0) { 1 } 
;
}}

\newcommand{\forestOJ}{
\tikz[planar forest ] {

\node [lc] at (0.0, 0.0) { 1 } 
;
\node [b] at (1.0, 0.0) {  } 
child {node [lc] at (-0.5, 1.0) { 2 }  
}
child {node [lc] at (0.5, 1.0) { 2 }  
}
;
\node [lc] at (2.0, 0.0) { 1 } 
;
}}

\newcommand{\forestPJ}{
\tikz[planar forest ] {

\node [lc] at (0.0, 0.0) { 1 } 
;
\node [lc] at (1.0, 0.0) { 1 } 
;
\node [b] at (2.0, 0.0) {  } 
child {node [lc] at (-0.5, 1.0) { 2 }  
}
child {node [lc] at (0.5, 1.0) { 2 }  
}
;
}}

\newcommand{\forestQJ}{
\tikz[planar forest ] {

\node [b] at (0.0, 0.0) {  } 
child {node [lc] at (-0.5, 1.0) { 2 }  
}
child {node [lc] at (0.5, 1.0) { 1 }  
}
;
\node [lc] at (1.0, 0.0) { 2 } 
;
\node [lc] at (2.0, 0.0) { 1 } 
;
}}

\newcommand{\forestRJ}{
\tikz[planar forest ] {

\node [b] at (0.0, 0.0) {  } 
child {node [lc] at (-0.5, 1.0) { 1 }  
}
child {node [lc] at (0.5, 1.0) { 2 }  
}
;
\node [lc] at (1.0, 0.0) { 2 } 
;
\node [lc] at (2.0, 0.0) { 1 } 
;
}}

\newcommand{\forestSJ}{
\tikz[planar forest ] {

\node [lc] at (0.0, 0.0) { 2 } 
;
\node [b] at (1.0, 0.0) {  } 
child {node [lc] at (-0.5, 1.0) { 2 }  
}
child {node [lc] at (0.5, 1.0) { 1 }  
}
;
\node [lc] at (2.0, 0.0) { 1 } 
;
}}

\newcommand{\forestTJ}{
\tikz[planar forest ] {

\node [lc] at (0.0, 0.0) { 2 } 
;
\node [b] at (1.0, 0.0) {  } 
child {node [lc] at (-0.5, 1.0) { 1 }  
}
child {node [lc] at (0.5, 1.0) { 2 }  
}
;
\node [lc] at (2.0, 0.0) { 1 } 
;
}}

\newcommand{\forestUJ}{
\tikz[planar forest ] {

\node [lc] at (0.0, 0.0) { 2 } 
;
\node [lc] at (1.0, 0.0) { 1 } 
;
\node [b] at (2.0, 0.0) {  } 
child {node [lc] at (-0.5, 1.0) { 2 }  
}
child {node [lc] at (0.5, 1.0) { 1 }  
}
;
}}

\newcommand{\forestVJ}{
\tikz[planar forest ] {

\node [lc] at (0.0, 0.0) { 2 } 
;
\node [lc] at (1.0, 0.0) { 1 } 
;
\node [b] at (2.0, 0.0) {  } 
child {node [lc] at (-0.5, 1.0) { 1 }  
}
child {node [lc] at (0.5, 1.0) { 2 }  
}
;
}}

\newcommand{\forestWJ}{
\tikz[planar forest ] {

\node [b] at (0.0, 0.0) {  } 
;
\node [b] at (1.0, 0.0) {  } 
;
\node [lc] at (2.0, 0.0) { 1 } 
;
\node [lc] at (3.0, 0.0) { 1 } 
;
}}

\newcommand{\forestXJ}{
\tikz[planar forest ] {

\node [b] at (0.0, 0.0) {  } 
;
\node [lc] at (1.0, 0.0) { 1 } 
;
\node [b] at (2.0, 0.0) {  } 
;
\node [lc] at (3.0, 0.0) { 1 } 
;
}}

\newcommand{\forestYJ}{
\tikz[planar forest ] {

\node [b] at (0.0, 0.0) {  } 
;
\node [lc] at (1.0, 0.0) { 1 } 
;
\node [lc] at (2.0, 0.0) { 1 } 
;
\node [b] at (3.0, 0.0) {  } 
;
}}

\newcommand{\forestAK}{
\tikz[planar forest ] {

\node [lc] at (0.0, 0.0) { 1 } 
;
\node [b] at (1.0, 0.0) {  } 
;
\node [b] at (2.0, 0.0) {  } 
;
\node [lc] at (3.0, 0.0) { 1 } 
;
}}

\newcommand{\forestBK}{
\tikz[planar forest ] {

\node [lc] at (0.0, 0.0) { 1 } 
;
\node [b] at (1.0, 0.0) {  } 
;
\node [lc] at (2.0, 0.0) { 1 } 
;
\node [b] at (3.0, 0.0) {  } 
;
}}

\newcommand{\forestCK}{
\tikz[planar forest ] {

\node [lc] at (0.0, 0.0) { 1 } 
;
\node [lc] at (1.0, 0.0) { 1 } 
;
\node [b] at (2.0, 0.0) {  } 
;
\node [b] at (3.0, 0.0) {  } 
;
}}

\newcommand{\forestDK}{
\tikz[planar forest ] {

\node [b] at (0.0, 0.0) {  } 
child {node [lc] at (0.0, 1.0) { 1 }  
}
;
\node [lc] at (1.0, 0.0) { 2 } 
;
\node [lc] at (2.0, 0.0) { 2 } 
;
\node [lc] at (3.0, 0.0) { 1 } 
;
}}

\newcommand{\forestEK}{
\tikz[planar forest ] {

\node [lc] at (0.0, 0.0) { 2 } 
;
\node [b] at (1.0, 0.0) {  } 
child {node [lc] at (0.0, 1.0) { 1 }  
}
;
\node [lc] at (2.0, 0.0) { 2 } 
;
\node [lc] at (3.0, 0.0) { 1 } 
;
}}

\newcommand{\forestFK}{
\tikz[planar forest ] {

\node [lc] at (0.0, 0.0) { 2 } 
;
\node [lc] at (1.0, 0.0) { 2 } 
;
\node [b] at (2.0, 0.0) {  } 
child {node [lc] at (0.0, 1.0) { 1 }  
}
;
\node [lc] at (3.0, 0.0) { 1 } 
;
}}

\newcommand{\forestGK}{
\tikz[planar forest ] {

\node [lc] at (0.0, 0.0) { 2 } 
;
\node [lc] at (1.0, 0.0) { 2 } 
;
\node [lc] at (2.0, 0.0) { 1 } 
;
\node [b] at (3.0, 0.0) {  } 
child {node [lc] at (0.0, 1.0) { 1 }  
}
;
}}

\newcommand{\forestHK}{
\tikz[planar forest ] {

\node [b] at (0.0, 0.0) {  } 
child {node [lc] at (0.0, 1.0) { 1 }  
}
;
\node [lc] at (1.0, 0.0) { 2 } 
;
\node [lc] at (2.0, 0.0) { 1 } 
;
\node [lc] at (3.0, 0.0) { 2 } 
;
}}

\newcommand{\forestIK}{
\tikz[planar forest ] {

\node [lc] at (0.0, 0.0) { 2 } 
;
\node [b] at (1.0, 0.0) {  } 
child {node [lc] at (0.0, 1.0) { 1 }  
}
;
\node [lc] at (2.0, 0.0) { 1 } 
;
\node [lc] at (3.0, 0.0) { 2 } 
;
}}

\newcommand{\forestJK}{
\tikz[planar forest ] {

\node [lc] at (0.0, 0.0) { 2 } 
;
\node [lc] at (1.0, 0.0) { 1 } 
;
\node [b] at (2.0, 0.0) {  } 
child {node [lc] at (0.0, 1.0) { 1 }  
}
;
\node [lc] at (3.0, 0.0) { 2 } 
;
}}

\newcommand{\forestKK}{
\tikz[planar forest ] {

\node [lc] at (0.0, 0.0) { 2 } 
;
\node [lc] at (1.0, 0.0) { 1 } 
;
\node [lc] at (2.0, 0.0) { 2 } 
;
\node [b] at (3.0, 0.0) {  } 
child {node [lc] at (0.0, 1.0) { 1 }  
}
;
}}

\newcommand{\forestLK}{
\tikz[planar forest ] {

\node [b] at (0.0, 0.0) {  } 
child {node [lc] at (0.0, 1.0) { 1 }  
}
;
\node [lc] at (1.0, 0.0) { 1 } 
;
\node [lc] at (2.0, 0.0) { 2 } 
;
\node [lc] at (3.0, 0.0) { 2 } 
;
}}

\newcommand{\forestMK}{
\tikz[planar forest ] {

\node [lc] at (0.0, 0.0) { 1 } 
;
\node [b] at (1.0, 0.0) {  } 
child {node [lc] at (0.0, 1.0) { 1 }  
}
;
\node [lc] at (2.0, 0.0) { 2 } 
;
\node [lc] at (3.0, 0.0) { 2 } 
;
}}

\newcommand{\forestNK}{
\tikz[planar forest ] {

\node [lc] at (0.0, 0.0) { 1 } 
;
\node [lc] at (1.0, 0.0) { 2 } 
;
\node [b] at (2.0, 0.0) {  } 
child {node [lc] at (0.0, 1.0) { 1 }  
}
;
\node [lc] at (3.0, 0.0) { 2 } 
;
}}

\newcommand{\forestOK}{
\tikz[planar forest ] {

\node [lc] at (0.0, 0.0) { 1 } 
;
\node [lc] at (1.0, 0.0) { 2 } 
;
\node [lc] at (2.0, 0.0) { 2 } 
;
\node [b] at (3.0, 0.0) {  } 
child {node [lc] at (0.0, 1.0) { 1 }  
}
;
}}

\newcommand{\forestPK}{
\tikz[planar forest ] {

\node [b] at (0.0, 0.0) {  } 
;
\node [lc] at (1.0, 0.0) { 2 } 
;
\node [lc] at (2.0, 0.0) { 2 } 
;
\node [lc] at (3.0, 0.0) { 1 } 
;
\node [lc] at (4.0, 0.0) { 1 } 
;
}}

\newcommand{\forestQK}{
\tikz[planar forest ] {

\node [b] at (0.0, 0.0) {  } 
;
\node [lc] at (1.0, 0.0) { 2 } 
;
\node [lc] at (2.0, 0.0) { 1 } 
;
\node [lc] at (3.0, 0.0) { 2 } 
;
\node [lc] at (4.0, 0.0) { 1 } 
;
}}

\newcommand{\forestRK}{
\tikz[planar forest ] {

\node [b] at (0.0, 0.0) {  } 
;
\node [lc] at (1.0, 0.0) { 1 } 
;
\node [lc] at (2.0, 0.0) { 2 } 
;
\node [lc] at (3.0, 0.0) { 2 } 
;
\node [lc] at (4.0, 0.0) { 1 } 
;
}}

\newcommand{\forestSK}{
\tikz[planar forest ] {

\node [lc] at (0.0, 0.0) { 2 } 
;
\node [b] at (1.0, 0.0) {  } 
;
\node [lc] at (2.0, 0.0) { 2 } 
;
\node [lc] at (3.0, 0.0) { 1 } 
;
\node [lc] at (4.0, 0.0) { 1 } 
;
}}

\newcommand{\forestTK}{
\tikz[planar forest ] {

\node [lc] at (0.0, 0.0) { 2 } 
;
\node [b] at (1.0, 0.0) {  } 
;
\node [lc] at (2.0, 0.0) { 1 } 
;
\node [lc] at (3.0, 0.0) { 2 } 
;
\node [lc] at (4.0, 0.0) { 1 } 
;
}}

\newcommand{\forestUK}{
\tikz[planar forest ] {

\node [lc] at (0.0, 0.0) { 1 } 
;
\node [b] at (1.0, 0.0) {  } 
;
\node [lc] at (2.0, 0.0) { 2 } 
;
\node [lc] at (3.0, 0.0) { 2 } 
;
\node [lc] at (4.0, 0.0) { 1 } 
;
}}

\newcommand{\forestVK}{
\tikz[planar forest ] {

\node [lc] at (0.0, 0.0) { 2 } 
;
\node [lc] at (1.0, 0.0) { 2 } 
;
\node [b] at (2.0, 0.0) {  } 
;
\node [lc] at (3.0, 0.0) { 1 } 
;
\node [lc] at (4.0, 0.0) { 1 } 
;
}}

\newcommand{\forestWK}{
\tikz[planar forest ] {

\node [lc] at (0.0, 0.0) { 2 } 
;
\node [lc] at (1.0, 0.0) { 1 } 
;
\node [b] at (2.0, 0.0) {  } 
;
\node [lc] at (3.0, 0.0) { 2 } 
;
\node [lc] at (4.0, 0.0) { 1 } 
;
}}

\newcommand{\forestXK}{
\tikz[planar forest ] {

\node [lc] at (0.0, 0.0) { 1 } 
;
\node [lc] at (1.0, 0.0) { 2 } 
;
\node [b] at (2.0, 0.0) {  } 
;
\node [lc] at (3.0, 0.0) { 2 } 
;
\node [lc] at (4.0, 0.0) { 1 } 
;
}}

\newcommand{\forestYK}{
\tikz[planar forest ] {

\node [lc] at (0.0, 0.0) { 2 } 
;
\node [lc] at (1.0, 0.0) { 2 } 
;
\node [lc] at (2.0, 0.0) { 1 } 
;
\node [b] at (3.0, 0.0) {  } 
;
\node [lc] at (4.0, 0.0) { 1 } 
;
}}

\newcommand{\forestAL}{
\tikz[planar forest ] {

\node [lc] at (0.0, 0.0) { 2 } 
;
\node [lc] at (1.0, 0.0) { 1 } 
;
\node [lc] at (2.0, 0.0) { 2 } 
;
\node [b] at (3.0, 0.0) {  } 
;
\node [lc] at (4.0, 0.0) { 1 } 
;
}}

\newcommand{\forestBL}{
\tikz[planar forest ] {

\node [lc] at (0.0, 0.0) { 1 } 
;
\node [lc] at (1.0, 0.0) { 2 } 
;
\node [lc] at (2.0, 0.0) { 2 } 
;
\node [b] at (3.0, 0.0) {  } 
;
\node [lc] at (4.0, 0.0) { 1 } 
;
}}

\newcommand{\forestCL}{
\tikz[planar forest ] {

\node [lc] at (0.0, 0.0) { 2 } 
;
\node [lc] at (1.0, 0.0) { 2 } 
;
\node [lc] at (2.0, 0.0) { 1 } 
;
\node [lc] at (3.0, 0.0) { 1 } 
;
\node [b] at (4.0, 0.0) {  } 
;
}}

\newcommand{\forestDL}{
\tikz[planar forest ] {

\node [lc] at (0.0, 0.0) { 2 } 
;
\node [lc] at (1.0, 0.0) { 1 } 
;
\node [lc] at (2.0, 0.0) { 2 } 
;
\node [lc] at (3.0, 0.0) { 1 } 
;
\node [b] at (4.0, 0.0) {  } 
;
}}

\newcommand{\forestEL}{
\tikz[planar forest ] {

\node [lc] at (0.0, 0.0) { 1 } 
;
\node [lc] at (1.0, 0.0) { 2 } 
;
\node [lc] at (2.0, 0.0) { 2 } 
;
\node [lc] at (3.0, 0.0) { 1 } 
;
\node [b] at (4.0, 0.0) {  } 
;
}}

\newcommand{\forestFL}{
\tikz[planar forest ] {

\node [lc] at (0.0, 0.0) { 3 } 
;
\node [lc] at (1.0, 0.0) { 3 } 
;
\node [lc] at (2.0, 0.0) { 2 } 
;
\node [lc] at (3.0, 0.0) { 2 } 
;
\node [lc] at (4.0, 0.0) { 1 } 
;
\node [lc] at (5.0, 0.0) { 1 } 
;
}}

\newcommand{\forestGL}{
\tikz[planar forest ] {

\node [lc] at (0.0, 0.0) { 3 } 
;
\node [lc] at (1.0, 0.0) { 3 } 
;
\node [lc] at (2.0, 0.0) { 2 } 
;
\node [lc] at (3.0, 0.0) { 1 } 
;
\node [lc] at (4.0, 0.0) { 2 } 
;
\node [lc] at (5.0, 0.0) { 1 } 
;
}}

\newcommand{\forestHL}{
\tikz[planar forest ] {

\node [lc] at (0.0, 0.0) { 3 } 
;
\node [lc] at (1.0, 0.0) { 3 } 
;
\node [lc] at (2.0, 0.0) { 2 } 
;
\node [lc] at (3.0, 0.0) { 1 } 
;
\node [lc] at (4.0, 0.0) { 1 } 
;
\node [lc] at (5.0, 0.0) { 2 } 
;
}}

\newcommand{\forestIL}{
\tikz[planar forest ] {

\node [lc] at (0.0, 0.0) { 3 } 
;
\node [lc] at (1.0, 0.0) { 2 } 
;
\node [lc] at (2.0, 0.0) { 3 } 
;
\node [lc] at (3.0, 0.0) { 2 } 
;
\node [lc] at (4.0, 0.0) { 1 } 
;
\node [lc] at (5.0, 0.0) { 1 } 
;
}}

\newcommand{\forestJL}{
\tikz[planar forest ] {

\node [lc] at (0.0, 0.0) { 3 } 
;
\node [lc] at (1.0, 0.0) { 2 } 
;
\node [lc] at (2.0, 0.0) { 3 } 
;
\node [lc] at (3.0, 0.0) { 1 } 
;
\node [lc] at (4.0, 0.0) { 2 } 
;
\node [lc] at (5.0, 0.0) { 1 } 
;
}}

\newcommand{\forestKL}{
\tikz[planar forest ] {

\node [lc] at (0.0, 0.0) { 3 } 
;
\node [lc] at (1.0, 0.0) { 2 } 
;
\node [lc] at (2.0, 0.0) { 3 } 
;
\node [lc] at (3.0, 0.0) { 1 } 
;
\node [lc] at (4.0, 0.0) { 1 } 
;
\node [lc] at (5.0, 0.0) { 2 } 
;
}}

\newcommand{\forestLL}{
\tikz[planar forest ] {

\node [lc] at (0.0, 0.0) { 3 } 
;
\node [lc] at (1.0, 0.0) { 2 } 
;
\node [lc] at (2.0, 0.0) { 2 } 
;
\node [lc] at (3.0, 0.0) { 3 } 
;
\node [lc] at (4.0, 0.0) { 1 } 
;
\node [lc] at (5.0, 0.0) { 1 } 
;
}}

\newcommand{\forestML}{
\tikz[planar forest ] {

\node [lc] at (0.0, 0.0) { 3 } 
;
\node [lc] at (1.0, 0.0) { 2 } 
;
\node [lc] at (2.0, 0.0) { 2 } 
;
\node [lc] at (3.0, 0.0) { 1 } 
;
\node [lc] at (4.0, 0.0) { 3 } 
;
\node [lc] at (5.0, 0.0) { 1 } 
;
}}

\newcommand{\forestNL}{
\tikz[planar forest ] {

\node [lc] at (0.0, 0.0) { 3 } 
;
\node [lc] at (1.0, 0.0) { 2 } 
;
\node [lc] at (2.0, 0.0) { 2 } 
;
\node [lc] at (3.0, 0.0) { 1 } 
;
\node [lc] at (4.0, 0.0) { 1 } 
;
\node [lc] at (5.0, 0.0) { 3 } 
;
}}

\newcommand{\forestOL}{
\tikz[planar forest ] {

\node [lc] at (0.0, 0.0) { 3 } 
;
\node [lc] at (1.0, 0.0) { 2 } 
;
\node [lc] at (2.0, 0.0) { 1 } 
;
\node [lc] at (3.0, 0.0) { 3 } 
;
\node [lc] at (4.0, 0.0) { 2 } 
;
\node [lc] at (5.0, 0.0) { 1 } 
;
}}

\newcommand{\forestPL}{
\tikz[planar forest ] {

\node [lc] at (0.0, 0.0) { 3 } 
;
\node [lc] at (1.0, 0.0) { 2 } 
;
\node [lc] at (2.0, 0.0) { 1 } 
;
\node [lc] at (3.0, 0.0) { 3 } 
;
\node [lc] at (4.0, 0.0) { 1 } 
;
\node [lc] at (5.0, 0.0) { 2 } 
;
}}

\newcommand{\forestQL}{
\tikz[planar forest ] {

\node [lc] at (0.0, 0.0) { 3 } 
;
\node [lc] at (1.0, 0.0) { 2 } 
;
\node [lc] at (2.0, 0.0) { 1 } 
;
\node [lc] at (3.0, 0.0) { 2 } 
;
\node [lc] at (4.0, 0.0) { 3 } 
;
\node [lc] at (5.0, 0.0) { 1 } 
;
}}

\newcommand{\forestRL}{
\tikz[planar forest ] {

\node [lc] at (0.0, 0.0) { 3 } 
;
\node [lc] at (1.0, 0.0) { 2 } 
;
\node [lc] at (2.0, 0.0) { 1 } 
;
\node [lc] at (3.0, 0.0) { 2 } 
;
\node [lc] at (4.0, 0.0) { 1 } 
;
\node [lc] at (5.0, 0.0) { 3 } 
;
}}

\newcommand{\forestSL}{
\tikz[planar forest ] {

\node [lc] at (0.0, 0.0) { 3 } 
;
\node [lc] at (1.0, 0.0) { 2 } 
;
\node [lc] at (2.0, 0.0) { 1 } 
;
\node [lc] at (3.0, 0.0) { 1 } 
;
\node [lc] at (4.0, 0.0) { 3 } 
;
\node [lc] at (5.0, 0.0) { 2 } 
;
}}

\newcommand{\forestTL}{
\tikz[planar forest ] {

\node [lc] at (0.0, 0.0) { 3 } 
;
\node [lc] at (1.0, 0.0) { 2 } 
;
\node [lc] at (2.0, 0.0) { 1 } 
;
\node [lc] at (3.0, 0.0) { 1 } 
;
\node [lc] at (4.0, 0.0) { 2 } 
;
\node [lc] at (5.0, 0.0) { 3 } 
;
}}

\newcommand{\Lab}{\text{Lab}}

\newcommand\restrict[1]{\raisebox{-.5ex}{$|$}_{#1}}
\newcommand{\graft}{\curvearrowright}

\newcommand{\gl}{\diamond}

\newcommand{\dF}{\mathbb{F}}
\newcommand{\dFdec}{\dF^{\text{dec}}}
\newcommand{\tildedFdec}{\tilde{\dF}^{\text{dec}}}
\newcommand{\ET}{\mathcal{ET}}
\newcommand{\EF}{\mathcal{EF}}

\newcommand{\blank}{{-}}
\newcommand{\tildegraft}{\mathbin{\tilde{\graft}}}

\DeclareMathOperator{\Irr}{Irr}
\DeclareMathOperator{\Conn}{Conn}

\allowdisplaybreaks

\newtheorem{theorem}{Theorem}[section]
\newtheorem{definition}[theorem]{Definition}
\newtheorem*{definition*}{Definition}

\newtheorem{proposition}[theorem]{Proposition}
\newtheorem{lemma}[theorem]{Lemma}
\newtheorem{remark}[theorem]{Remark}
\newtheorem*{remark*}{Remark}

\newtheorem*{remarks*}{Remarks}

\newtheorem{ass}[theorem]{Assumption}
\newtheorem*{notation*}{Notation}
\newtheorem{ex}[theorem]{Example}
\newtheorem*{ex*}{Example}

\newtheorem*{exs*}{Examples}

\newtheorem*{app*}{Application}
\newtheorem{conjecture*}{Conjecture}

\def\ts{\thinspace}

\usepackage{shuffle}

\title{High order integration of stochastic dynamics on Riemannian manifolds with frozen flow methods
}

\author{
Eugen Bronasco\textsuperscript{1}, Adrien Busnot Laurent\textsuperscript{2} and Baptiste Huguet\textsuperscript{2}
}

\begin{document}
\footnotetext[1]{
Department of Mathematical Sciences, Chalmers University of Technology and University of Gothenburg, Sweden. bronasco@chalmers.se.}
\footnotetext[2]{
Univ Rennes, INRIA (Research team MINGuS), IRMAR (CNRS UMR 6625) and ENS Rennes,
France.
Adrien.Busnot-Laurent@inria.fr,
Baptiste.Huguet@math.cnrs.fr.}

\maketitle


\begin{abstract}
We present a new class of numerical methods for solving stochastic differential equations with additive noise on general Riemannian manifolds with high weak order of accuracy.
In opposition to the popular approach with projection methods, the proposed methods are intrinsic: they only rely on geometric operations and avoid coordinates and embeddings.
We provide a robust and general convergence analysis and an algebraic formalism of exotic planar Butcher series for the computation of order conditions at any high order.
To illustrate the methodology, an explicit method of second weak order is introduced, and several numerical experiments confirm the theoretical findings and extend the approach for the sampling of the invariant measure of Riemannian Langevin dynamics.

\smallskip

\noindent
{\it Keywords:\,} geometric numerical integration, stochastic differential equations, Riemannian manifolds, Riemannian Langevin, ergodicity, Lie-group methods, frozen flow, Butcher series, exotic series, post-Lie algebra, Hopf algebra, order conditions.
\smallskip

\noindent
{\it AMS subject classification (2020):\,} 16T05, 41A58, 60H35, 37M25, 65L06, 70H45.
\end{abstract}


\section{Introduction}

Let $(\MM,g)$ be a smooth connected complete Riemannian manifold equipped with a metric $g$ and its Levi-Civita connection $\nabla$.
Let $E_1$,\dots,$E_D$ be a frame, possibly redundant, that is, $D$ smooth vector fields with $\dim(\MM)\leq D$ such that
\[\Span_{\R} (E_1(y),\dots,E_D(y))=T_y\MM, \quad y\in\MM.\]
Given a smooth vector field $F(x)=\sum_{d=1}^D f^{d}(x)E_d(x)$, we consider Stratonovich stochastic differential equations on $\MM$ of the form,
\begin{equation}
\label{equation:SDE_Strato}
dX(t)=F(X(t))dt+\sqrt{2}\sum_{d=1}^D E_d(X(t)) \circ dW_d(t).
\end{equation}
Equation \eqref{equation:SDE_Strato} can be seen as a stochastic differential equation with additive noise w.r.t.\ts the frame $(E_d)$.
Under reasonable assumptions, the flow of \eqref{equation:SDE_Strato} is well-posed on $\MM$. The aim of the present paper is to introduce a new class of intrinsic one-step integrators
\begin{equation}
\label{equation:integrator}
X_{n+1}=\Phi_h(X_n),
\end{equation}
where $\Phi_h$ is a (random) local diffeomorphism of $\MM$ for a timestep $h$ small enough, for solving dynamics of the form \eqref{equation:SDE_Strato} with high order of accuracy in the weak sense.
We take inspiration from Lie-group integrators \cite{MuntheKaas95lbt, MuntheKaas97nio, MuntheKaas98rkm, MuntheKaas99hor, Iserles00lgm, MuntheKaas24gio} and more precisely from Crouch-Grossman and commutator-free methods \cite{Crouch93nio, Owren99rkm, Celledoni03cfl, Owren06ocf}. In opposition to \cite{Malham08slg}, we consider an approach with frames allowing us to define a new versatile class of methods that works on any smooth Riemannian manifold.
An important feature of the new integrators is that they do not rely on the use of embeddings or on local coordinates. Their formulation, convergence analysis, order theory, and implementation are entirely intrinsic.

It is crucial that a numerical approximation respects the geometry of the problem and lies on $\MM$, especially for the sampling of measures. In particular, our methods extend naturally to the sampling of measures interpreted as invariant measure of ergodic dynamics.
Let us consider for instance the central example from molecular dynamics of Riemannian Langevin dynamics:
\begin{align}
\label{equation:Langevin}
dX(t)=-\nabla V(X(t))dt+\sqrt{2}dB_\MM(t),
\end{align}
where $V$ is a smooth potential on $\MM$ and 
\[dB_\MM(t) = \sum_{d=1}^D E_d(X(t))  dW_d(t)\] 
is a Brownian increment (see \cite{Hsu02sao}) for $(E_d)$ an orthonormal frame basis. Equation \eqref{equation:Langevin} is a special case of Equation \eqref{equation:SDE_Strato}, once rewritten as 
\[X(t) =-\sum_{d=1}^D(E_d [V]E_d+\nabla_{E_d}E_d)(X(t))dt +\sqrt{2}\sum_{d=1}^D E_d(X(t)) \circ dW_d(t),\]
where the vector field $F=-\nabla V-\sum \nabla_{E_d}E_d$ contains the Itô correction.
Note that the choice $V=0$ yields Brownian dynamics $B_\MM(t)$ on $\MM$.
Riemannian Langevin dynamics \eqref{equation:Langevin} are naturally ergodic under mild assumptions, that is, the dynamics in long time behaves according to a deterministic probability measure, called the invariant measure, which takes the explicit form $d\mu_\infty=e^{-V}d\vol$, where $d\vol$ denotes the Riemannian volume measure. The generator of the SDE is $\LL=-\nabla V\cdot\nabla +\Delta_g$, where $\Delta_g$ is the Laplace-Beltrami operator.
Actually, an integrator of weak order $p$ automatically has at least order $p$ for the invariant measure \cite{Talay90eot}. We refer for instance to~\cite{Debussche12wbe} and references therein for appropriate assumptions to obtain the ergodicity of a numerical scheme in the context of dynamics in $\R^d$.
To the best of our knowledge, there does not exist any intrinsic method of order greater than one for sampling the invariant measure of Riemannian Langevin dynamics.

There exists a handful of extrinsic stochastic integrators on manifolds, that is, integrators that rely on an embedding of $\MM$ in a Euclidean space of higher dimension.
In the case of smooth embedded manifolds $\MM=\{x\in\R^D,\zeta(x)=0\}$ defined by constraints, the approach using projection methods is the most popular in the literature (see \cite{Lelievre10fec} and references therein). We cite in particular the constrained Euler method:
\[X_{n+1}=X_n+hF(X_n)+\sqrt{2h}\xi_n+\lambda_n \nabla \zeta(X_n),\quad \zeta(X_{n+1})=0,\]
where $\lambda_n$ is a Lagrange multiplier and the $\xi_n$ are independent standard Gaussian random vectors.
While the approach using projection methods is straightforward to implement and widely used in applications, these methods rely on an embedding of $\mathcal{M}$ into a vector space of much higher dimension and use non-intrinsic quantities, so that they often face severe timestep restrictions.
Moreover, the algebraic structure associated to order theory is difficult and does not allow at the present time for the construction of methods of arbitrarily high order, nor for a systematic approach to stochastic backward error analysis \cite{Bronasco22cef}.
Stochastic Lie group methods seldom appear in the literature \cite{Malham08slg, Davidchack15nla, Marjanovic15asa, Ableidinger17wsr, Marjanovic18nmf, Cheng22eso}, but focus on low order approximations on specific manifolds (typically variants of Lie-Euler on the sphere), with an extrinsic convergence analysis, and no high order theory.
In the recent work~\cite{Bharath23sae}, a new approach for sampling~\eqref{equation:Langevin} is developed using the following intrinsic integrator, called Riemannian Langevin method
\[X_{n+1}=\exp^{\text{Riem}}(-h\nabla V(X_n)+ \sqrt{2h}\xi_{n+1}) X_n,\]
where $\exp^{\text{Riem}}$ is the geodesic exponential for the Levi-Civita connection $\nabla$.
This method has order one for solving \eqref{equation:Langevin} in the weak sense and for the invariant measure.
One could be tempted to say that such a method is exact for $V=0$, but it still has only order one, in contrast to the Euclidean setting where it is exact in law.

We introduce the first intrinsic methods of high weak order for solving \eqref{equation:SDE_Strato} using only evaluations of $F$ and the frame.
In particular, we highlight the following new explicit integrator of second weak order for solving \eqref{equation:SDE_Strato}, and using the minimal number of evaluations of $F$ per step:
\begin{align*}
H_n&=\exp\bigg(\sum_{d=1}^D \Big(\frac{1}{2} hf^d(X_n)+\sqrt{h}\xi^{d,1}_n \Big) E_d\bigg) X_n\\
X_{n+1}&=
\exp\bigg(\sum_{d=1}^D \Big(
\big(\frac{\sqrt{2}}{2}-1\big) hf^d(X_n)+ (2-\sqrt{2}) hf^d(H_n)
+\big(1-\sqrt{2}\big)\sqrt{h}\xi^{d,1}_n+\sqrt{h}\xi^{d,2}_n \Big)E_d
\bigg)\\&
\exp\bigg(\sum_{d=1}^D \Big(
\big(1-\frac{\sqrt{2}}{2}\big) hf^d(X_n)+ (\sqrt{2}-1) hf^d(H_n)
+\sqrt{2h}\xi^{d,1}_n \Big)E_d
\bigg)
X_n.
\end{align*}
where the $\xi^{d,l}_n$ are independent standard Gaussian random variables (or discrete symmetric approximations with matching moments up to order 4) and we use frozen flows (see Section \ref{section:new_FF_integrators}).
One of the originalities of our approach is that we present a general framework for the analysis of a new class of stochastic intrinsic methods at any order on any manifold in the weak sense using a new algebraic formalism similar to the one of Butcher series. The new methods outperform in accuracy and versatility the previous attempts.
The methods apply to any dynamics of the form \eqref{equation:SDE_Strato} and to any manifold, though the need for efficient computation of the intrinsic operations (frozen-flows) make them best suited for specific manifolds, such as homogeneous spaces.

The calculations of the order conditions are intricate and require the use of appropriate algebraic tools. 
While Butcher trees and series \cite{Hairer06gni,Butcher16nmf,McLachlan17bsa,Butcher21bsa} represent naturally Taylor expansions of deterministic flows in $\R^D$, planar trees and Lie-Butcher series were introduced for the study of flows on manifolds \cite{Iserles00lgm, Owren06ocf, MuntheKaas08oth} and later extended for the study of the connection algebra \cite{MuntheKaas13opl, Ebrahimi15otl, AlKaabi22aao, Grong23pla, MuntheKaas23lat}.
On the other hand, the formalism of exotic series was introduced in \cite{Laurent20eab,Laurent21ata,Bronasco22ebs} for the construction of integrators for solving stochastic dynamics with additive noise with high order in the weak sense and for the invariant measure.
This formalism, combined with the aromatic B-series \cite{Chartier07pfi,Iserles07bsm,Bogfjellmo22uat,Laurent23tab,Laurent23tld}, was extended in \cite{Laurent21ocf} into the exotic aromatic series formalism for the numerical integration of SDEs on embedded manifolds with projection methods.
The geometric and algebraic properties of the exotic formalism of trees were later studied in \cite{Bronasco22cef, Laurent23tue} (see also \cite{Chartier10aso,McLachlan16bsm,MuntheKaas16abs,Bogfjellmo19aso} in the deterministic setting).
We propose a new extension of exotic and Lie--Butcher series, called exotic Lie series, for the systematic computation of order conditions for the new frozen flow methods at any order in the weak sense.
Following the modern approach, which relies extensively on Hopf algebras \cite{Chartier10aso, Calaque11tih, Lundervold11hao, Bogfjellmo19aso, Bronasco22ebs, Rahm22aoa, Bronasco22cef, Ebrahimi24aso, Busnot26osr}, we identify an appropriate Hopf algebraic structure on exotic forests in the context of stochastic numerics on manifolds. In particular, we describe the Munthe-Kaas--Wright and Grossman--Larson Hopf algebras over exotic forests, which are dual to each other and serve as the key algebraic constructions for the systematic description of the terms appearing in the Taylor expansions of the exact solution and the frozen flow methods introduced in this paper. Ultimately, this leads to the derivation of the order conditions for the new methods at any order.

The paper is organised in the following way.
We present in Section \ref{section:preliminaries} the main numerical results of this paper: a robust framework for the Riemannian convergence analysis of stochastic integrators on manifolds, the new stochastic frozen flow methods, and the associated order conditions for high weak order.
Section \ref{section:exotic_Lie_series} introduces the formalism of planar decorated and exotic forests and translates the geometric operations in the connection algebra in terms of algebraic operations on graphs. This approach is then applied for the explicit description at any order of the weak Taylor expansion of the exact flow of \eqref{equation:SDE_Strato}, of the new frozen flow methods, and of their associated order conditions.
The Hopf algebra structure of exotic planar forests is further detailed in Section \ref{section:decorated_exotic_forests}. In particular, we generalise the Munthe-Kaas-Wright Hopf algebra \cite{MuntheKaas08oth} to the stochastic context.
Numerical experiments confirm the theoretical findings in Section \ref{section:numerical_experiments} and outlooks and future works are presented in Section \ref{section:conclusion}.

\section{New frozen flow methods of high weak order}
\label{section:preliminaries}

We generalise the standard Euclidean framework for weak convergence in the context of Riemannian manifolds, and we introduce new frozen flow integrators and their order conditions for high weak order. We emphasize that the following convergence analysis is new, even for the underlying deterministic differential equations (see also \cite{Curry20col}).

\subsection{Stochastic differential equations on manifolds}
\label{section:exact_solution}

For $v\in T\MM$, we denote $|v| = g(v,v)^{1/2}$, its Riemannian norm. Let $o$ be an arbitrary point on $\MM$. For $x\in\MM$, we define $r(x)=d(o,x)$ the Riemannian distance function. This function is $1$-Lipschitz and the squared distance $r^2$ is smooth on $\MM\setminus Cut_0$, where $Cut_o$ denote the cut-locus  of $o$ in $\MM$, i.e the set of points where minimizing geodesics issued from $o$ are no longer unique or minimizing. This allows us to define moments of a $\MM$-valued random variable.

Let us introduce the main classes of test functions. The space $\CC^\infty_c(\MM)$ denotes the class of smooth compactly supported functions. For $\phi\in\CC^\infty_c(\MM)$ we denote by $d\phi$ its differential and by $|d\phi|$ its Riemannian norm on $T^*\MM$, inherited from the norm on $T\MM$. We denote by $v[\phi](x)$ the differential of a test function $\phi$ in a direction $v\in T_x\MM$, which can be rewritten, with duality bracket, as
\[v[\phi](x)= \langle d\phi(x), v\rangle.\]
We shall use the wider class of test functions $\CC^{\infty}_P(\MM)$, which consists of smooth functions satisfying polynomial growth estimates of the form:
\[ \abs{E_{d_q}[\dots E_{d_1}[\phi]\dots]}(x)\leq C(1+r(x)^K),\quad q=0,1,\dots.\]
Given a vector field $F$ and a fixed decomposition in the frame $F=\sum f^dE_d$, we say that $F\in\mathfrak{X}_P(\MM)$ if its components are Lipschitz continuous, $\abs{df^d}\leq C$, and satisfy: $f^d\in \CC^{\infty}_P(\MM)$, $d=1,\dots,D$.
We mention that considering test functions and vector fields with polynomial growth of their derivatives up to a given order $p+2$ would be sufficient for the construction of numerical schemes of order up to $p$.

\begin{ass}
\label{assumption:regularity}
The vector fields $E_1$,\dots, $E_D$ are smooth and bounded. The vector field $F$ is in $\mathfrak{X}_P(\MM)$.
Moreover, there exist constants $\nu \geq1$ and $\lambda\in\R$ such that on $M\setminus Cut_o$, we have
\begin{equation}
\label{equation:croissance_L}
\LL r^2 \leq \nu + \lambda r^2,
\end{equation}
where $\LL$ is the generator
\[\LL\phi
=F[\phi]
+\sum_{d=1}^D E_{d}[E_{d}[\phi]].\]
\end{ass}

Under the smoothness assumption on $F$, $E_1$,\dots, $E_D$, there exists a unique solution of \eqref{equation:SDE_Strato}, up to a stopping time $\tau$. Following \cite{Hsu02sao}, we recall that a stochastic process $X$, defined on a stochastic interval $[0,\tau]$, is a solution of \eqref{equation:SDE_Strato} if and only if for all test functions $\phi\in\CC^\infty_c(\MM)$, we have
\[
\phi(X(t))=\phi(X_0)+ \int_0^t F[\phi](X(t)) dt+ \sqrt{2}\sum_{d=1}^D \int_0^t E_d[\phi](X(t))\circ dW_d(t),\quad 0<t<\tau.
\]
Unlike the Euclidean case, where Lipschitz conditions are sufficient to avoid finite time explosion, on a manifold, a blow-up can be induced from the geometry, even for a Brownian motion (see \cite{Hsu02sao} for examples). The Lyapunov-like assumption \eqref{equation:croissance_L} ensures the completeness of the SDE. This means that the solution $X$ of \eqref{equation:SDE_Strato} is defined on $\R_+$, i.e. $\tau =+\infty$ a.s. In addition to the completeness, it provides finite moments estimates (even finite exponential moments) for the solution of the SDE. This assumption is adapted from \cite{Thompson20bma}. In the case where $(E_i)_i$ is a basis there exists a new metric $\tilde{g}$ such that $(E_i)_i$ is an orthonormal basis. Using this metric and the associated Levi-Civita connection $\tilde{\nabla}$, we find
\[\LL\phi = \Delta_{\tilde{g}}\phi + (F + \tilde{\nabla}_{E_i}E_i)[\phi],\]
which is the framework of \cite{Thompson20bma}. The general case is treated likewise, by extracting local basis on a cover of $\MM$.

\begin{remark}
There exist various simpler criteria which ensure that inequality \eqref{equation:croissance_L} is satisfied. The first criterion is compactness, which can be used for classical spaces such as the sphere $\S^n$ and the Lie group $\SO_n(\R)$. This condition being seldom satisfied, a second classical criterion is the Bakry-\'Emery criterion from \cite{Bakry86cne}. This criterion applies to the case of equation \eqref{equation:Langevin}, where the generator writes as $\LL=\Delta-\nabla V \cdot\nabla$. One says that the potential $V$ satisfies the criterion if there exists $\kappa\in\R$ such that  
\[\Ric + \Hess(V)\geq \kappa.\]
The term $\Ric$ denotes the Ricci tensor on $(\MM,g)$ and is defined as the trace of the Riemann tensor. In this context, the symmetric operator $\Ric + \Hess(V)$ is called the Bakry-\'Emery curvature and the term $\Hess(V)$ can be interpreted as an additional curvature. Note that this handy criterion has been adapted in \cite{Antonyuk07nonexp} for equation \eqref{equation:SDE_Strato} with a lower bound of the operator $\Ric-\nabla F$. We also mention the moment conditions from \cite{Li94stochdiff} which imply stochastic completeness.
\end{remark}

By a standard induction, we obtain the following estimates.
\begin{lemma}\label{prop:growth_L}
Under Assumption \ref{assumption:regularity}, for all $q\geq 1$ we have $\LL (r^{2q})\leq C(1+r^{2q})$.
\end{lemma}
\begin{proof}
The proof comes from classical $\Gamma$ calculus for diffusions (see \cite{Bakry14aag} for instance). For all $u\in\CC^2(\MM)$ and $\eta\in\CC^2(\R)$, we have
\[\LL\phi(u) = \eta'(u)\LL u + \eta''(u)\Gamma(u),\]
where 
\[\Gamma(u) =\sum_{i=1}^D |E_i[u]|^2.\]
Applying this formula to $u= r^2$ and $\eta : x\mapsto x^q$, we obtain the result.
\end{proof}

A key tool for the weak expansion of the exact flow is the Kolmogorov equation. Under Assumption \ref{assumption:regularity}, the SDE \eqref{equation:SDE_Strato} generates a semigroup, well-defined on $\CC^\infty_P(M)$, and Markovian. Namely, for all $\phi\in\CC^\infty_c(M)$, the function $u(t,x)=\E[\phi(X(t))|X(0)=x]$ satisfies the Kolmogorov equation
\begin{equation}
\label{equation:Kolmogorov_u}
\partial_t u=\LL u, \quad u(0,x)=\phi(x).
\end{equation}

In order to compare an expansion of the exact solution to an expansion of a numerical integrator, we ask for more regularity on the semigroup itself. Since the generator $\LL$ is elliptic, the semigroup $u$ is $\CC^\infty$ on $]0,T[\times\MM$. We further assume that the semigroup preserves the $\CC^\infty_P$ regularity.

\begin{ass}[Regularity]
\label{assumption:semigroupregularity}
For all test function $\phi\in\CC^{\infty}_P(\MM)$, the semigroup \eqref{equation:Kolmogorov_u} satisfies $u\in \CC^\infty\left(]0,T[,\CC^{\infty}_P(\MM)\right)$, that is, for all $k\geq0$, there exist $C>0$ and $\kappa\geq0$, depending on $k$, $T$, and $\phi$ such that for all $x\in\MM$
\[\sup_{t\in]0,T[}|\partial^k_t u(t,x)|\leq C\left(1+r^{\kappa}(x)\right).\]
\end{ass}

Note that Assumption \ref{assumption:semigroupregularity} is satisfied automatically in $\R^d$ with our regularity assumptions. For clarity, we leave the study of sufficient conditions to obtain polynomial growth estimates of the semigroup for future works.

We are now able to derive the Taylor expansion of the semigroup in the following result.
The proof is identical to the Euclidean setting \cite{Talay90eot, Milstein04snf}.
\begin{proposition}
    \label{prop:Taylor_expansion_exact}
Under Assumption \ref{assumption:semigroupregularity} and for any fixed order $p>0$, for all $\phi\in\CC^{\infty}_P(M)$ and all $h\leq h_0$ small enough, $u$ has the Taylor expansion
\begin{equation}
\label{equation:dvp_exact}
u(h,x) = \phi(x) + \sum_{j=1}^{p} \frac{h^j}{j!} \LL^{j}\phi(x)+ h^{p+1}R_p^h(\phi,x),\quad x\in \MM,
\end{equation}
with $\abs{R_p^h(\phi,x)}\leq C(1+r(x)^K)$, where $C$ and $K$ do not depend on $x$ and $h$.
\end{proposition}

\subsection{Weak order theory on Riemannian manifolds}

Following the Euclidean works \cite{Talay90eot, Milstein04snf}, we present a general theory for achieving high order weak estimates for intrinsic methods. The approach does not rely on an embedding and does not assume compactness of $\MM$. In order to obtain estimates, we use the Riemannian structure of $\MM$. This structure is not needed to perform the calculations of order conditions of the new methods.

\begin{definition}
An intrinsic one-step integrator of the form \eqref{equation:integrator} is of local weak order $p$ if for all $\phi\in\CC^{\infty}_P(\MM)$, there exists $C>0$ such that the following estimate holds for all $x\in\MM$ and all $h\leq h_0$ small enough,
\[\abs{\E[\phi(X_1)|X_0=x]-\E[\phi(X(h))|X_0=x]}\leq C(1+r(x)^K)h^{p+1}.\]
Similarly, the integrator is of global weak order $p$ if for all $\phi\in \CC^{\infty}_P(\MM)$, $T>0$, $Nh=T$, $h\leq h_0$ small enough, and $x\in \MM$, the following estimate is satisfied:
\[\sup_{n=0,\dots,N}\abs{\E[\phi(X_n)]-\E[\phi(X(nh))]}\leq Ch^{p}.\]
\end{definition}

Let us assume that the integrator satisfies a Taylor expansion, often called Talay-Tubaro expansion \cite{Talay90sod} in the Euclidean context, that is similar to \eqref{equation:dvp_exact}.
\begin{ass}
\label{assumption:dvp_num}
For all $\phi \in \CC^{\infty}_P(\MM)$ and all $h\leq h_0$ small enough, the numerical integrator~\eqref{equation:integrator} has a weak Taylor expansion of the form
\begin{equation}
\label{equation:dvp_num}
\E[\phi(X_1)|X_0=x] = \phi(x) + \sum_{j=1}^{p} h^j \AA_{j}\phi(x)+ h^{p+1}R_p^h(\phi,x),\quad x\in \MM,
\end{equation}
where the remainder satisfies~$\abs{R_p^h(\phi,x)}\leq C(1+r(x)^K)$ and the~$\AA_j$ are linear differential operators.
\end{ass}

We adapt the standard assumption of bounded moments in the Riemannian context. We shall give sufficient assumptions on the integrator in Subsection \ref{section:new_FF_integrators} so that this assumption is satisfied.
\begin{ass}
\label{assumption:moments}
The numerical integrator~\eqref{equation:integrator} has bounded moments of any order $K\geq 0$,
\begin{equation}
\label{equation:moments}
\sup_{n=0,\dots,N} \E[r(X_n)^K] \leq C(T), \quad Nh=T.
\end{equation}
\end{ass}

Under the above assumptions, we obtain the following convergence result, in the spirit of the Euclidean theory \cite{Talay90sod, Talay90eot, Milstein04snf}.
\begin{theorem}
\label{theorem:weak_high_order}
Consider a one-step integrator of the form~\eqref{equation:integrator} for solving \eqref{equation:SDE_Strato}. Under Assumptions \ref{assumption:regularity}, \ref{assumption:semigroupregularity}, \ref{assumption:dvp_num}, and \ref{assumption:moments}, if the Taylor expansion \eqref{equation:dvp_num} of the integrator satisfies
\begin{equation}
\label{equation:condition_TT_high_order}
\AA_j=\frac{1}{j!}\LL^j,\quad j=1,\dots p,
\end{equation}
then the integrator has global weak order $p$.
\end{theorem}

\begin{proof}
Let $X^x(t_n)$ be the exact solution of \eqref{equation:SDE_Strato} at time $t_n=nh$ starting from $x\in\MM$ and $X_n^x$ be the n-th step of the numerical solution.
The global error satisfies, thanks to a telescopic sum argument,
\begin{align*}
\E[\phi(X_N)-\phi(X(T))]
&=\sum_{i=1}^N \E[\phi(X^{X^x_{N-i}}(t_i))-\phi(X^{X^x_{N-i+1}}(t_{i-1}))]\\
&=\sum_{i=1}^N \E[\phi(X^{X^{X^x_{N-i}}(h)}(t_{i-1}))-\phi(X^{X^{X^x_{N-i}}_1}(t_{i-1}))]\\
&=\sum_{i=1}^N \E[\tilde{\phi}(X_1^{X^x_{N-i}})-\tilde{\phi}(X^{X^x_{N-i}}(h))],
\end{align*}
where $\tilde{\phi}(x)= \E[\phi(X^x(t_{i-1})|x]\in\CC^{\infty}_P(\MM)$ using Assumption \ref{assumption:semigroupregularity}.
Then, the local order $p$ condition, together with Assumptions \ref{assumption:dvp_num} and \ref{assumption:moments} give
\begin{align*}
\E[\phi(X_N)-\phi(X(T))]
&\leq \sum_{i=1}^N Ch^{p+1}\left(1+\E\left[r(X^x_{N-i})^K\right]\right)
\leq CNh^{p+1}\leq Ch^p.
\end{align*}
Hence, the global weak order $p$.
\end{proof}

\subsection{Stochastic frozen flow integrators}
\label{section:new_FF_integrators}

The deterministic Lie-group methods are divided into two categories that correspond to the two kinds of coordinates: the Runge-Kutta-Munthe-Kaas (RKMK) methods (see \cite{MuntheKaas95lbt, MuntheKaas97nio, MuntheKaas98rkm, MuntheKaas99hor}) and the frozen flow methods. We choose the latter option in this paper for the sake of generality, as the RKMK approach relies on the existence of a transitive group action on $\MM$.
The frozen flow methods originally appeared as Crouch-Grossmann methods \cite{Crouch93nio, Owren99rkm}. The generalisation presented in the papers \cite{Celledoni03cfl, Owren06ocf} allowed the construction of more efficient methods, called commutator-free Lie-group methods.

The methods rely heavily on the concept of frozen vector fields and flows: given a decomposition of a (potentially random) vector field in the frame $G=\sum_d g^d E_d$, and a point $x\in\MM$, the vector field $G_x\in \mathfrak{X}(\MM)$  frozen at point $x$ is
\[G_x(p)=\sum_d g^d(x) E_d(p), \quad p\in\MM.\]

The solution of the ODE $y'(t)=G(y(t))$ with $y(0)=p$ is denoted $\exp(tG)p$. The frozen flow $\exp(tG_x)p$ is the exact flow driven by the frozen vector field $G_x$.

\begin{lemma}
If $G\in\mathfrak{X}_P(\MM)$ then, the frozen flow $\exp(tG_x)p$ is globally defined. Moreover, it satisfies
\begin{equation}\label{equation:frozenbound}
d\left(p, \exp(tG_x)p\right) \leq C\left(1+r(x)\right)
\end{equation}
\end{lemma}

\begin{proof}
The coefficients of the ODE being globally Lipschitz continuous, there exists a global solution. Let us denote $\gamma : t\mapsto \exp(tG_x)p$. It is a smooth curve, joining $p$ and $\exp(G_x)p$. As $|E_d|$ are bounded, $\gamma$ has a bounded speed
\[ \abs{\dot{\gamma}_t}^2 \leq D\sum_{d=1}^D\abs{g^d(x)}^2|E_d(\gamma_t)|^2\leq C(1+r^2(x)).\]
Therefore, we have
\[d(p,\exp(tG_x)p) \leq \int_0^t \abs{\dot{\gamma}_s}\, ds\leq C(1+r(x))\] 
Hence, the result.
\end{proof}

\begin{remark}
\label{remark:FF_formulation}
The frozen vector fields depend on a choice of frame $(E_d)$, but the definitions of the frozen vector field, and thus, of the frozen flow, do not depend on the choice of a decomposition of $G$ in this frame.
Indeed, given two decompositions of $G$ in the frame $(E_d)$ 
\[G(p) = \sum_{d=1}^D g_d(p) E_d(p) =\sum_{d=1}^D  \hat{g}_d(p) E_d(p),\] 
we define two, presumably different, frozen vector fields $G_x$ and $\hat{G}_x$ associated to these  decompositions. Now, let us extract a basis $(E_d)_{d\in I}$ of the frame valid in an open set $U\subset \MM$.
The vector fields $G_x$ and $\hat{G}_x$ decompose uniquely in $U$ as
\[G_x(p)=\sum_{d\in I} \alpha_d E_d(p),\quad \hat{G}_x(p) =\sum_{d\in I} \hat{\alpha}_d E_d(p).\] 
As $G(x)=G_x(x) = \hat{G}_x(x)$, we obtain $\alpha_d=\hat{\alpha}_d$ and $G_x=\hat{G}_x$.
Therefore, the formulation of the integrators is independent of the choice of a decomposition in the frame.
\end{remark}

The central assumption that motivates the use of these methods is that while computing the exact flow of a SDE is difficult, computing the flow associated to a frozen vector field often is a much simpler problem for a good choice of frame. For instance, in $\R^D$ with the standard frame basis $E_i=\partial_i$, the frozen flow is the Euler method: $\exp(G_x)p=p+G(x)$.
A second example is given by homogeneous manifolds. In this context, a natural choice of frame is derived from a basis $(A_d)$ of the Lie algebra $\mathfrak{g}$ by $E_d(y)=A_d\cdot y$, where we use for simplicity the same notation $\cdot$ for the Lie group and Lie algebra action on $\MM$.
On matrix homogeneous manifolds, the frozen flow is explicitly given by
\[\exp(G_x)p=\Exp\Big(\sum_d g^{d}(x)A_d\Big) p,\]
where $\Exp$ is the matrix exponential.
We emphasize that a significant part of the manifolds used in numerical experiments and applications of stochastic dynamics on manifolds are homogeneous manifolds, such as Lie-groups, spheres, the Stiefel manifold or symmetric positive definite matrices.
We further discuss the implementation of the methods in Section \ref{section:practical_implementation}.

We consider the following new class of stochastic frozen flow methods, in the spirit of \cite{Owren99rkm} for the deterministic part and \cite{Rossler04rkm, Debrabant10rkm, Laurent20eab} for the stochastic part,
\begin{align}
H^i_n=\prod_{k=1,\dots, K}^{\leftarrow} \exp\bigg(\sum_{d=1}^D \Big(h\sum_{j=1}^s Z^{0}_{i,j,k} f^{d}(H^j_n)+ \sqrt{h} Z^{d}_{i,k} \Big) E_{d} \bigg) X_n \nonumber\\
\label{equation:def_CGsto}
X_{n+1}=\prod_{k=1,\dots, K}^{\leftarrow} \exp\bigg(\sum_{d=1}^D \Big(h\sum_{i=1}^s z^{0}_{i,k} f^{d}(H^i_n)+ \sqrt{h} z^{d}_{k} \Big) E_{d} \bigg) X_n,
\end{align}
where $\overset{\leftarrow}{\prod}_{k} a_k=a_K\dots a_1$.
The coefficients $z^{0}_{i,K}$, $Z^{0}_{i,j,k}$ are deterministic and $z^{d}_{k}$, $Z^{d}_{i,k}$ are of the form
\[
z^{d}_{k}=\sum_{l=1}^L \hat{z}_{k,l} \xi^{d,l}_n,\quad Z^{d}_{i,k}=\sum_{l=1}^L \hat{Z}_{i,k,l} \xi^{d,l}_n,
\]
where the $\xi^{d,l}$ are independent standard Gaussian random variables or approximations of Gaussian random variables with the same $2p$ first moments, vanishing odd moments, and finite moments of all order. The dependence in $n$ of the Runge-Kutta coefficients $z^{d}_{k}$, $Z^{d}_{i,k}$ is omitted for clarity, and we enforce that the random coefficients between two steps in time are independent. The constant $K$ is the number of exponentials used per stage. It is well-known \cite{Iserles00lgm} that choosing $K=1$ does not allow to reach arbitrarily high order already in the context of ODEs.

\begin{ex}
The simplest consistent method of the form \eqref{equation:def_CGsto} is the frozen flow Euler method
\begin{equation}
\label{equation:Euler_FF}
X_{n+1}=\exp\bigg(\sum_{d=1}^D \Big(h f^{d}(X_n)+ \sqrt{2h} \xi^d_n \Big) E_{d} \bigg) X_n.
\end{equation}
\end{ex}

The assumptions of Subsection \ref{section:exact_solution} are in general not sufficient to ensure bounded moments and bounds on the Talay-Tubaro remainder of the numerical integrators. Aside from the compactness assumption, we provide the following sufficient condition.

\begin{proposition}
\label{proposition:well_posedness_integrator}
Assume that the frame $E_d$, $d=1,\dots,D$, and the vector field $F$ satisfy Assumption \ref{assumption:regularity} and
\[
\abs{E_{d_1}[E_{d_1}[ f^d]]}\leq C.
\]
Assume further that for an infinite number of arbitrarily large $q$:
\[
\abs{E_{d_k} [\dots E_{d_1} [r^q]\dots]}\leq C(1+r^{q-k}), \quad k\leq 4.
\]
Then, all consistent explicit frozen flow methods of the form \eqref{equation:def_CGsto} satisfy the bounded moments property and the Talay-Tubaro expansion of Assumptions \ref{assumption:dvp_num} and \ref{assumption:moments}.
\end{proposition}

\begin{proof}
Let us first show that the first remainder in the expansion \eqref{equation:dvp_num} acts on the distance function by 
\begin{equation}\label{equation:sqrtdist_num}
\abs{R_1^h(r^{q},x)}\leq C(1+r(x)^{q}),
\end{equation}
for an infinite number of arbitrarily large $q$.
Without loss of generality, we consider the following method, where $G_{y}=\sum_{d=1}^D g^d_y E_d$ and $g^d_y=h f^{d}(y)+ \sqrt{2h} \xi^d_n$,
\[X_{n+1}=\exp(G_{H_n})X_n,\quad H_{n}=\exp(G_{X_n})X_n.\]
A Taylor expansion around $x=X_0$ yields (see Section \ref{section:exotic_Lie_series} for the detailed calculations),
\begin{align*}
f^d(H)&=f^d(x)+ G_x[\phi](x)+\int_0^1 (1-t) g^{d_1}_x g^{d_2}_x E_{d_2} [E_{d_1} [f^d]](\exp(t G_x)x)dt,
\end{align*}
and we obtain
\[
\abs{f^d(H)}\leq C(\xi^d)(1+r(x)),\quad
\abs{\E[f^{d}(H)-f^{d}(x)]}\leq Ch(1+r(x)),
\]
where $C(\xi^d)$ is a polynomial in $\abs{\xi^d}$.
Then, a Taylor expansion for $\phi=r^q$ yields
\begin{align*}
\E[\phi(X_{1})]&=\phi(x)+h\LL\phi(x)+h^2 R_1^h(\phi,x)\\
R_1^h(\phi,x)&=\E[
\frac{1}{h}(f^{d}(H)-f^{d}(x))E_d[\phi](x)
+\frac{1}{2}f^{d_1}(H) f^{d_2}(H) E_{d_2}[E_{d_1}[\phi]](x)\\
&+\frac{1}{3}f^{d_1}(H)\xi^{d_2}\xi^{d_3} (E_{d_3}[E_{d_2}[E_{d_1}[\phi]]]+E_{d_3}[E_{d_1}[E_{d_2}[\phi]]]+E_{d_1}[E_{d_3}[E_{d_2}[\phi]]])(x)\\
&+\frac{h}{6}f^{d_1}(H) f^{d_2}(H)f^{d_3}(H) E_{d_3}[E_{d_2}[E_{d_1}[\phi]]](x)\\
&+\frac{1}{h^2}\int_0^1 \frac{(1-t)^3}{3!}g^{d_1}_Hg^{d_2}_Hg^{d_3}_Hg^{d_4}_HE_{d_4} [E_{d_3}[E_{d_2} [E_{d_1} [\phi]]]](\exp(tG_H)x)dt].
\end{align*}
The estimate on $R_1^h(\phi,x)$ is straightforwardly obtained from \eqref{equation:frozenbound}. Indeed, one shows that the distance function satisfies
\[
r(\exp(tG_H)x)\leq C(\xi^d)(1+r(x)),
\]
where $C(\xi^d)$ is a polynomial in $\abs{\xi^d}$.

For the Talay-Tubaro expansion of Assumption \ref{assumption:dvp_num}, the expansion \eqref{equation:dvp_num} is obtained straightforwardly with a Taylor expansion (see also \cite{Talay90eot}). The operators $\AA_j$ typically write with polynomials in $F$, $E_d$, $d=1,\dots, D$, and their derivatives\footnote{The Talay-Tubaro expansion typically writes in the jet space over $F$ and the frame $(E_d)$. Explicit formulas are provided in Theorem \ref{theorem:expansion_numerical_flow}.}.
The estimate on the remainder is obtained thanks to the regularity Assumption \ref{assumption:semigroupregularity}.
For Assumption \ref{assumption:moments}, we use the Talay-Tubaro expansion \eqref{equation:dvp_num} applied to $\phi=r^{q}$:
\[
\E[r(X_{n+1})^{q}] = \E[r(X_n)^{q}] + h \E[\LL r(X_n)^q]+ h^2 \E[R^h_1(r^q,X_n)],
\]
where $\abs{\E[\LL\phi(X_n)]}\leq C(1+\E[r(X_n)^{q}])$ by Proposition \ref{prop:growth_L} and $\abs{\E[R^h_1(r^q,X_n)]}\leq C(1+\E[r(X_n)^{q}])$ by equation \eqref{equation:sqrtdist_num} (for $q$ large enough).
Thus, we find
\[
\E[r(X_{n+1})^{q}] \leq C(1+h)\E[r(X_n)^{q}] + Ch,
\]
and a standard discrete Gr\"onwall argument yields the bounded moments \eqref{equation:moments}.
\end{proof}

\subsection{High-order frozen flow integrators}
\label{section:order_conditions_FF}

Let us present the order conditions of the new class of frozen flow methods \eqref{equation:def_CGsto} for solving the stochastic dynamics \eqref{equation:SDE_Strato}, as well as an example of second order frozen flow integrator.
We rely on factorial sums for expressing the order conditions.
If $\textbf{k}\in \{1,\dots,K\}^n$ is a multi-index that takes $r_1$ times the value $k_1$, \dots, $r_p$ times the value $k_p$, the factorial of $\textbf{k}$ is
\[
\textbf{k}!=r_1!\dots r_p!.
\]
The factorial sum is given for a finite set $S\subset \{1,\dots,K\}^n$ of multi-indices and a functional $f$ by 
$
\sum_{\textbf{k}\in S}^! f(\textbf{k})
=\sum_{\textbf{k}\in S} \frac{1}{\textbf{k}!} f(\textbf{k}).$

The weak order conditions for stochastic Runge-Kutta methods in $\R^d$ are presented in \cite{Rossler04rkm, Debrabant10rkm, Laurent20eab, Bronasco22ebs} and require 4 order conditions for weak order two.
On embedded manifolds, the approach with projection methods of \cite{Laurent21ocf} yields 28 order conditions, applies only in codimension one, and relies on an extrinsic approach with tedious calculations (see \cite[App.\ts D]{Laurent21ocf}).
The approach proposed here is intrinsic, versatile, generalises the order theory of Lie-group methods to the stochastic setting, and yields a surprisingly low number of order conditions for weak order two.

\begin{theorem}
\label{theorem:weak_order_conditions}
Let a frozen flow integrator of the form \eqref{equation:def_CGsto} satisfy the order conditions in Table \ref{table:weak_order_conditions}, then, under Assumption \ref{assumption:semigroupregularity} and the assumptions of Proposition \ref{proposition:well_posedness_integrator}, the integrator has local weak order two for solving \eqref{equation:SDE_Strato}. The additional order conditions for weak order three are presented in Appendix \ref{app:order_cond_3}.
\end{theorem}

\begin{figure}[ht]
\begin{longtable}{C|C|C}
\text{Exotic forest } \pi &\text{Differential } \F(\pi)[\phi] &\text{Order condition } a(\pi)=e(\pi)\\\hline
\forestA & f^i E_i[\phi] & z^{0}_{i,k_1}=1\\
\forestB & E_{d_1}[E_{d_1}[\phi]] & \sum^!_{k_1\geq k_2} \E[z^{d_1}_{k_2}  z^{d_1}_{k_1}] = 1\\
\hline
\forestC &  f^jE_{j}[f^{i}] E_{i}[\phi] & z^{0}_{i,k_2} Z^{0}_{i,j,k_1} = \frac{1}{2}\\
\forestD & E_{d_1}[E_{d_1}[f^{i}]] E_{i}[\phi] & \sum^!_{k_2\geq k_3} \E[Z^{d_1}_{i,k_3} Z^{d_1}_{i,k_2}] z^{0}_{i,k_1} = \frac{1}{2}\\
\forestE & f^j f^i E_{j}[E_{i}[\phi]] & \sum^!_{k_1\geq k_2} z^{0}_{j,k_2} z^{0}_{i,k_1} = \frac{1}{2} \\
\forestF & E_{d_1}[f^{i}] E_{i}[E_{d_1}[\phi]] & \sum^!_{k_1\geq k_2} z^{0}_{i,k_2} \E[Z^{d_1}_{i,k_3} z^{d_1}_{k_1}] = 0 \\
\forestG & E_{d_1}[f^{i}] E_{d_1}[E_{i}[\phi]] & \sum^!_{k_2\geq k_1} z^{0}_{i,k_2} \E[Z^{d_1}_{i,k_3} z^{d_1}_{k_1}]  = 1 \\
\forestH & f^{i} E_{i}[E_{d_1}[E_{d_1}[\phi]]] & \sum^!_{k_1\geq k_2\geq k_3} z^{0}_{i,k_3} \E[z^{d_1}_{k_2} z^{d_1}_{k_1}] = \frac{1}{2} \\
\forestI & f^{i} E_{d_1}[E_{i}[E_{d_1}[\phi]]] & \sum^!_{k_1\geq k_3\geq k_2} z^{0}_{i,k_3} \E[z^{d_1}_{k_2} z^{d_1}_{k_1}] = 0 \\
\forestJ & f^{i} E_{d_1}[E_{d_1}[E_{i}[\phi]]] & \sum^!_{k_3\geq k_1\geq k_2} z^{0}_{i,k_3} \E[z^{d_1}_{k_2} z^{d_1}_{k_1}] = \frac{1}{2} \\
\forestK & E_{d_2}[E_{d_2}[E_{d_1}[E_{d_1}[\phi]]]] & \sum^!_{k_1\geq k_2\geq k_3\geq k_4}  \E[z^{d_2}_{k_4} z^{d_2}_{k_3}] \E[z^{d_1}_{k_2} z^{d_1}_{k_1}] = \frac{1}{2} \\
\forestL & E_{d_2}[E_{d_1}[E_{d_2}[E_{d_1}[\phi]]]] & \sum^!_{k_1\geq k_2\geq k_3\geq k_4} \E[z^{d_2}_{k_4} z^{d_2}_{k_2}] \E[z^{d_1}_{k_3} z^{d_1}_{k_1}] = 0 \\
\forestM & E_{d_1}[E_{d_2}[E_{d_2}[E_{d_1}[\phi]]]] & \sum^!_{k_1\geq k_2\geq k_3\geq k_4} \E[z^{d_2}_{k_3} z^{d_2}_{k_2}] \E[z^{d_1}_{k_4} z^{d_1}_{k_1}] = 0\\
\caption{Order conditions of frozen flow methods up to weak order 2. The sums on the indices that do not satisfy inequalities (except the $d_i$) are omitted for clarity. The order conditions do not depend on the dimension of the problem.}
\label{table:weak_order_conditions}
\end{longtable}
\end{figure}

Theorem \ref{theorem:weak_order_conditions} is a direct consequence of Theorem \ref{theorem:expansion_exact_flow} and Theorem \ref{theorem:expansion_numerical_flow}, postponed to Section \ref{section:exotic_Lie_series}. The approach relies on a new algebraic framework of exotic planar forests and series that gives the explicit expression of the operators $\AA_j$ in \eqref{equation:condition_TT_high_order}. It allows us to provide the order conditions for arbitrarily high weak order.

\begin{remark}
\label{remark:shuffle_identities}
We shall prove in Section \ref{section:Butcher_expansion_flows} that the coefficient maps $a$ and $e$ of the numerical and exact flow are characters of the shuffle algebra over exotic forests $(\EF,\shuffle,\Delta_\cdot)$.
Following \cite{Owren06ocf}, an important consequence is that the order conditions are not independent. Up to second order, the following "shuffle relations" are satisfied (and similarly for $e$):
\begin{align*}
a(\forestN)^2&=2a(\forestO),\\
a(\forestP)a(\forestQ)&=a(\forestR)+a(\forestS)+a(\forestT),\\
a(\forestU)^2&=2a(\forestV)+2a(\forestW)+2a(\forestX).
\end{align*}
The number of order conditions is not given by the number of entries in Table \ref{table:weak_order_conditions}, as one has to take into account the shuffle relations.
In particular, there are 2 conditions for order 1, 8 conditions for order 2, and 73 conditions for order 3 (against 2, 11, and 95 exotic forests of order 1,2, and 3).
This is further discussed in Remark \ref{rmk:shuffle_relations}.
\end{remark}

Let us now provide an example of a second order method.
Let $\xi^{d,1}_n$, $\xi^{d,2}_n$, $d=1,\dots,D$, be bounded independent random variables satisfying
\[\E[(\xi^{d,l}_n)^2]=1, \quad \E[(\xi^{d,l}_n)^4]=3, \quad \E[(\xi^{d,l}_n)^{2q+1}]=0, \quad q=0,1,2,\dots.\]
For weak order two, one can use for instance
\[
\P(\theta=0)=\frac{2}{3},\quad \P(\theta=\pm \sqrt{3})=\frac{1}{6}
\]
so that the first four moments coincide with the ones of standard Gaussians.


\paragraph{New SFF2 method.}
The following scheme is of weak order two for solving equation \eqref{equation:SDE_Strato}. It is explicit, intrinsic, uses two evaluations of the vector field $F$ per step, and three frozen flow exponentials.
\begin{align}
H^1_n&=X_n \nonumber\\
\label{equation:new_weak_method}
H^2_n&=\exp\bigg(\sum_{d=1}^D \Big(\frac{1}{2} hf^d(H^1_n)+\sqrt{h}\xi^{d,1}_n \Big) E_d\bigg) X_n\\
X_{n+1}&=
\exp\bigg(\sum_{d=1}^D \Big(
\big(\frac{\sqrt{2}}{2}-1\big) hf^d(H^1_n)+ (2-\sqrt{2}) hf^d(H^2_n)
+\big(1-\sqrt{2}\big)\sqrt{h}\xi^{d,1}_n+\sqrt{h}\xi^{d,2}_n \Big)E_d
\bigg) \nonumber\\&
\exp\bigg(\sum_{d=1}^D \Big(
\big(1-\frac{\sqrt{2}}{2}\big) hf^d(H^1_n)+ (\sqrt{2}-1) hf^d(H^2_n)
+\sqrt{2h}\xi^{d,1}_n \Big)E_d
\bigg)
X_n. \nonumber
\end{align}


\begin{ex}
\label{example:BD}
Brownian dynamics on $\MM$ are defined as the solution of
\[dX(t)=-\frac{1}{2}\sum_{d=1}^D (\nabla_{E_d}E_d)(X(t))dt
+\sum_{d=1}^D E_d(X(t)) \circ dW_d(t),\]
where $E_d$ is an orthonormal frame of $\MM$.
In the specific case of a Lie group equipped with a bi-invariant metric, we find $\sum_{d} \nabla_{E_d}E_d=0$ and a second order integrator is
\begin{equation}
\label{equation:new_Brownian_method}
X_{n+1}=
\exp\bigg(\sum_{d=1}^D \Big(\big(\frac{\sqrt{2}}{2}-1\big)\sqrt{h}\xi^{d,1}_n+\frac{\sqrt{2}}{2}\sqrt{h}\xi^{d,2}_n\Big) E_d\bigg)
\exp\bigg(\sum_{d=1}^D \sqrt{h}\xi^{d,1}_n E_d\bigg)
X_n.
\end{equation}
\end{ex}

\section{Order theory of stochastic frozen flow methods}
\label{section:exotic_Lie_series}

This section is devoted to the formalisation of the order theory of stochastic frozen flow methods with the help of the new algebraic formalism of exotic Lie series governing the Taylor expansions of the exact flow and the frozen flow integrators.

\subsection{Connection algebra of vector fields}
\label{section:connection_algebra}

The theory of the frozen flow methods relies on an additional geometric structure on $\MM$. The idea is to define an affine connection $\triangleright$ on $TM$ whose geodesics are frozen flows. Such a connection is given by the choice of a right inverse to $(x,v)\in\MM\times\R^D\mapsto v_d E_d(x) \in TM$, i.e. a choice of decomposing vector fields in the frame $(E_d)$. A convenient way to choose such an inverse is to extract a basis $(E_d)_{d\in I}$ from the frame. Then, the connection $\triangleright$ is uniquely defined by 
\[\triangleright E_d =0, \quad d\in I.\]
The extraction of a basis is licit only locally. However, for a given initial value, the new integrators (with bounded random variables) evolve in one step in an open neighbourhood $U$ of $\MM$. Taking the stepsize small enough, we can extract a frame basis from the frame $(E_d)$.
Following Remark \ref{remark:FF_formulation}, we do not lose any generality in assuming that $(E_d)$ is a frame basis: the methods do not depend on the choice of a basis, but the Taylor expansion in the frame bundle does. We shall write our statements globally on $\MM$ for the sake of simplicity, and all of them are valid in a neighbourhood of the starting point of the method.

The new connection can be seen as a product $\triangleright\colon \mathfrak{X}(\MM)\times \mathfrak{X}(\MM)\rightarrow \mathfrak{X}(\MM)$ given by
\[
F\triangleright G =\sum_{d=1}^D F[g^d] E_d, \quad G=\sum_{d=1}^D g^d E_d.
\]
The product is well-defined as there is a unique decomposition of $G$ in the frame basis.
The connection $\triangleright$ can be seen as the Weitzenböck connection w.r.t.\ts the frame basis. We extend the notation to test functions: $F\triangleright \phi=F[\phi]$ (independent of the choice of connection).

\begin{remark}\label{remark://trans}
Given a vector field $G$, the frozen flow $t\mapsto\exp(tG(p))p$ is the geodesic for the connection $\triangleright$ starting at $p$ with direction $G(p)$. Moreover, the parallel transport of $G(p)$ along any curve going from $p$ to $q$ is given by the frozen vector field $G_p(q)$.
\end{remark}

This new connection $\triangleright$ defines a geometric structure far from the original structure associated to the metric and the Levi-Civita connection $\nabla$. Indeed, $\triangleright$ is neither compatible with the metric nor torsion-free. Let $\llbracket F,G \rrbracket$ denote the Jacobi bracket of vector fields. 
The torsion $T$ of the connection $\triangleright$ is defined as
\[T(F,G) = F\triangleright G - G\triangleright F - \llbracket F,G \rrbracket.\]
This tensor is entirely determined by $T(E_i,E_j) = -\llbracket E_i,E_j \rrbracket$, and does not vanish in general.
An additional difference between $\triangleright$ and $\nabla$ concerns the curvature.
As for a general connection, the curvature tensor $R$ is defined as
\[
R(F,G)H = F\triangleright(G\triangleright H) - G\triangleright(F\triangleright H) - \llbracket F,G \rrbracket \triangleright H,
\]
and it vanishes in our context, i.e. $R=0$. The flatness of $\triangleright$ explains the path-independence of parallel transport, raised in Remark \ref{remark://trans}.

\begin{remark}
In the deterministic context, the definition and analysis of Lie group methods were naturally tied to Lie group structures. For example, the connection $\triangleright$ naturally arises as the $(-)$ Cartan connection associated to left/right-invariant vector fields. In this context, the torsion is constant and $-T(-,-)$ defines a Lie bracket, which naturally yields a post-Lie algebra structure \cite{Ebrahimi15otl,Grong23pla}. 
We emphasize that the frozen-flow methods, as well as their analysis, apply on any smooth manifold, not just on Lie groups.
\end{remark}

Let us now consider the (associative, non-commutative) concatenation product $\cdot$ of vector fields $F_1$, \dots, $F_n$ as the differential operator
\[
(F_n\cdots F_2 \cdot F_1)\triangleright \phi=\sum_{i_1,\dots, i_n} f_n^{i_n}\dots f_1^{i_1} E_{i_n}[ \dots  E_{i_1}[\phi]\cdots],
\]
defined on the tensor algebra $(T(\mathfrak{X}(\MM)),\cdot)$.
This product is called the frozen composition of differential operators.
In particular, the Laplace-Beltrami operator is given by $\sum_d E_d\cdot E_d$ for any orthonormal frame $E_d$.

We shall see that the Taylor expansions of the numerical and exact flows for the equation \eqref{equation:SDE_Strato} typically write in $T(\mathfrak{X}(\MM))$.
The number and complexity of the differential operators involved in the Taylor expansions of numerical and exact flows motivate the introduction of decorated and exotic planar forests.


\subsection{Decorated planar forests}
\label{section:decorated_forests}

We introduce decorated planar forests and decorated Lie series. 

\begin{definition}
  \label{def:FD}
    A \emph{decorated planar tree} $\tau$ is a connected, oriented graph in which every node has an outgoing edge, except for a node called the \emph{root}. It is equipped with a map $\alpha : V(\tau) \to D$ from the set of nodes to the set of decorations. An ordered, possibly empty, collection of decorated planar trees is called a \emph{decorated planar forest}. The sets of decorated planar trees and forests are denoted by $T_D$ and $F_D$, respectively.
\end{definition}

The vector spaces spanned by $T_D$ and $F_D$ are denoted by $\TT_D$ and $\FF_D$. Decorated planar trees $\TT_D$ are used as a convenient formalism for representing vector fields $\mathfrak{X}(\MM)$ that appear in a Taylor expansion of a numerical or exact solution of a differential equation, while decorated planar forests $\FF_D$ are used to represent differential operators $T(\mathfrak{X}(\MM))$ that appear in a Taylor expansion of a test function applied to a numerical or exact solution. We note that $\FF_D = T(\TT_D)$.

We use decorated planar forests $\FF_D$ with $D = \{ \bullet, \circ^1, \dots, \circ^L \}$ to represent the Taylor expansion of a test function applied to the one-step of a frozen flow method introduced in Section \ref{section:preliminaries}.
We consider the elementary differential map $\dFdec : \FF_D \to T(\mathfrak{X}(\MM))$ that translates the algebraic structure to differential operators.

\begin{definition}
  \label{def:dFdec}
  Let the elementary differential map $\dFdec : \FF_D \to T(\mathfrak{X}(\MM))$ be defined for $\pi \in \FF_D$ and $\phi \in \CC^\infty(\MM)$ as
  \[ \dFdec (\pi) \triangleright \phi := \sum_{\underset{v \in V(\pi)}{i_v = 1}}^D \Bigg(\prod_{v \in V} E_{I_\Pi(v)} [\dFdec(v)^{i_v}]\Bigg) E_{I_R}[\phi], \]
  with $\dFdec(\bullet) = f$ and $\dFdec(\circ^l) = \xi^{d,l} E_d$ where $R$ is the set of roots of $\pi$, $V$ is the set of nodes of $\pi$, $\Pi(v)$ is the set of predecessors of $v$ ordered from right to left, $E_I[\phi] = E_{i_p}[ \dots E_{i_1}[\phi]\dots]$ for $I = (i_p, \dots, i_1)$. The notation can also be frozen for $x, p \in \MM$,
  \[ (\dFdec_x (\pi) \triangleright \phi)(p) := \sum_{\underset{v \in V(\pi)}{i_v = 1}}^D \Bigg(\prod_{v \in V} E_{I_\Pi(v)} [\dFdec(v)^{i_v}](x)\Bigg) E_{I_R}[\phi](p), \]
\end{definition}

\begin{ex}
For example, the decorated forest $\forestY$ has two trees and thus represents the second order differential operator
\[\dFdec (\forestAB) \triangleright \phi = \xi^{i,1} E_i [f^j] f^k E_{j}[E_{k} [\phi]].\]
For the sake of clarity, the indices $i$, $j$, $k$ each correspond to a node of the fully decorated forest $\forestBB$.
The operator can alternatively be rewritten in terms of the products $\triangleright$ and $\cdot$ as
\[\dFdec (\forestCB) = ((\xi^{i,1} E_i)\triangleright F) \cdot F.\]
\end{ex}

Decorated Lie series are formal series of differential operators indexed by decorated planar forests. They do not converge in general, but their truncations at any order belong to $T(\mathfrak{X}(\MM))$.
\begin{definition}
\label{def:dec_series}
Given a one form $a \in \FF_D^*$, called a coefficient map, a decorated Lie S-series is the following formal power series in $h$ in $T(\mathfrak{X}(\MM))$,
\[
  S_h^{\text{dec}}(a) \triangleright \phi := \sum_{\pi \in F_D} h^{\abs{\pi}} a(\pi) \dFdec(\pi) \triangleright \phi, \quad S^{\text{dec}}(a)=S^{\text{dec}}_1(a).
\]
A decorated Lie series is only indexed on decorated trees and is the formal vector field,
\[
B_h^{\text{dec}}(a) := \sum_{\tau \in T_D} h^{\abs{\tau}} a(\tau) \dFdec(\tau), \quad B^{\text{dec}}(a)=B_1^{\text{dec}}(a).
\]
These series can also be frozen:
\[
S_{h,p}^{\text{dec}}(a) \triangleright \phi := \sum_{\pi \in F_D} h^{\abs{\pi}} a(\pi) \dFdec_p(\pi) \triangleright \phi, \quad B_{h,p}^{\text{dec}}(a) := \sum_{\tau \in T_D} h^{\abs{\tau}} a(\tau) \dFdec_p(\tau).
\]
\end{definition}

We endow the space $\FF_D$ with the shuffle product $\shuffle$ defined inductively for $\tau, \gamma \in \TT_D$ and $\pi, \eta \in \FF_D$ as
\[ \tau \pi \shuffle \gamma \eta  =  \tau (\pi \shuffle \gamma \eta) + \gamma (\tau \pi \shuffle \eta).  \]
For example,
\begin{align*}
   \forestDB \shuffle \forestEB &= \forestFB + \forestGB + \forestHB, \\
   \forestIB \shuffle \forestJB &= \forestKB + \forestLB + \forestMB + \forestNB + \forestOB + \forestPB.
\end{align*}
A case of particular interest is when the coefficient map $a$ is a character over $(\FF_D,\shuffle)$, that is, when $a(\pi_1\shuffle \pi_2)=a(\pi_1)a(\pi_2)$.
This condition is satisfied by the frozen flow, as expressed in Proposition \ref{prop:a_character}.

\begin{proposition}
  \label{prop:a_character}
    Let $\tilde{a}\in \TT_D^*$, the frozen flow associated to the vector field $B^{\text{dec}}(\tilde{a})$ satisfies
    \[\phi(\exp(B_p^{\text{dec}}(\tilde{a})))=\exp^\cdot(B^{\text{dec}}(\tilde{a}))\triangleright \phi=S^{\text{dec}}(a)\triangleright \phi,\]
    where $a\in\FF_D^*$ extends $\tilde{a}$ to a character of $(\FF_D,\shuffle)$, that is, $a(\pi_1\shuffle \pi_2)=a(\pi_1)a(\pi_2)$. Moreover,
    \begin{equation}
      \label{eq:a_character}
    a(\tau_1\cdots \tau_n)=\frac{1}{n!}\tilde{a}(\tau_1)\dots \tilde{a}(\tau_n).
    \end{equation}
\end{proposition}
\begin{proof}
  The proof is analogous to the proof of Lemma 2.4 from \cite{Owren06ocf}. The fact that $a \in \FF_D^*$ is a character with respect to the shuffle product follows from the formula \eqref{eq:a_character} using the fact that there are ${n+m \choose n}$ terms in $\pi \shuffle \eta$ with $\pi$ and $\eta$ having $n$ and $m$ trees, respectively.
\end{proof}

\begin{remark}
Representing Taylor expansions of vector fields with formal series indexed by tree-like structures is very natural from the geometric point of view.
In \cite{McLachlan16bsm,MuntheKaas16abs,Laurent23tue}, it is shown that exotic, aromatic and standard B-series are completely characterised by universal geometric properties of equivariance and locality. The extension of these properties to the planar case are open questions, even for standard Lie-Butcher series.
\end{remark}

\subsection{Exotic planar forests}
\label{section:exotic_forests}

This section adds planarity to the exotic forests from \cite{Laurent20eab, Bronasco22ebs, Bronasco22cef}. 
We use decorated planar forests to define exotic planar forests in Definition \ref{def:EF} which are the main object of interest for representing weak stochastic Taylor expansions. Exotic planar forests are decorated planar forests in which the decoration is used to denote the pairing between certain leaves. Let $\SS_\N$ denote the group of permutations of natural numbers.

\begin{definition}
\label{def:EF}
An exotic forest is a decorated forest~$(\pi, \alpha)\in F_D$ with the decorations~$D=\{\bullet\}\cup \N$,~$\N=\{1,2,3,\dots\}$, that follows the following rules.
  If a natural number is used as a decoration, then it must decorate exactly two leaves, that is,~$|\alpha^{-1}(n)|\in \{0, 2\}$ for all $n \in \N$.
  Two exotic forests~$(\pi_1, \alpha_1)$ and~$(\pi_2, \alpha_2)$ are considered to be identical if~$\pi_1 = \pi_2$ and there exists a map~$\varphi : D \to D$ with $\varphi \restrict{\N} \in \SS_\N$ and~$\varphi(\bullet) = \bullet$ such that~$\alpha_1 = \varphi \circ \alpha_2$.
  The set of numbered leaves that correspond to the same number, that is $\alpha^{-1}(n)$ for $n \in \N$, is called a \emph{liana}.
\end{definition}

We gather the numbered nodes in $V_{\circ}$ and the black nodes in $V_{\bullet}$. The order of an exotic forest~$\pi$ is then defined as
\[\abs{\pi}=\abs{V_{\bullet}}+\frac{\abs{V_{\circ}}}{2}.\]
The set of exotic forests is denoted by $EF$.
The set of all exotic forests of order $p$ is $EF^p$, and the ones of order less than or equal to $p$ is $EF^{\leq p}$. Exotic forests with one root are called exotic trees and form a set denoted~$ET$. The corresponding vector spaces are denoted by~$\EF$ and $\ET$, and similarly for the graded subspaces $\EF^{p}$ and $\EF^{\leq p}$.

The exotic trees of order up to three are the following. Note that the order does not coincide with the number of black nodes in general.
\[ \forestQB, \ \forestRB, \ \forestSB, \ \forestTB, \ \forestUB, \ \forestVB, \ \forestWB, \ \forestXB, \ \forestYB, \ \forestAC, \ \forestBC, \ \forestCC, \ \forestDC, \ \forestEC.\]
We also note that the choice of the number used to denote the liana does not matter, for example, the following exotic forests are identical:
\[ \forestFC = \forestGC. \]


Our main focus with exotic planar forests is to represent some specific elementary differentials in $\TT(\mathfrak{X}(\MM))$ and to translate the connection algebra to a natural algebraic structure on $\EF$.

The vector space $\EF$ of exotic forests forms an algebra~$(\EF, \cdot)$ w.r.t.\ts the non-commutative concatenation product.
We emphasize that it does not coincide with the tensor algebra $T^\cdot(\ET)$ as for instance $\forestHC$ cannot be expressed as the concatenation of exotic trees.

We introduce the new formalism of exotic Lie series for the study of order conditions of stochastic frozen flow methods. The formalism mixes the features of the exotic formalism \cite{Laurent20eab, Laurent21ata, Bronasco22ebs, Bronasco22cef} for the stochastic part and the planar forests formalism \cite{Iserles00lgm, Owren06ocf, MuntheKaas08oth} for Lie-group methods.

We consider the elementary differential map $\dF : \EF \to T(\mathfrak{X}(\MM))$ on exotic forests, which mirrors the elementary differential map $\dFdec$ on decorated forests, with the crucial addition of the term $\delta_{i_L}$. For a given exotic forest $\pi$, with vertex pairing $\alpha$ and indices $i_v$ for $v \in V(\pi)$, the term $\delta_{i_L}$ is defined to be $0$ if there exist $v,w \in V(\pi)$ such that $i_v \neq i_w$ and $\alpha(v)=\alpha(w)$ (that is, $(v,w)$ is a liana), and $1$ otherwise. Thus, $\delta_{i_L}$ ensures that the indices associated with paired vertices are equal, which is a crucial property motivating the introduction of exotic forests.

\begin{definition}
  \label{def:exotic_elementary_differential}
  Let the map $\dF : \EF \to T(\mathfrak{X}(\MM))$ for $\pi\in EF$ and $\phi\in\CC^\infty(\MM)$ be defined as
\begin{align*}
\F(\pi) \triangleright \phi=\sum_{\underset{v\in V}{i_v=1,\dots,D}}
\delta_{i_{L}} \left(\prod_{v\in V_{\bullet}} E_{I_{\Pi(v)}}[f^{i_v}] \right)
E_{I_{R}}[\phi],
\end{align*}
where~$R$ is the set of roots,~$\Pi(v)$ is the set of predecessors of~$v$ ordered from right to left, $E_I=E_{i_p}\dots E_{i_1}$ for $I={i_p,\dots,i_1}$, and $\delta_{i_{L}}=0$ if there exists $i_v\neq i_w$ with $\alpha(v)=\alpha(w)$ (that is, $(v,w)$ is a liana), and is $1$ otherwise.
The notation can also be frozen for $x$, $p\in \MM$:
\[
(\F_x(\pi) \triangleright \phi)(p)=\sum_{\underset{v\in V}{i_v=1,\dots,D}}
\delta_{i_{L}} \left(\prod_{v\in V_{\bullet}} E_{I_{\Pi(v)}}[f^{i_v}](x) \right)
E_{I_{R}}[\phi](p).
\]
Exotic Lie S-series and exotic Lie-Butcher series are defined analogously to Definition \ref{def:dec_series}: given a one form $a \in \EF^*$ (resp.\ts $a \in \ET^*$), an exotic Lie S-series (resp.\ts exotic Lie-Butcher series) is the following formal series
\[
  S_h(a) \triangleright \phi := \sum_{\pi \in EF} h^{\abs{\pi}} a(\pi) \F(\pi) \triangleright \phi,\quad
  B_h(a) := \sum_{\tau \in ET} h^{\abs{\tau}} a(\tau) \F(\tau),
\]
and analogously for their frozen counterparts $S_{h,p}$ and $B_{h,p}$.
\end{definition}

The exotic series arise naturally by considering the expectation of decorated series. Let $\Phi : \EF \to \FF_D$ be the map that forgets the pairings of the white leaves and sends each liana to a sum of paired decorated nodes $\sum_{l=1}^L \circ^l \circ^l$. For example,
\[ \Phi(\forestIC) = \sum_{i,j,k = 1}^L \forestJC, \quad \Phi(\forestKC) = \sum_{i,j,k=1}^L \forestLC, \quad \Phi(\forestMC) = \sum_{i,j,k=1}^L \forestNC. \]

\begin{proposition}
\label{proposition:Isserlis}
  Let $a^{\text{dec}}$ be a character on $(\FF_D, \shuffle)$ with $D = \{\bullet, \circ^1, \dots, \circ^L\}$, then
\[\E[S^{\text{dec}}(a^{\text{dec}})\triangleright \phi]=S(a)\triangleright \phi,\]
  where $a = a^{\text{dec}} \circ \Phi$ is a character on $\EF$ with respect to the shuffle product
\end{proposition}

For example, one finds
\[ \E \Big[ \dFdec(\forestOC) \Big] = \dF(\forestPC + \forestQC + \forestRC). \]

\begin{proof}
  The proof is analogous to \cite[Sect.\ts 4.1]{Bronasco22ebs} where the white leaves decorated by the same number are paired using the Isserlis formula \cite{Isserlis18oaf}. That is,
  \[ \E[S^{\text{dec}}(a^{\text{dec}})\triangleright \phi] = \E \Big[ \sum_{\pi \in F_D} h^{|\pi|} a^{\text{dec}}(\pi) \dFdec(\pi) \triangleright \phi \Big] = \sum_{\pi \in F_D} h^{|\pi|} a^{\text{dec}}(\pi) \E [\dFdec(\pi)] \triangleright \phi. \]
  We use Isserlis formula to write
  \[ \E [\dFdec(\pi, \alpha)] = \sum_{\alpha_e} \dF (\pi,\alpha_e), \]
  where $(\pi, \alpha)$ and $(\pi, \alpha_e)$ denote decorated and exotic planar forests as defined in Definitions \ref{def:FD} and \ref{def:EF} with the sum being over all $\alpha_e$ such that $\alpha_e^{-1} (\bullet) = \alpha^{-1} (\bullet)$ and $\alpha_e^{-1}(\N)$ is a union of partitions into pairs of $\alpha^{-1}(\circ^l)$ for $l = 1, \dots,L$. We note that $\E[\dFdec(\pi)] = 0$ if $\pi$ contains an odd number of nodes of color $\circ^l$ for an $l \in \{1, \dots, L \}$. Therefore, we have,
  \[ \E[S^{\text{dec}}(a^{\text{dec}})\triangleright \phi] = \sum_{\pi \in EF} h^{|\pi|} a^{\text{dec}}\big(\Phi(\pi)\big) \dF(\pi) \triangleright \phi = S(a^{\text{dec}} \circ \Phi) \triangleright \phi, \]
  and the statement is proved.
\end{proof}

\begin{ex}
Consider the Euler frozen flow method \eqref{equation:Euler_FF}.
Using Proposition \ref{prop:a_character}, its Taylor expansion in decorated forests is given by
\begin{align*}
\phi(&\exp(hF_p+\sqrt{2h}\xi^d E_d)p)
=\exp^\cdot(hF+\sqrt{2h}\xi^d E_d)\triangleright\phi(p)
=\dFdec(\exp^\cdot(h\forestSC+\sqrt{2h}\forestTC))\triangleright\phi(p)\\
&=\Big(\id
+h^{1/2}\sqrt{2}\xi^d E_d
+h(F+\xi^{d_2}\xi^{d_1} E_{d_2}\cdot E_{d_1})\\
&+h^{3/2}(\frac{1}{\sqrt{2}}\xi^d F\cdot E_d+\frac{1}{\sqrt{2}}\xi^d E_d\cdot F+\frac{\sqrt{2}}{3}\xi^{d_3}\xi^{d_2}\xi^{d_1} E_{d_3}\cdot E_{d_2}\cdot E_{d_1})
+\dots\Big)\triangleright\phi\\
&=\dFdec\Big(\textbf{1}
+h^{1/2}\sqrt{2}\forestUC
+h(\forestVC+\forestWC)
+h^{3/2}(\frac{1}{\sqrt{2}}\forestXC+\frac{1}{\sqrt{2}}\forestYC+\frac{\sqrt{2}}{3}\forestAD)\\
&+h^2 (\frac{1}{2}\forestBD+\frac{1}{3}\forestCD+\frac{1}{3}\forestDD+\frac{1}{3}\forestED+\frac{1}{6}\forestFD)
+\dots
\Big)\triangleright\phi(p).
\end{align*}
Then, the expectation pairs the decorated nodes together as detailed in Proposition \ref{proposition:Isserlis} and yields
\begin{align*}
\E[&\phi(\exp(hF_p+\sqrt{2h}\xi^d E_d)p)]
=\F\Big(\textbf{1}
+h(\forestGD+\forestHD)\\
&+h^2 (\frac{1}{2}\forestID+\frac{1}{3}\forestJD+\frac{1}{3}\forestKD+\frac{1}{3}\forestLD+\frac{1}{6}\forestMD+\frac{1}{6}\forestND+\frac{1}{6}\forestOD)
+\dots
\Big)\triangleright\phi(p)\\
&=\Big(\id
+h(F+E_{d}\cdot E_{d})
+h^2 (\frac{1}{2}F\cdot F+\frac{1}{3}F\cdot E_d\cdot E_d+\frac{1}{3}E_d\cdot F\cdot E_d+\frac{1}{3}E_d\cdot E_d\cdot F\\&+\frac{1}{6}E_{d_2}\cdot E_{d_2}\cdot E_{d_1}\cdot E_{d_1}+\frac{1}{6}E_{d_2}\cdot E_{d_1}\cdot E_{d_2}\cdot E_{d_1}+\frac{1}{6}E_{d_1}\cdot E_{d_2}\cdot E_{d_2}\cdot E_{d_1})
+\dots\Big)\triangleright\phi(p).
\end{align*}
\end{ex}

\subsection{Conversion of geometric operations into exotic Lie series}

We saw in Proposition \ref{prop:a_character} that the frozen flow rewrites as a tree series operation.
The following result, proved anologously to \cite[Lemma 2.3]{Owren06ocf}, provides an algebraic description of the Taylor expansion of a frozen vector field.
\begin{lemma}
Let a smooth map $\varphi\colon\MM\rightarrow \MM$ with expansion $\phi(\varphi(p))=(S(a)\triangleright \phi)(p)$ with $a\in \EF^*$,  then the frozen vector field $F_{\varphi}\colon p\rightarrow F_{\varphi(p)}(p)$ satisfies
\[
F_{\varphi}\triangleright \phi=(S(a)\triangleright hF)\triangleright \phi=B(\tilde{a})\triangleright \phi,\quad \tilde{a}(\tau)=a(B^-(\tau)),\quad \tilde{a}\in \ET^*.
\]
\end{lemma}

Let two smooth maps $\varphi^1$, $\varphi^2\colon \MM \times \MM\rightarrow \MM$, typically $\varphi^k_x(p)=\exp(F_{\psi^k(x)})\cdot p$ with $\psi^k\colon \MM\rightarrow \MM$.
We use two possible ways to compose the $\varphi^k$.
The standard composition of the maps would read
\[
\varphi^1\circ\varphi^2(p)=\varphi^1_{\varphi^2_p (p)}(\varphi^2_p (p)).
\]
On the other hand, the frozen flow methods \eqref{equation:def_CGsto} rely on the frozen composition:
\[
\varphi^1\cdot\varphi^2(p)=\varphi^1_p\cdot\varphi^2_p(p)=\varphi^1_p (\varphi^2_p (p)).
\]
The decorated and exotic Lie series allow to represent naturally all the geometric operations needed for the analysis of the frozen-flow methods \eqref{equation:def_CGsto}. We present the following result in the context of exotic series (with $h=1$ for simplicity), and we mention that decorated series satisfy analogous identities. Let the deconcatenation product $\Delta : \EF \to \EF \otimes \EF$ be defined as
\[ \Delta_\cdot(\pi_{1, \dots, n}) = \sum_{k=0}^n \pi_{1,\dots,k} \otimes \pi_{k+1,\dots,n}, \quad \text{with } \pi_i \in \Irr, \]
where $\pi_{1,\dots,n} = \pi_1 \cdots \pi_n$, $\pi_\emptyset = \mathbf{1}$, and $\Irr$ is the space of irreducible exotic forests introduced in Section \ref{sec:exotic_forests}.
The composition and the Munthe-Kaas-Wright coproduct are further discussed in Section \ref{section:decorated_exotic_forests}.
\begin{theorem}
\label{proposition:Relations_B_S_series}
Let two smooth maps $\varphi^1$, $\varphi^2\colon \MM\times \MM\rightarrow \MM$ with expansions $\phi(\varphi^k_x(p))=(S_{x}(a^k)\triangleright \phi)(p)$ with $a\in \EF^*$.
The frozen composition is given by
\[
\phi(\varphi^1\cdot\varphi^2)=(S(a^2)\cdot S(a^1))\triangleright \phi
=S(a^2\cdot a^1)\triangleright \phi,\quad a^2\cdot a^1=\mu\circ (a^2\otimes a^1)\circ \Delta_\cdot.
\]

Let two smooth maps $\varphi^1$, $\varphi^2\colon\MM\rightarrow \MM$ with expansions $\phi(\varphi^k(p))=(S(a^k)\triangleright \phi)(p)$ with $a^k\in \EF^*$.
The composition of exotic series as differential operators is given by the composition law, dual of the MKW coproduct,
\[
\phi(\varphi^1\circ \varphi^2)=S(a^2)\triangleright(S(a^1)\triangleright\phi) = S(a^2 * a^1)\triangleright\phi,\quad a^2 * a^1=\mu\circ(a^2\otimes a^1)\circ \Delta_{MKW}.
\]
\end{theorem}

\begin{proof}
The expression for the frozen composition is derived similarly to \cite[Lemma 2.2]{Owren06ocf}.
Both compositions follow from Proposition \ref{prop:F_morphism} which states that $\F$ is a morphism of algebras from $(\EF, \cdot)$ to $(T(\mathfrak{X}(\MM)), \cdot)$ and from $(\EF, \gl)$ to $(T(\mathfrak{X}(\MM)), *)$. In particular, apply Taylor expansion as
    \begin{align*}
        \phi (\varphi^1 \circ \varphi^2) &= S(a^2) \triangleright (S(a^1) \triangleright \phi) \,,
        \intertext{use the definition of S-series,}
                                          &= \sum_{\pi^1, \pi^2 \in EF} h^{|\pi^1| + |\pi^2|} a^2(\pi^2) a^1(\pi^1) \F(\pi^2) \triangleright (\F(\pi^1) \triangleright \phi) \,,
                                          \intertext{apply Proposition \ref{prop:F_morphism} to rewrite the composition of flows as a product of exotic forests,}
                                          &= \sum_{\pi^1, \pi^2 \in EF} h^{|\pi^1| + |\pi^2|} a^2(\pi^2) a^1(\pi^1) \big( \F(\pi^2) * \F(\pi^1) \big) \triangleright \phi \,, \\
                                          &= \sum_{\pi^1, \pi^2 \in EF} h^{|\pi^1| + |\pi^2|} a^2(\pi^2) a^1(\pi^1) \big( \F(\pi^2 \gl \pi^1) \big) \triangleright \phi \,,
                                          \intertext{use the duality between Grossman-Larson product and the Munthe-Kaas-Wright coproduct shown in Theorem \ref{theorem_MKW},}
                                          &= S(a^2 * a^1) \triangleright \phi \,.
    \end{align*}
    The frozen composition is derived analogously.
\end{proof}

\subsection{Exotic Lie series of the exact and numerical flows}
\label{section:Butcher_expansion_flows}

The exotic Lie series formalism allows us to give explicit expressions of the Taylor expansions of both exact and numerical flows.
The generator of equation \eqref{equation:SDE_Strato} rewrites as
\begin{align*}
\LL[\phi]
=\sum_{d=1}^D f^{d}E_d[\phi]
+\sum_{d=1}^D E_{d}[E_{d}[\phi]]
=\F(\forestPD+\forestQD)[\phi]
=S(l)[\phi], \quad l=\delta_{\forestRD}+\delta_{\forestSD}.
\end{align*}
The formal expansion \eqref{equation:dvp_exact} then rewrites straightforwardly in terms of forests.
The proof is a generalisation of the works \cite{Owren99rkm, Rossler04ste, Bronasco22ebs}.
\begin{theorem}
\label{theorem:expansion_exact_flow}
The Taylor expansion of the exact flow of \eqref{equation:SDE_Strato} is given by the exotic Lie series
\[u(h,p)=\exp(h\LL)[\phi](p)=S_h(e)\triangleright \phi(p), \quad e=\exp^*(l)=\sum_{n=0}^\infty \frac{l^{*n}}{n!}.\]
In addition, let $\alpha(\pi)$ be the number of different labelings $l: V(\pi) \to \{1, \dots, |\pi|\}$ such that the numbered leaves in a liana share the same label and, given $v, u \in V(\pi)$ we have $l(v) > l(u)$ if there exists an edge from $v$ to $u$ or if $v$ and $u$ are attached to the same vertex $w \in V(\pi)$ and $v$ is to the left of $u$. Then, the coefficient $e$ is given by
\[
e(\pi)=\frac{\alpha(\pi)}{\abs{\pi}!}.
\]
\end{theorem}

    \begin{proof}
        We use Proposition \ref{prop:Taylor_expansion_exact} to write $u(h,p)=\exp(h\LL)[\phi](p)$. Using the identity $\LL[\phi] = S(l)[\phi] = S(l) \triangleright \phi$ we write,
        \[ u(h,p)=\exp\big(hS(l) \triangleright \blank \big) \phi (p) \,. \]
        Applying Theorem \ref{proposition:Relations_B_S_series}, we obtain,
        \[ u(h,p)=S_h\big(\exp^*(l) \big) \triangleright \phi (p) \,, \]
        and define $e := \exp^*(l)$. The value of $e(\pi) = \frac{l^{*|\pi|}(\pi)}{|\pi|!}$ can be computed by using the definition of MKW coproduct from Theorem \ref{theorem_MKW} and the fact that $l=\delta_{\forestTD}+\delta_{\forestUD}$.
    \end{proof}

The first terms are the following
\begin{align*}
\exp(t\LL)
&=\F\Big(\ind 
+t(\forestVD+\forestWD)
+\frac{t^2}{2}\Big(
\forestXD
+\forestYD
+\forestAE
+2\forestBE
+\forestCE
+\forestDE
+\forestEE
\Big)
+\dots\Big),
\end{align*}
where the second order term is $\frac{t^2}{2}\LL^2=\frac{t^2}{2}S_t(l)*S_t(l)=\frac{t^2}{2}S_t(l*l)$.


Let us now present the expansion in exotic series of the frozen flow integrators. For the sake of clarity, the random Runge-Kutta coefficients $z^d_k$, $z^d_{i,k}$ are chosen as centered Gaussians. They can be replaced with bounded approximations of Gaussians with the same first moments up to order $2p$ and the expansion of the integrator will be an exotic series up to order $p$.


\begin{definition}
For $K>0$, a labelling in $\Lab_K(\pi)$ of a decorated forest $\pi$ is a multi-index map $\textbf{k}\colon V\rightarrow \{1,\dots,K\}$ such that $k_v\geq k_w$ if $v$ and $w$ are in the same tree, with same height and $v$ is on the right of $w$.
\end{definition}

For instance, a possible labelling in $\Lab_8(\forestFE)$ is $\forestGE$.
Then, the Talay-Tubaro expansion of the new methods \eqref{equation:def_CGsto} is also given by an exotic Lie series.
\begin{theorem}
\label{theorem:expansion_numerical_flow}
The Talay-Tubaro expansion of a frozen flow method of the form \eqref{equation:def_CGsto} is an exotic Lie series:
\[\E[\phi(X_1)|X_0=p]=S_h(a)\triangleright \phi(p),\]
where $a$ is given by
\[
a(\pi)=\E\Big[\sum_{\underset{v\in V}{i_v=1,\dots,s}}
\sum_{\textbf{k}\in \Lab_K(\pi)}^! \sum_{\underset{v\in V_{\circ}}{d_v=1,\dots,D}}
\delta_{d_{L}}
\prod_{\underset{v\in V_{\bullet}}{(v,w)\in E}} Z^{0}_{i_w,i_v,k_v}
\prod_{\underset{v\in V_{\circ}}{(v,w)\in E}} Z^{d_v}_{i_w,k_v}
\prod_{r\in R\cap V_{\bullet}} z^{0}_{i_r,k_r}
\prod_{r\in R\cap V_{\circ}} z^{d_r}_{k_r}  
\Big],
\]
where $\delta_{d_{L}}=0$ if there exists $d_v\neq d_w$ with $(v,w)\in L$, and is 1 else.
Moreover, the coefficient map $a$ is a character on $(\EF,\shuffle)$, that is, $a(\pi_1\shuffle \pi_2)=a(\pi_1)a(\pi_2)$.
\end{theorem}

The proof is a stochastic generalisation of \cite{Owren06ocf}.
\begin{proof}
Let us consider one step of the algorithm, with one Gaussian for the sake of simplicity: $Z^{d}_{i,k}=\hat{Z}_{i,k}\xi^d$, $z^{d}_{k}=\hat{z}_{k}\xi^d$.
Let the intermediate steps
\[
g^i_{k}=\exp\bigg(h\sum_{j=1}^s Z^{0}_{i,j,k} f(H^j)+ \sqrt{h} \hat{Z}_{i,k} \sum_d \xi_d E_{d} \bigg)\circ g^i_{k-1},\quad
g^i_0=\id.
\]
Theorem \ref{proposition:Relations_B_S_series} yields $\phi(g^i_{k})=S_h^{\text{dec}}(\alpha_k^i)\triangleright \phi$ and $\phi(H^i)=S_h^{\text{dec}}(\alpha_K^i)\triangleright \phi$ with $\alpha_0^i=\delta_{\textbf{1}}$, $\alpha_k^i(1)=1$, and
\begin{align*}
\alpha_k^i(\tau_1\dots \tau_n)&=\sum_{p=0}^n \frac{1}{(n-p)!} \alpha_{k-1}^i(\tau_1\dots \tau_p) \Big(\sum_{j_{p+1}=1}^s Z^{0}_{i,j_{p+1},k}  \alpha_K^{j_{p+1}}(B^-\tau_{p+1})\ind_{\tau_{p+1}\neq \forestHE}+ \hat{Z}_{i,k}\ind_{\tau_{p+1}=\forestIE}\Big)\\&\dots \Big(\sum_{j_n=1}^s Z^{0}_{i,j_n,k}  \alpha_K^{j_n}(B^-\tau_n)\ind_{\tau_n\neq\forestJE}+ \hat{Z}_{i,k}\ind_{\tau_n=\forestKE}\Big).
\end{align*}
An induction then yields
\begin{align*}
\alpha_K^i(\tau_1\dots \tau_n)
&=\sum_{k_1\leq\dots \leq k_n} \frac{1}{\textbf{k}!} \Big(\sum_{j_1=1}^s Z^{0}_{i,j_1,k_1}  \alpha_K^{j_1}(B^-\tau_1)\ind_{\tau_1\neq \forestLE}+ \hat{Z}_{i,k_1}\ind_{\tau_1=\forestME}\Big)\\&
\dots
\Big(\sum_{j_n=1}^s Z^{0}_{i,j_n,k_n}  \alpha_K^{j_n}(B^-\tau_n)\ind_{\tau_n\neq \forestNE}+ \hat{Z}_{i,k_n}\ind_{\tau_n=\forestOE}\Big).
\end{align*}
Similar reasoning yields the desired expansion for the last stage of the method as $S_h^{\text{dec}}(a^{\text{dec}})$, where $a^{\text{dec}}$ is a character by construction.
Proposition \ref{proposition:Isserlis} yields the expected formula for the coefficient map $a$, as well as the character property.
%
\end{proof}

Examples of computations of the coefficient map $a$ associated to frozen flow methods are presented in Table \ref{table:weak_order_conditions} and Appendix \ref{app:order_cond_3}.
A more involved example of order 7 is the following:
\begin{align*}
\pi&=\forestPE\\
a(\pi)&=\E\sum_{i_1,\dots, i_9=1}^s \sum_{\underset{k_4\geq k_5,k_6\geq k_7\geq k_8}{k_1\geq k_2\geq k_3}}^! \sum_{d_1, d_2=1}^D
Z_{i_8,i_9,k_9}^{d_2}
Z_{i_3,i_8,k_8}^0 Z_{i_3,i_7,k_7}^{d_1} Z_{i_3,i_6,k_6}^{d_1} Z_{i_1,i_5,k_5}^0 Z_{i_1,i_4,k_4}^0
z_{i_3,k_3}^0 z_{k_2}^{d_2} z_{i_1,k_1}^0
\end{align*}

\begin{remark}
  \label{rmk:shuffle_relations}

  Order conditions for order $p$ are obtained by requiring $(a - e)(\pi) = 0$ for all $\pi \in EF \setminus \{\mathbf{1}\}$ of order $|\pi| \leq p$. This results in the set of order conditions being indexed by the elements of $EF \setminus \{ \mathbf{1} \}$.
  The fact that the coefficient maps $a$ and $e$ are characters of $(\EF, \shuffle)$ induces relations between order conditions as detailed in Remark \ref{remark:shuffle_identities}. In particular, $(a - e) (\pi \shuffle \eta) = 0$ if $(a - e)(\pi) = 0$ and $(a - e)(\eta) = 0$ using,
  \[ (a - e)(\pi \shuffle \eta) = a(\pi)a(\eta) - e(\pi)e(\eta) = 0. \]
  To reduce the number of superfluous order conditions, we consider the quotient of $\EF$ by the ideal
  $\II := \langle \pi \shuffle \eta \; : \; \pi, \eta \in \EF \setminus \{ \mathbf{1} \} \rangle$.
  Using the fact that the order conditions corresponding to $\pi \shuffle \eta$ are satisfied automatically, we take $a - e$ to be a functional over $\EF /_\II$, that is $a - e \in (\EF /_\II)^*$.
  Dualizing $\EF/_\II$ yields the space of primitive elements $\Prim$ of $\EF$ with respect to the deshuffle coproduct $\Delta_\shuffle$ defined as
\[ \Prim := \Span\{ \pi \in \EF \; : \; \Delta_\shuffle (\pi) = \pi \otimes \mathbf{1} + \mathbf{1} \otimes \pi \}. \]
  Therefore, using the identification $(a - e)(\pi) = \langle a - e, \pi \rangle$, we have $a - e \in \Prim$. The detailed description of $\Prim$ is left for future work.
\end{remark}

\section{Exotic planar forests and their algebraic structure}
\label{section:decorated_exotic_forests}

In this section, we focus on the algebraic structures of exotic planar forests. We first recall the well-known D-algebra structure of the connection algebra of vector fields in Section \ref{sec:post_Hopf_algebra} and decorated planar forests in Section \ref{sec:decorated_forests}. We then explore how these structures change in the context of exotic planar forests in Section \ref{sec:exotic_forests}.
In particular, we show that exotic planar forests form both the Grossman-Larson and Munthe-Kaas-Wright Hopf algebras. We also observe that the D-algebra structure is preserved.
However, exotic planar forests, unlike decorated planar forests, do not form a post-Hopf algebra, since a connected exotic forest can have multiple roots, which violates the Leibniz rule for the grafting product.

\subsection{D-algebra of vector fields}
\label{sec:post_Hopf_algebra}

We recall the connection algebra $\mathfrak{X}(\MM)$ introduced in Section \ref{section:connection_algebra} with the connection $\triangleright$.
The deshuffle coproduct $\Delta_\shuffle\colon T(\mathfrak{X}(\MM))\rightarrow T(\mathfrak{X}(\MM))\otimes T(\mathfrak{X}(\MM))$ is, for $A$, $B\in T(\mathfrak{X}(\MM))$, $F\in \mathfrak{X}(\MM)$,
\[ \Delta_\shuffle (AB) = \Delta_\shuffle A \cdot \Delta_\shuffle B, \quad \Delta_\shuffle(F)=\textbf{1}\otimes F+F\otimes \textbf{1}, \]
where we recall $(A \otimes B) \cdot (C \otimes D) = (A \cdot C) \otimes (B \cdot D)$.
Following the extension \cite{Ebrahimi15otl} of the so-called Guin-Oudom process \cite{Oudom08otl}, the product $\triangleright$ extends to $T(\mathfrak{X}(\MM))$, for $A$, $B$, $C\in T(\mathfrak{X}(\MM))$, $F$, $G\in \mathfrak{X}(\MM)$, by
\begin{align*}
  \textbf{1} \triangleright A&=A, \quad A \triangleright \textbf{1} = 0,\\
(F\cdot A)\triangleright G&=F\triangleright (A\triangleright G)-(F\triangleright A)\triangleright G,\\
A\triangleright (B\cdot C)&=(A_{(1)}\triangleright B)\cdot(A_{(2)}\triangleright C),
\end{align*}
where we use the Sweedler notation w.r.t.\ts $\Delta_\shuffle$. We note that the extension of the Guin-Oudom process does not require $(\mathfrak{X}, -T, \triangleright)$ to have a post-Lie structure. The triple $(T(\mathfrak{X}(\MM)), \cdot, \triangleright)$ is called a D-algebra \cite{MuntheKaas08oth}.
That is, $T(\mathfrak{X}(\MM))$ is a graded associative algebra and for all $F \in \mathfrak{X}(\MM)$ and $A, B \in T(\mathfrak{X}(\MM))$, we have, $A \triangleright F \in \mathfrak{X}(\MM)$, as well as,
\begin{align*}
  F \triangleright (A \cdot B) &= (F \triangleright A) \cdot B + A \cdot (F \triangleright B), \\
  F \triangleright (A \triangleright B) &= (F \triangleright A) \triangleright B + (F \cdot A) \triangleright B.
\end{align*}

The tensor algebra is equipped with the associative non-commutative Grossman-Larson product
\[
A*B=A_{(1)}\cdot(A_{(2)}\triangleright B),\quad (A*B)\triangleright C=A\triangleright(B\triangleright C).
\]
Then, $(T(\mathfrak{X}(\MM)),\cdot, \Delta_\shuffle)$ and $(T(\mathfrak{X}(\MM)),*, \Delta_\shuffle)$ are Hopf algebras.
By the Cartier-Quillen-Milnor-Moore theorem, the Hopf algebra $(T(\mathfrak{X}(\MM)),*, \Delta_\shuffle)$ coincides with the universal enveloping algebra $\UU(\mathfrak{X}(\MM))$ of $(\mathfrak{X}(\MM),\llbracket \blank,\blank \rrbracket)$.
The Grossman-Larson product satisfies, in particular, for $F$, $G\in \mathfrak{X}(\MM)$:
\[
\llbracket F,G \rrbracket\triangleright \phi=(F*G-G*F)\triangleright \phi.
\]

\subsection{Decorated planar forests}
\label{sec:decorated_forests}

We recall the discussion of the decorated planar forests in Section \ref{section:decorated_forests}. We present an alternative definition of the decorated planar forests in Definition \ref{def:FD_alt}.

\begin{definition}
  \label{def:FD_alt}
  Let $\tau_1, \dots, \tau_n \in T_D$ be decorated planar trees, then, a \textit{decorated planar forest} $\pi \in F_D$ is a concatenation $\tau_1 \cdots \tau_n$. Let $\pi \in F_D$ be a decorated planar forest, then, a \textit{decorated planar tree} $\tau \in T_D$ is defined as $B^+_d (\pi)$ where the map $B^+_d : F_D \to T_D$ produces a tree with the root decorated by $d$ and branches $\pi$.
\end{definition}

We define the left grafting $\graft : \TT_D \otimes \TT_D \to \TT_D$, for $\tau, \gamma_1, \dots, \gamma_n \in T_D$, as
\[ \tau \graft B^+_d(\mathbf{1}) := B^+_d(\tau), \quad \tau \graft B_d^+(\gamma_1 \cdots \gamma_n) := B_d^+ \big( \tau \cdot \gamma_1 \cdots \gamma_n + \sum_{i=1}^n \gamma_1 \cdots (\tau \graft \gamma_i) \cdots \gamma_n \big). \]
It is extended to $\graft : \FF_D \otimes \FF_D \to \FF_D$, for $\tau, \gamma \in T_D, \pi, \eta, \mu \in F_D$, by Guin-Oudom process,
\begin{align*}
    \mathbf{1} \graft \pi &:= \pi, \quad \pi \graft \mathbf{1} := 0, \\
    (\tau \cdot \pi) \graft \gamma &:= \tau \graft (\pi \graft \gamma) - (\tau \graft \pi) \graft \gamma, \\
    \pi \graft (\eta \cdot \mu) &:= (\pi_{(1)} \graft \eta) \cdot (\pi_{(2)} \graft \mu).
\end{align*}
We note that $(\FF_D, \cdot, \graft)$ forms the free D-algebra over the set of decorations $D$, \cite{MuntheKaas08oth}.
The universal property of $(\FF_D, \cdot, \graft)$ allows us to define a D-algebra homomorphism $\dFdec : \FF_D \to \TT(\mathfrak{X}(\MM))$ by $\dFdec(\forestQE) = f^d$ satisfying
\[ \dFdec (\pi_1 \cdot \pi_2) = \dFdec(\pi_1) \dFdec(\pi_2), \quad \dFdec(\pi_1 \graft \pi_2) = \dFdec(\pi_1) \triangleright \dFdec(\pi_2). \]

\begin{remark}
  \label{rmk:universal_enveloping_F_D}
    The space of decorated planar forests forms the tensor algebra $T(\TT_D)$ of the space of decorated planar trees. It is the universal enveloping Lie algebra $\FF_D = \UU(\Lie(\TT_D))$ of the post-Lie algebra $\Lie(\TT_D)$ with the Lie bracket $[x, y] = x \cdot y - y \cdot x$ for $x, y \in \Lie(\TT_D)$. Moreover, $(\FF_D, \cdot, \Delta_\shuffle, \graft)$ is a post-Hopf algebra \cite{Li23pha} and $\Lie(\TT_D)$ is the space of primitive elements with respect to the deshuffle coproduct $\Delta_\shuffle$, that is,
  \[ \Lie(\TT_D) = \{ \pi \in \FF_D \; : \; \Delta_\shuffle (\pi) = \pi \otimes \mathbf{1} + \mathbf{1} \otimes \pi \}. \]
    We emphasize that $\dFdec : \FF_D \to T(\mathfrak{X}(\MM))$ does not restrict to a Lie algebra morphism $Lie(\TT_D) \to \mathfrak{X}(\MM)$.
\end{remark}

The Grossman-Larson product over decorated planar forests is defined as
\[ \pi \gl \eta := \pi_{(1)} \cdot (\pi_{(2)} \graft \eta), \quad \text{for } \pi, \eta \in \FF_D, \]
and $\dFdec(\pi \gl \eta) = \dFdec(\pi) * \dFdec(\eta)$ due to the fact that $\dFdec$ is a D-algebra homomorphism. We generalize the concatenation and Grossman-Larson products by specifying the order in which the trees of the operands are concatenated.

Let $\pi, \eta \in \FF_D$ and let $\omega$ denote an admissible ordering of the roots of $\pi \cdot \eta$. An ordering $\omega$ is \emph{admissible} if the order of roots in $\pi$ and $\eta$ is preserved, that is, if $r_1$ and $r_2$ are roots of $\pi$ or $\eta$ and $r_1 \leq r_2$, then $r_1 \leq r_2$ according to $\omega$. We define $\cdot_\omega : \FF_D \times \FF_D \to \FF_D$ to be the concatenation product with roots arranged according to the ordering $\omega$. The Grossman-Larson product $\gl_\omega : \FF_D \times \FF_D \to \FF_D$ is defined as follows,
\[ \pi \gl_\omega \eta := \pi_{(1)} \cdot_\omega (\pi_{(2)} \graft \eta), \]
where the ordering $\omega$ is defined over the roots of $\pi \cdot \eta$.
For example
\begin{align*}
  \forestRE \cdot_\omega \forestSE &= \forestTE, \quad \text{for } \omega : c \leq a \leq b \leq d, \\
  \forestUE \gl_\omega \forestVE &= \forestWE + (\forestXE \graft \forestYE) \forestAF + \forestBF (\forestCF \graft \forestDF) \\
  &\quad + (\forestEF \graft \forestFF) \forestGF + \forestHF (\forestIF \graft \forestJF) + \forestKF \graft \forestLF.
\end{align*}
We note that $\pi \cdot_\omega \eta = \pi \cdot \eta$ and $\pi \gl_\omega \eta = \pi \gl \eta$ for $\omega$ which takes the roots of $\pi$ to be smaller than the roots of $\eta$. If the admissible ordering $\omega$ in $\pi \cdot_\omega \eta$ does not specify the order between some roots in $\pi \cdot \eta$, then, the roots of $\pi$ are taken to be smaller than the roots of $\eta$.
The generalised concatenation and Grossman-Larson products are used in the context of planar exotic forests to describe the structure of the Grossman-Larson Hopf algebra.

\subsection{Exotic planar forests}
\label{sec:exotic_forests}

The algebraic structures described in this section are applied in Section \ref{section:exotic_Lie_series} for the description of the composition law and the order conditions of stochastic frozen flow methods.

We construct the space $\EF$ from the space of decorated planar forest $\FF_D$ and inherit a number of algebraic structures from it. There are two possible approaches: with an appropriate subspace or with a quotient. We use the first approach to define the algebras and the second approach to define coalgebras. 

Let us consider decorated forests $\FF_D$ with $D = \{\forestMF, \forestNF \; : \; d \in \N \}$ and let $\FF_D^L$ denote its quotient by the subspace $\KK$ spanned by the forests in which a numbered node is found at a position which is not a leaf. It can be checked that $\KK$ is an ideal with respect to concatenation, shuffle, grafing, and Grossman-Larson products, and, therefore, $\FF^L_D$ forms an algebra when endowed with one of these products.
Let $\overline{\FF}^L_D$ denote its completion with respect to the grading given by the number of nodes, that is, $\overline{\FF}^L_D$ is a space of formal sums over $F_D^L$. We define the subspace $\FF_D^\EE$ of $\overline{\FF}_D^L$, corresponding to exotic forests, spanned by
\[  \varphi(\pi, \alpha_e) := \sum_{\alpha \in P(\alpha_e)} (\pi, \alpha),\]
with~$\varphi : \EF \to \FF_D^\EE$ an isomorphism.~$P(\alpha_e)$ is the set of decorations~$\alpha$ with~$\alpha^{-1}(\bullet) = \alpha_e^{-1}(\bullet)$ and 
\[ \alpha(v_1) = \alpha(v_2) \quad \text{if } \quad \alpha_e(v_1) = \alpha_e(v_2), \quad \text{for } v_1, v_2 \in V(\pi). \]
Some examples of the values of $\varphi$ are
\[ \varphi(\forestOF) = \sum_{i = 1}^\infty \forestPF, \quad \varphi(\forestQF) = \sum_{i, j = 1}^\infty \forestRF. \]

We take~$\varphi : \EF \to \FF_D^\EE$ to be an algebra homomorphism with respect to the generalised concatenation, shuffle, grafting, and generalised Grossman-Larson products which induces the corresponding algebraic structures over $\EF$.
For instance, the shuffle product $\shuffle : \EF\times \EF\rightarrow\EF$ permutes the roots of forests. Let $\pi \cdot_\omega \eta$ be the generalised concatenation obtained by concatenating the forests $\pi$ and $\eta$ with the roots permuted according to the admissible ordering $\omega$ of the roots of $\pi \cdot \eta$. The shuffle product is defined as
\[  \pi \shuffle \eta = \sum_\omega \pi \cdot_\omega \eta,\]
where the sum is over admissible orderings.
For example, we find
\begin{align*}
    \forestSF \shuffle \forestTF &= \forestUF + \forestVF + \forestWF, \\
    \forestXF \shuffle \forestYF &= \forestAG + \forestBG + \forestCG + \forestDG + \forestEG + \forestFG,
\end{align*}
with $\mathbf{1}$ being the neutral element for $\shuffle$.

Let us now consider the subspace of $\FF_D^L$ orthogonal to $\FF_D^\EE$ with respect to the inner product with which $F_D^L$ forms an orthonormal basis of $\FF_D^L$. We denote it by $\II$ and we can see that it is a coideal in the coalgebra $(\FF_D^L, \Delta_\shuffle)$ using the fact that the deshuffle coproduct is adjoint to the shuffle product. Taking a quotient of $(\FF_D^L, \Delta_\shuffle)$ by $\II$ gives a definition of the coalgebra of exotic forests $(\EF, \Delta_\shuffle)$ through the isomorphism
\begin{equation}
\label{eq:EF_codef}
  \psi (\pi, \alpha_e) := (\pi, \alpha_e) + \II \quad \in \FF^L_D /_{\II}.
\end{equation}
We refer to \cite{Bronasco22cef} for a precise analogous definition with non-planar exotic forests.

An exotic forests $\pi$ is \emph{irreducible} if it cannot be written as a concatenation of two exotic forests, that is, $\pi \neq \gamma \cdot \eta$ for $\gamma, \eta \in EF \setminus \{\mathbf{1}\}$. The space of irreducible exotic forests is denoted by $\Irr \subset \EF$.
An exotic forests $\pi$ is \emph{connected} if it cannot be written as $\pi = \gamma \cdot_\omega \eta$ for any $\gamma, \eta \in EF \setminus \{\mathbf{1}\}$ and any ordering $\omega$ of the roots of $\gamma \cdot \eta$. The set of connected exotic forests is denoted by $EF_C$ and the corresponding space by $\Conn := \Span EF_C$. We note that $\Conn \subset \Irr$.

\begin{remark}
Analogously to Remark \ref{rmk:universal_enveloping_F_D}, the space of planar exotic forests forms the tensor algebra $T(\Irr)$ of $\Irr$, and, thus, a universal enveloping Lie algebra $\EF = \UU (\Lie(\Irr))$ with the Lie bracket $[x,y]=x\cdot y-y\cdot x$ on $\EF$.
The first terms are
\[
\Lie(\Irr)=\Span\{\forestGG,\forestHG,\forestIG, \forestJG,\forestKG,\forestLG,\forestMG,\forestNG,\forestOG,
[\forestPG,\forestQG],\dots
\}.
\]
\end{remark}

In order to show that the space of exotic planar forests forms a Hopf algebra, we need to prove the compatibility between the algebraic and coalgebraic structures. We consider the product $\tilde\graft_\omega : \EF \otimes \EF \to \EF$ defined as
\begin{align*}
  \pi \tilde\graft_\omega \eta &= \pi \gl_\omega \eta - \pi \cdot_\omega \eta, \quad \text{for } \pi \in EF_C, \eta \in EF, \\
  (\pi \cdot_\omega \gamma) \tilde\graft_\omega \eta &= \pi \tilde\graft_\omega (\gamma \tilde\graft_\omega \eta) - (\pi \tilde\graft_\omega \gamma) \tilde\graft_\omega \eta, \quad \text{for } \gamma \in EF.
\end{align*}
Intuitively, the product $\tilde\graft_\omega$ is a modification of the generalised Grossman-Larson product $\gl_\omega$ in which we require each element of $EF_C$ in the left operand to attach at least one root to a node of the right operand. For example, let $\omega: a \leq c \leq b$, then,
\[ \forestRG \tilde\graft_\omega \forestSG = \forestTG + \forestUG + \forestVG. \]
We note the following property
\[  \pi \gl_\omega \eta = \pi_{(1)} \cdot_\omega (\pi_{(2)} \tilde\graft_\omega \eta), \quad \text{for } \pi, \eta \in \EF,\]
where $\Delta_\shuffle(\pi) = \pi_{(1)} \otimes \pi_{(2)}$ is the deshuffle coproduct defined on $\EF$. We note that $\pi \gl \eta = \pi \gl_\omega \eta$ with $\omega$ being the order for which $\pi \cdot_\omega \eta = \pi \cdot \eta$.
We emphasize that the product $\gl_\omega$ on $\EF$ is associative, since $(\EF, \gl_\omega)$ is a subalgebra of $(\overline{\FF}^L_D, \gl_\omega)$, where $\gl_\omega$ is associative. This follows from the fact that $\gl_\omega$ can be defined as $\gl$, with the roots arranged according to $\omega$ in the final step.

\begin{theorem}
  \label{thm:EF_GL_Hopf}
  The space of exotic forests $\EF$ together with Grossman-Larson product and deshuffle coproduct forms a Hopf algebra.
\end{theorem}
\begin{proof}
    We note that the algebra structure is obtained by considering a subalgebra of $(\overline{\FF}^L_D, \gl)$, while the coalgebra structure is obtained by taking a quotient of the coalgebra $(\FF^L_D, \Delta_\shuffle)$. This means that the algebra and coalgebra structures are not automatically compatible. We prove the compatibility condition of the generalised Grossman-Larson product for any ordering $\omega$,
    \[ \Delta_\shuffle (\pi \gl_\omega \eta) = \Delta_\shuffle (\pi) \gl_\omega \Delta_\shuffle (\eta), \]
    by induction on the number of connected components in $\pi$. Assume the compatibility is proven for all forests with the number of components less or equal to the number of components in $\pi$. Consider $\tilde\pi \cdot_\omega \pi$ with $\tilde\pi \in EF_C$ and ordering $\omega$ of roots of $\tilde\pi \cdot \pi \cdot \eta$ with an arbitrary order between the roots of $\tilde\pi \cdot \pi$ and $\eta$, then,
    \begin{align*}
      \Delta_\shuffle ((\tilde\pi \cdot_\omega \pi) \gl_\omega \eta) &= \Delta_\shuffle (\tilde\pi \gl_\omega (\pi \gl_\omega \eta) - (\tilde\pi \tildegraft_\omega \pi) \gl_\omega \eta) \\
        &= \Delta_\shuffle (\tilde\pi) \gl_\omega \Delta_\shuffle(\pi) \gl_\omega \Delta_\shuffle(\eta) - \Delta_\shuffle (\tilde\pi \tildegraft_\omega \pi) \gl_\omega \Delta_\shuffle (\eta) \\
        &= \Delta_\shuffle (\tilde\pi \gl_\omega \pi - \tilde\pi \tildegraft_\omega \pi) \gl_\omega \Delta_\shuffle(\eta) = \Delta_\shuffle (\tilde\pi \cdot_\omega \pi) \gl_\omega \Delta_\shuffle(\eta),
    \end{align*}
    where we use the associativity of $\gl_\omega$ and coassociativity of $\Delta_\shuffle$. It remains to show that 
    \[ \Delta_\shuffle (\tilde\pi \gl_\omega \eta) = \Delta_\shuffle (\tilde\pi) \gl_\omega \Delta_\shuffle (\eta), \quad \text{for } \tilde\pi \in \Conn, \eta \in \EF. \]
    We note that $\tilde\pi$ has multiple roots which, when grafted onto different connected components of $\eta$, connect them into a single component. Therefore, we have
    \[ \Delta_\shuffle (\tilde\pi \gl_\omega \eta) = (\tilde\pi \gl_\omega \eta_{(1)}) \otimes \eta_{(2)} + \eta_{(1)} \otimes (\tilde\pi \gl_\omega \eta_{(2)}) = \Delta_\shuffle (\tilde\pi) \gl_\omega \Delta_\shuffle (\eta). \]
  This proves the compatibility between the Grossman-Larson product and deshuffle coproduct. Therefore, $\EF$ forms a graded connected bialgebra, that is, a Hopf algebra.
\end{proof}

\begin{remark}
  The space $(\EF, \cdot, \Delta_\shuffle, \graft)$ is not a post-Hopf algebra, as defined in \cite{Li23pha}, since in general
\[\pi \graft (\pi_1 \cdot \pi_2) \neq (\pi_{(1)} \graft \pi_1) \cdot (\pi_{(2)} \graft \pi_2) \text{ in } \EF. \]
  Moreover, $(\EF, \cdot, \Delta_\shuffle, \tilde\graft)$ is not a post-Hopf algebra for the same reason. This is due to the fact that a connected planar exotic forest can have multiple roots. However, we see that $(\EF, \cdot, \graft)$ is a D-subalgebra of $(\overline{ \FF }^L_D, \cdot, \graft)$, as defined in \cite{MuntheKaas08oth}, with the grading of the associative algebra $(\EF, \cdot)$ given by the number of roots. A straightforward extension of the arguments in \cite{Oudom08otl} fails to show that the spaces $(\EF, \graft)$ and $(\EF, \tilde\graft)$ are either brace or symmetric brace algebras for the same reason.
\end{remark}


The Connes-Kreimer Hopf algebra structure is the dual of the Grossman-Larson one in a Euclidean setting. On homogeneous spaces, it generalises into the Munthe-Kaas-Wright Hopf algebra \cite{MuntheKaas08oth, Ebrahimi24aso}.

Let an admissible cut $c$ of an exotic planar tree $\tau$ be a subset of edges of $\tau$ such that it contains at most one edge from each path in $\tau$ from a leaf to the root. If $e = (u, v)$ is in $c$, then all incoming edges of $v$ which are to the left of $e$ are also in $c$. Moreover, given a liana with nodes $v_1$ and $v_2$, if any edge in the path from $v_1$ to the root is in $c$, then the path from $v_2$ to the root must also contain an edge in $c$.

Given an admissible cut $c$, $\tau \setminus c$ is an exotic planar forest. Let $R^c(\tau)$ be the component of $\tau \setminus c$ which contains the root of $\tau$. Let $P^c(\tau)$ be the linear combination of exotic planar forests obtained by taking the components of $\tau \setminus c$ that do not contain the root of $\tau$ and shuffling the roots while preserving the order of the roots cut off from the same node.

\begin{theorem}
\label{theorem_MKW}
Let the Munthe-Kaas-Wright coproduct on~$\EF$ be, for $\tau \in ET$ and $\pi \in EF$, defined as
  \begin{align*}
    \Delta_{MKW} (\tau) &:= \sum_{\text{adm. cut } c} P^c(\tau) \otimes R^c(\tau), \\
    \Delta_{MKW} (\pi) &:= (\text{id} \otimes B^-)\Delta_{MKW}(B^+(\pi)),
  \end{align*}
where the sum is over admissible cuts $c$. 
  Then, $(\EF,\shuffle,\Delta_{MKW})$ is a Hopf algebra dual to the planar Grossman-Larson Hopf algebra $(\EF, \gl, \Delta_\shuffle)$. Its convolution product,
  \[ a * b = (a \otimes b) \circ \Delta_{MKW}, \]
  is called the composition law.
\end{theorem}

\begin{proof}
  We recall that the dual of a Hopf algebra is a Hopf algebra. It is clear that $\shuffle$ and $\Delta_\shuffle$ are dual. It remains to show that the Munthe-Kaas-Wright coproduct $\Delta_{MKW}$ is dual to the Grossman-Larson product $\gl$. We recall that Munthe-Kaas-Wright coproduct and Grossman-Larson product are dual over the planar decorated forests. The Munthe-Kaas-Wright coproduct over exotic forests is obtained by taking the quotient over $\II$ defined in \eqref{eq:EF_codef} for the construction of deshuffle coproduct. This is the dual approach to the way the Grossman-Larson product over exotic forests is obtained.
\end{proof}

Some examples of the values of $\Delta_{MKW}$ are
\begin{align*}
  \Delta_{MKW} (\forestWG) &= \forestXG \otimes \mathbf{1} + \mathbf{1} \otimes \forestYG, \\
  \Delta_{MKW} (\forestAH) &= \forestBH \otimes \mathbf{1} + \forestCH \otimes \forestDH + \mathbf{1} \otimes \forestEH, \\ 
  \Delta_{MKW} (\forestFH) &= \forestGH \otimes \mathbf{1} + 2 \forestHH \otimes \forestIH + \big( \forestJH + \forestKH ) \otimes \forestLH + \big( \forestMH + \forestNH ) \otimes \forestOH\\& + \big( \forestPH + \forestQH ) \otimes \forestRH  + 2 \forestSH \otimes \forestTH + \mathbf{1} \otimes \forestUH.
\end{align*}

The Hopf algebra structures on $\EF$ translate to the structures on $T(\mathfrak{X}(\MM))$ via the use of the elementary differential map $\dF$ defined in Section \ref{section:exotic_forests}.
We recall that decorated planar forest form a free D-algebra with the concatenation and grafting products \cite{MuntheKaas08oth} which is generated by the set of decorations $D$. The elementary differential map over decorated planar forests $\dFdec : \FF_D \to T(\mathfrak{X}(\MM))$ is a D-algebra homomorphism that is fully characterized by its values on the nodes decorated by $D = \{ \bullet, \circ^l \; : \; 1 \leq l \leq L \}$,
\[ \dFdec(\bullet) = f, \quad \dFdec(\circ^l) = \xi^{d,l} E_d, \]
where the $\xi^{d,l}$ are independent standard Gaussian random variables. The map $\dFdec$ thus defined is equivalent to the map $\dFdec$ from Definition \ref{def:dFdec}.

The space of exotic forests $\EF$ is isomorphic to a D-subalgebra $\FF_D^\EE$ of the D-algebra of decorated forests with $D = \{ \forestVH, \forestWH \; : \; d \in \N \}$ through the isomorphism $\varphi: \EF \to \FF_D^\EE$ as is described in the beginning of Section \ref{sec:exotic_forests}, and, thus, inherits the elementary differential homomorphism as $\dF := \tildedFdec \circ \varphi$ with $\tildedFdec : \FF_D^\EE \to T(\mathfrak{X}(\MM))$ defined as
\[ \tildedFdec(\forestXH) = f, \quad \tildedFdec(\forestYH) = E_d. \]
This definition is equivalent to the one given in Definition \ref{def:exotic_elementary_differential}.

\begin{proposition}
\label{prop:F_morphism}
The elementary differential $\F\colon \EF\rightarrow T(\mathfrak{X}(\MM))$ is a Hopf algebra morphism:
\[
\F(\pi_1\graft\pi_2)=\F(\pi_1)\triangleright \F(\pi_2),\quad
\F(\pi_1\cdot\pi_2)=\F(\pi_1)\cdot \F(\pi_2),\quad
\F(\pi_1\diamond\pi_2)=\F(\pi_1)* \F(\pi_2).
\]
In particular, the frozen composition is represented by concatenation and the composition of differential operators is represented by the Grossman-Larson product:
\begin{align*}
(\F(\pi_1)\cdot \F(\pi_2))\triangleright\phi(p)&=
\F_p(\pi_1)\triangleright(\F_p(\pi_2)\triangleright\phi)(p)=\F(\pi_1\cdot\pi_2)\triangleright\phi(p),\\
(\F(\pi_1)* \F(\pi_2))\triangleright\phi(p)&=
\F(\pi_1)\triangleright (\F(\pi_2) \triangleright \phi)(p) = \F(\pi_1 \gl \pi_2)\triangleright\phi(p).
\end{align*}
\end{proposition}

\section{Numerical experiments}
\label{section:numerical_experiments}

We provide several numerical experiments to confirm the theoretical findings and the numerical properties of the new schemes.
In particular, we present a practical guide for the implementation of frozen flow methods, convergence curves for the weak error on the orthogonal group, a comparison with the sampling geometric methods from the literature on the sphere, and an ergodicity test for the simulation of generalised Cauchy measures in a context where the assumptions of our convergence analysis are not satisfied.

\subsection{Practical implementation of frozen flow methods}
\label{section:practical_implementation}

The implementation of frozen flow methods relies on the ability to find local frames and to compute the solution of the frozen equation
\[y'(t)=\sum_d \alpha_d E_d(y(t)),\quad y(0)=y_0.\]
On homogeneous manifolds, we recall that global frames can be derived from any basis $(A_d)$ of the Lie algebra $\mathfrak{g}$ with the derivative of Lie-group action: $E_d(p)=A_d\cdot p$.
On matrix homogeneous manifolds, the solution of the frozen flow equation coincides with the matrix exponential:
\[y(t)=\Exp(\sum_d \alpha_d A_d)\cdot y_0.\]

For sampling the invariant measure of Riemannian Langevin dynamics \eqref{equation:Langevin}, the algorithms require at each step a local orthonormal frame basis and the solution of the frozen flow equation. While one can change the orthonormal local frame between different steps, it is crucial to rely only on one fixed frame per step.
On Lie groups, choosing an orthonormal basis $(A_d)$ of the Lie algebra $\mathfrak{g}$ for the definition of the $E_d$ naturally yields an orthonormal frame. On general Riemannian manifolds, one can provide an orthonormal basis of $T_{X_0} \MM$ and parallel transport it around $X_0$ to obtain a local orthonormal frame basis.

In general, solving the frozen flow equation is not straightforward. It is equally difficult to solving the geodesic equation, as done in \cite{Bharath23sae}.
Similar to the idea of retractions (see, for instance, the textbook \cite{Absil08oao}), the frozen flow could be replaced by any approximation with sufficiently high order of accuracy and that lies on the manifold. A standard example of such an approximation is the Cayley map \cite{Iserles00lgm}.
In our numerical experiments, we rely on exact expressions of the frozen flows using either Lie geometry or coordinates for the sake of simplicity.


We emphasize that the geometric operations needed for stochastic intrinsic methods are available on a wide collection on manifolds and in a variety of programming language (see, for instance, \cite{Engo01dao, MuntheKaas16ioh, Axen23mae}).


\subsection{Brownian dynamics on the special orthogonal group}

Let $\MM=\SO_p(\R)$ be the compact Lie group of special orthogonal matrices of size $p\times p$. Let $(A_d)$ be an orthonormal basis of its Lie algebra $\mathfrak{so}_d$ of antisymmetric matrices and for the bi-invariant metric $g(A,B)=\Trace(A^TB)$.
Let $E_d(y)=A_d y$ be the associated right-invariant orthonormal frame basis.
As the manifold is smooth and compact, the geometric assumptions are all satisfied automatically and our analysis applies.

Following Example \ref{example:BD}, Brownian dynamics on $\SO_p(\R)$ are given by
\[
dX(t)=\sum_{d=1}^D E_d(X(t)) \circ dW_d(t), \quad D=\dim(\mathfrak{so}_p)=\frac{p(p-1)}{2}. 
\]
In this context, the implementation of the new frozen flow methods is very natural as it relies only on fast linear algebra operations.
In particular, the new second order method \eqref{equation:new_Brownian_method} becomes (under a time rescaling)
\[
X_{n+1}=
\Exp\bigg(\sum_{d=1}^D \Big(\big(\frac{\sqrt{2}}{2}-1\big)\sqrt{h}\xi^{d,1}_n+\frac{\sqrt{2}}{2}\sqrt{h}\xi^{d,2}_n\Big) A_d\bigg)
\Exp\bigg(\sum_{d=1}^D \sqrt{h}\xi^{d,1}_n A_d\bigg)
X_n.
\]
Note that there is no timestep restriction for the scheme to be well-posed.
This is in striking contrast to projection methods, which would rely on a tedious non-geometric implementation, on the embedding of $\SO_p(\R)$ into $\R^{p\times p}$ (of more than double dimension), and would require a small enough timestep for the projection map to be well-defined, even in this simple context of a compact manifold.
%

We consider the weak approximation of Brownian dynamics on $\SO_3(\R)$ with the Euler frozen flow method \eqref{equation:Euler_FF} and the new second order method \eqref{equation:new_Brownian_method} (with discrete bounded random variables with correct first moments - see Section \ref{section:new_FF_integrators}) for the test function
\[
\phi(x)=\exp(-\frac{\Trace((x-I_d)^T(x-I_d))}{2d})
,\]
the initial condition $X_0=I_d$, final time $T=1$, and different timesteps to observe the order of convergence.
We emphasize that the explicit expression of Brownian dynamics $X(t)$ on $\SO_3(\R)$ are not given by $\Exp(\sum A_d W_d(t))X_0$, but rather by a stochastic Magnus expansion (see, for instance, \cite{Burrage99rkm, Burrage99hso, Iserles00lgm, Wang20tme, Kamm21ots}):
\[X(t)=\Exp(\Omega(t))X_0, \quad d\Omega(t)=d\Exp^{-1}_{\Omega(t)}(\sum_d A_d \circ d W_d(t)), \quad \Omega(0)=0,\]
whose first terms are
\begin{align*}
\Omega(t)&=\sum_{d_1} A_{d_1} W_{d_1}(t)
-\frac{1}{2}\int_0^t \sum_{d_1,d_2} [A_{d_2},A_{d_1}] W_{d_2}(s) \circ dW_{d_1}(s)\\
&+\frac{1}{12}\int_0^t \sum_{d_1,d_2,d_3} [A_{d_3},[A_{d_2},A_{d_1}]] W_{d_3}(s) W_{d_2}(s) \circ dW_{d_1}(s)\\
&+\frac{1}{4}\int_0^t \int_0^s \sum_{d_1,d_2,d_3} [[A_{d_3},A_{d_2}] ,A_{d_1}]] W_{d_3}(u) \circ dW_{d_2}(u) \circ dW_{d_1}(s)
+\dots
\end{align*}
For simplicity, the reference solution for $\E[\phi(X(T))]$ is chosen as the output of the new second order scheme with the timestep $h_{\text{ref}}=2^{-12}$.
We compute the Monte-Carlo estimator~$\widebar{J}=\frac{1}{M}\sum_{m=1}^M \phi(X_N^{(m)}) \simeq \E[\phi(X_N)]$, where~$X_n^{(m)}$ is the~$m$-th realisation of the integrator at time~$t_n=nh$,~$Nh=T$, and we use $M=10^8$ trajectories.
We also plot an estimate of the Monte-Carlo error by using the standard error of the mean estimator $\big(\sum_{m=1}^M (\phi(X_N^{(m)})-\widebar{J})^2\big)^{1/2}/\sqrt{M(M-1)}$.
We observe in Figure \ref{figure:Plot_SO} the correct order of convergence of the two new methods. The second order integrator reaches the Monte-Carlo error threshold with much larger timesteps compared to the Euler frozen-flow method.

\begin{figure}[ht]
	\begin{minipage}[c]{.44\linewidth}
		\begin{center}
		\includegraphics[scale=0.48]{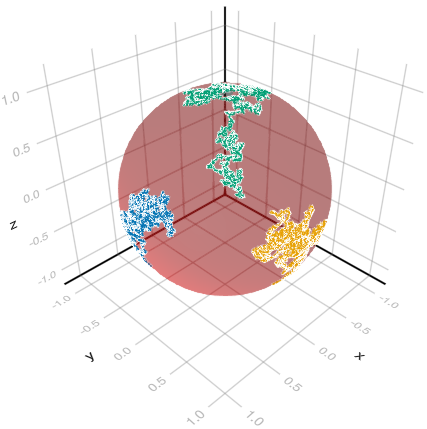}
		\end{center}
	\end{minipage} \hfill
	\begin{minipage}[c]{.54\linewidth}
		\begin{center}
		\includegraphics[scale=0.4]{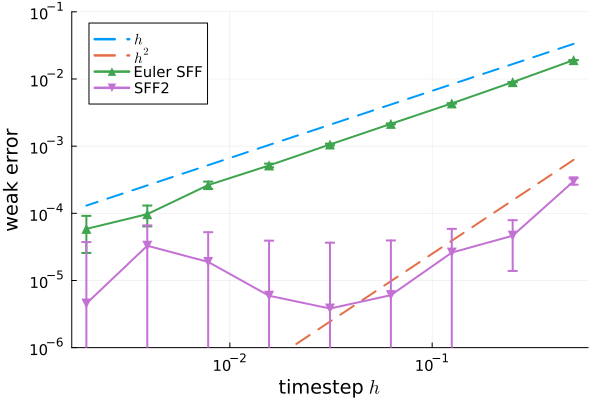}
		\end{center}
	\end{minipage}
	\caption{Brownian excursion at time $T=1$ on $\SO_3(\R)$ represented by the trajectories of the three columns on the sphere $\S^2$ (left) and error curve for the weak approximation of Brownian dynamics on $\SO_3(\R)$ with $\phi(x)=\exp(-\Trace((x-I_d)^T(x-I_d))/2d)$, $T=1$ and $M=10^8$ trajectories (right).}
	\label{figure:Plot_SO}
\end{figure}

%
%

\subsection{Ergodic dynamics on the sphere}

Our second test case focuses on the sampling of Riemannian Langevin dynamics on the sphere $\S^2$, following the experiment of \cite{Bharath23sae}.
While the approximation for the invariant measure is not the focus of the new methods implemented in this paper, this example allows us to conveniently compare the performance of the new frozen flow methods to the alternatives from the literature.
The sphere $\S^d$ can be seen as a homogeneous manifold with respect to the Lie group $G=\SO_{d+1}(\R)$, so that we could consider a natural frame derived from a basis of the Lie algebra $\mathfrak{so}_{d+1}$ (which would typically fit the simulation of a stochastic rigid body). However, we need an orthonormal frame to sample from the Riemannian Langevin dynamics \eqref{equation:Langevin}, so that we use a different choice of frame relying on the standard spherical coordinates.
We consider the coordinates on $\S^2$ minus the north/south poles:
\[
y_{\theta,\varphi}=\begin{pmatrix}
\cos(\theta)\cos(\varphi)\\\cos(\theta)\sin(\varphi)\\\sin(\theta)
\end{pmatrix}\in \S^2,\quad \theta\in ]-\frac{\pi}{2},\frac{\pi}{2}[,\quad \varphi\in[0,2\pi[,
\]
with the associated orthonormal frame basis:
\[
E_1(y_{\theta,\varphi})=\begin{pmatrix}
-\sin(\theta)\cos(\varphi)\\-\sin(\theta)\sin(\varphi)\\\cos(\theta)
\end{pmatrix},
\quad
E_2(y_{\theta,\varphi})=\begin{pmatrix}
-\sin(\varphi)\\\cos(\varphi)\\0
\end{pmatrix},
\]
and a second frame associated to the spherical coordinates on $\S^2$ minus $[\pm 1,0,0]$.
Using two frames guarantees a bounded Lipschitz constant for the chosen frame and does not impact the sampled measure.
As the manifold is compact, the assumptions of our analysis are satisfied automatically.
The frozen flow is computed explicitly in coordinates.
In order to obtain a fair comparison between the methods, we use Gaussian random variables.

We now compare the numerical behaviour of the frozen flow Euler method \eqref{equation:Euler_FF}, the new second order method \eqref{equation:new_weak_method}, the Riemannian Langevin method \cite{Bharath23sae}, and the extrinsic projection methods of order one \cite{Lelievre10fec} for sampling the invariant measure of Riemannian Langevin dynamics \eqref{equation:Langevin} on $\S^2$.
We consider the potential $V(y_{\theta,\varphi})=-\sin(\theta)$, the associated vector field (containing the Ito correction)
\[
F=-\nabla V-\nabla_{E_1}E_1-\nabla_{E_2}E_2,
\]
and the test function $\phi(x)=x_3^2$.
In this context, the generator is $\LL\phi=-\phi'(\nabla V)+ \Delta_{\S^2}$, with the spherical Laplacian $\Delta_{\S^2}$.
Assuming ergodicity (which is observed numerically for the methods studied here), the numerical schemes approximate the quantity
\[
\frac{\int_{\S^2}\phi(x) e^{-V(x)}d\vol(x)}{\int_{\S^2}e^{-V(x)}d\vol(x)}
=3-\frac{2}{\tanh(1)}.
\]

We observe the weak error curves in Figure \ref{figure:Plot_sphere}, as well as the expected orders of convergence. The Monte-Carlo error estimate is represented with error bars as before, but only for the SFF2 method for the sake of clarity, as the error is negligible compared to the discretisation bias for the other methods.
We observe that the new second-order frozen flow method outperforms the other integrators and reaches the Monte-Carlo threshold with much larger timestep.

\begin{figure}[ht]
	\begin{minipage}[c]{.44\linewidth}
		\begin{center}
		\includegraphics[scale=0.6]{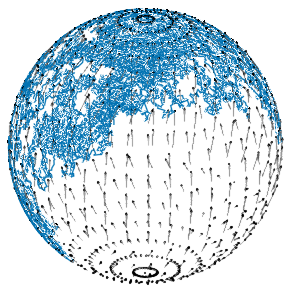}
		\end{center}
	\end{minipage} \hfill
	\begin{minipage}[c]{.54\linewidth}
		\begin{center}
		\includegraphics[scale=0.4]{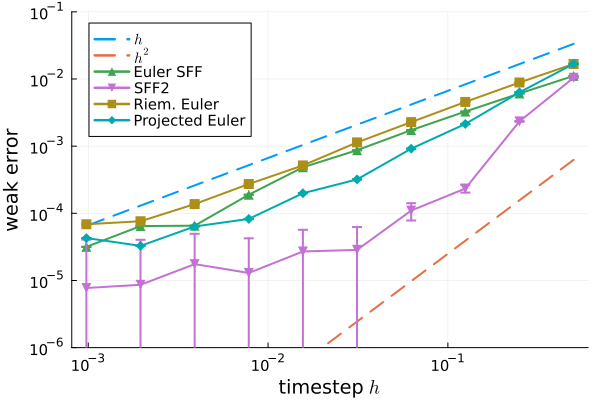}
		\end{center}
	\end{minipage}
	\caption{A trajectory of the new second order method \eqref{equation:new_Brownian_method} (left) and the convergence curve for the weak approximation of stochastic rigid body equation on the sphere (right) with $\phi(x)=x_3^2$, $T=10$ and $M=10^8$ trajectories.}
	\label{figure:Plot_sphere}
\end{figure}

\subsection{Generalised Cauchy measures}

The generalised Cauchy measures are probability measures $\mu_\beta$ defined on $\R^n$ by
\[\mu_\beta(dx) = \left(1+|x|^2\right)^{-\beta}\frac{dx}{Z_\beta},\quad Z_\beta = \frac{\Gamma(\beta-n/2)}{\Gamma(\beta)},\]
with $\beta>n/2$. They can be interpreted as finite dimensional approximations of Gaussian distributions, as in \cite{Bobkov2009wpi}, and are related to the fast diffusion equation for which they play the role of the heat kernel and are known as Barenblatt' solutions (see \cite{Vasquez07tpm}).
A generalised Cauchy measure is the ergodic distribution of the dynamics
\begin{equation}
\label{equation:Cauchy_Euc}
dX_t = -2(\beta-1)X_t dt + \sqrt{2}(1+|X_t|^2)dB(t),
\end{equation}
where $B(t)$ is a 2-dimensional Brownian motion.
A direct numerical simulation of this SDE presents numerous difficulties as the multiplicative non-Lipschitz noise brings tedious stability problems, so that most standard Euclidean numerical methods fail.
However, the noise can be interpreted as additive at the cost of a change of metric $g_x = (1+|x|^2)^{-1}\langle -,-\rangle$. This change is not simply a transport of difficulties from noise to geometry, as it reinterprets the problem in a convenient space. The same approach has been fruitful in the study of the associated spectral gap, i.e. the speed of convergence of the semigroup to the ergodic measure, in \cite{Huguet22IRF}. This context calls for methods on manifolds, as Euclidean methods are not adapted.
We emphasize that the process $(X_t)_t$ does not have bounded moments at all orders. Thus, the convergence results of the standard Euclidean analysis, and of our analysis as well, do not hold.
This makes the simulation of generalised Cauchy measures a convenient toy problem for testing our methods.

We are concerned with the simulation of generalised Cauchy measure on $\R^2$. As explained in \cite{Huguet20sco}, the adapted manifold can be interpreted as the revolution surface 
\begin{equation}
\label{equation:def_S_Cauchy}
\SS = \left\{(\tanh(r)\cos(\theta), \tanh(r)\sin(\theta), h(r)) ; (r,\theta)\in\R_+\times\S^1 \right\},
\end{equation}
with $h(r) = \asinh(\cosh(r)) -\cosh^{-1}(r)\sqrt{\cosh^2(r)+1}$. However, it is more handy to see it through a global map, as $\R_+\times\S^1$, with coordinate $(r,\theta)$, endowed with the metric
\[ g = \begin{pmatrix} 1 &0\\ 0&\tanh^2(r)\\ \end{pmatrix}\]
in the basis $(\partial_r, \partial_\theta)$. This map comes from fitting polar coordinates on $\R^2$, 
\[(x,y) = (\sinh(r)\cos(\theta), \sinh(r)\sin(\theta)),\] which trivialise the metric $g$. We define the following orthonormal frame basis on $\R^*_+\times\S^1$
\[E_1(r,\theta) = \partial_r,\quad E_2(r,\theta) = \frac{1}{\tanh(r)}\partial_\theta.\]
In this manifold, generalised Cauchy measures correspond to Gibbs measures associated to the potential $V_\beta(r,\theta) = 2(\beta-1)\log(\cosh(r))$ and the generator
\[\LL = E_1^2 + E_2^2 +\left(\frac{1}{\tanh(r)}-(2\beta-1)\tanh(r)\right)E_1.\]
Note that the frozen flow is explicit on $\SS$. For all $\alpha,\beta\in\R$ and $(r_0,\theta_0)\in\R_+^*\times\S^1$, we find
\[\exp\left(t(\alpha E_1 + \beta E_2)\right)\cdot(r_0,\theta_0) = \left(r_0 + \alpha t,\theta_0 + \frac{\beta}{\alpha}\log\left(\frac{\sinh(r_0 + \alpha t)}{\sinh(r_0)}\right)\right),\]
for all $t\geq0$ such that $r_t$ stays positive.
This brings an additional difficulty in this experiment. Although the potential $f^1$ blows up at the origin, the coefficient 
\[\alpha = hf^1(r) + \sqrt{2h}\xi\]
with $\xi$ a Gaussian random variable, can be negative and the frozen flow not defined. The use of bounded random variable is instrumental in avoiding this instability. In particular, the method \eqref{equation:Euler_FF} with Rademacher random variables is well-posed.

Our numerical experiment aims to illustrate the ergodicity of the simplest frozen flow method \eqref{equation:Euler_FF}, which, to the best of our knowledge, has not yet been established in the literature.
We use the first eigenfunctions of the generator $\phi^1(r,\theta) = \sinh(r)^2-(\beta-2)^{-1}$, $\phi^2(r,\theta) = \sinh(r)\cos(\theta)$, which both satisfy $\int\phi^k\, d\mu_\beta=0$.
As a consequence of the Poincaré inequality, the variance of the semigroup associated to $\LL$ converges at exponential speed to $0$. The coefficient of the exponential is given by the spectral gap (see \cite{Huguet24pia} for the spectral gap expression). The test functions $\phi^k$ are extremal for the Poincaré inequality. Hence, for these particular test functions, the exponential convergence can be observed not only for the variance but also pointwise. 
In Figure \ref{figure:Plot_Cauchy}, we illustrate numerically the exponential convergence to equilibrium of the numerical methods, in the spirit of \cite{Debussche12wbe} in the Euclidean setting. In particular, we observe a convergence in $Ce^{-4(\beta-2)t}$ and $Ce^{-2(\beta-1)t}$ for our specific choice of test functions, confirming numerically the results in \cite{Huguet24pia}.
For the first eigenfunction $\phi^1$, the error for the invariant measure reaches the discretisation bias with exponential speed. We observe that the bias threshold evolves linearly in $h$ (up to Monte-Carlo error), as expected for a method of order one.
For $\phi^2$, we do not observe the bias of the Euler frozen flow method \eqref{equation:Euler_FF}, hinting that for this specific test function, the method might be of higher order of accuracy for the invariant measure.
For the sake of comparison, we mention that the standard Euclidean integrators applied to \eqref{equation:Cauchy_Euc} face severe instability issues and thus fail to converge, let alone be ergodic.

\begin{figure}[ht]
	\begin{minipage}[c]{.48\linewidth}
		\begin{center}
		\includegraphics[scale=0.35]{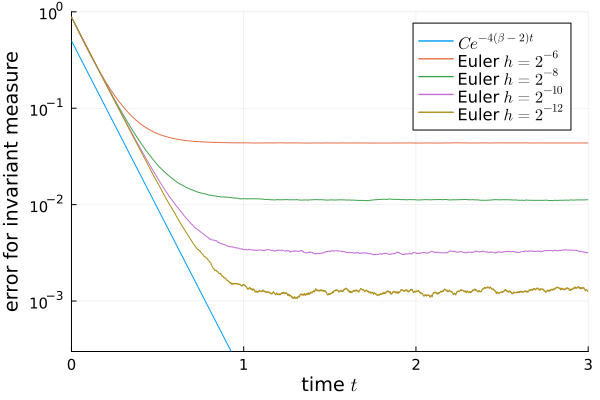}
		\end{center}
	\end{minipage} \hfill
	\begin{minipage}[c]{.48\linewidth}
		\begin{center}
		\includegraphics[scale=0.35]{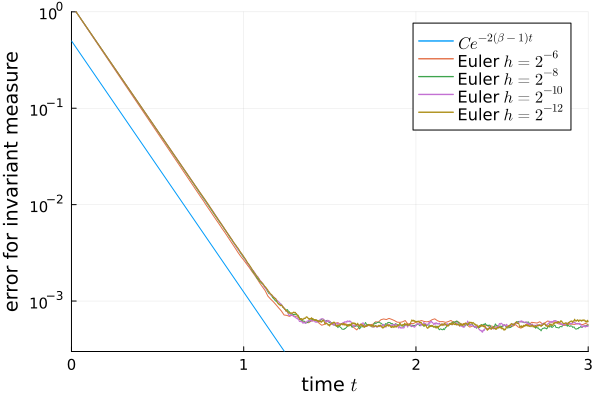}
		\end{center}
	\end{minipage}
	\caption{Evolution in time of the error for the invariant measure of the Euler frozen flow method \eqref{equation:Euler_FF} for the simulation of generalised Cauchy measures on the manifold \eqref{equation:def_S_Cauchy} with $\beta=4$, $M=10^8$ trajectories, and the test functions $\phi^1(r,\theta) = \sinh(r)^2-(\beta-2)^{-1}$ (left) and $\phi^2(r,\theta) = \sinh(r)\cos(\theta)$ (right).}
	\label{figure:Plot_Cauchy}
\end{figure}

\section{Conclusion}
\label{section:conclusion}

In this paper, we provided a brand-new class of intrinsic discretisations for the approximation of stochastic dynamics on Riemannian manifolds with high weak order of accuracy.
We presented a new robust convergence analysis, that naturally generalises the Euclidean analysis, and the new algebraic formalism of planar exotic trees, forests, and series for the study of the order conditions.
The analysis applies in particular for the construction of high-order sampling methods for Riemannian Langevin dynamics.
The new methods were tested on a variety of different manifolds, where we observe that the new second order frozen flow method outperforms the other approaches from the literature in terms of accuracy.

The tools developed in the present paper open several avenues for future research.
Similar to the Euclidean setting, we will extend the analysis to general SDEs with multiplicative noise and to the construction of integrators with high-order of accuracy for sampling the invariant measure, and low weak order (in the spirit of \cite{Leimkuhler13rco, Vilmart15pif}). This requires generalisations of the exotic forests formalism \cite{Busnot26osr}, the challenging extension of the integration by parts of trees in the planar case \cite{Laurent20eab, Laurent21ocf, Bronasco22ebs, Bronasco22cef}, and the extension of stochastic backward error analysis for the invariant measure \cite{Shardlow06mef, Zygalakis11ote, Debussche12wbe, Abdulle14hon, Bronasco22cef} on manifolds.
While the focus of this paper is on discretisations and improvements in accuracy only, our approach could naturally be combined with popular sampling techniques, such as MLMC \cite{Giles08mmc,Giles15mmc}, perturbations \cite{Lelievre13onr,Duncan16vru,Abdulle19act},...
Then, following the work \cite{Bharath23sae}, one could wonder whether a purely Riemannian approach of algebraic order theory could be used. Our approach follows the one of deterministic Lie-group methods and thus relies on a curvature-free connection, unrelated to the natural Levi-Civita connection on Riemannian manifolds. The study of such new methods is already open in the deterministic setting and relies on the challenging general understanding of the connection algebra \cite{AlKaabi22aao}.
Finally, the new algebraic formalism of planar exotic series is interesting in itself and could be studied for its universal combinatorial, algebraic, and geometric properties, but also for its potential applications in different fields, in the spirit of the use of Butcher series and their extensions in rough paths \cite{Hairer15gvn, Lejay22cgr}, renormalisation theory \cite{Brouder00rkm}, variational calculus \cite{Laurent23tab, Laurent23tld} or approximation of PDEs \cite{Bronsard22aod}.
These projects will be studied in upcoming works.

\bigskip

\noindent \textbf{Acknowledgements.}\
The authors would like to thank R.\ts Bergmann, E.\ts Celledoni, H.\ts Munthe-Kaas, and B.\ts Owren for helpful discussions and advice on manipulating the package Manifolds.jl \cite{Axen23mae}.
The authors acknowledge the support of the Swiss National Science Foundation, projects No 200020\_214819, and No. 200020\_192129, of the Research Council of Norway through project 302831 ``Computational Dynamics and Stochastics on Manifolds'' (CODYSMA), of the French program ANR-11-LABX-0020-0 (Labex Lebesgue) and ANR-25-CE40-2862-01 (MaStoC), and of the Knut and Alice Wallenberg Foundation (grant number KAW 2023.0433).
The computations were performed on the Baobab cluster of the University of Geneva.

\begin{small}

\bibliographystyle{abbrv}
\bibliography{ma_bibliographie}

\end{small}

\newpage
\begin{appendices}

\section{Order conditions for third weak order}
\label{app:order_cond_3}

For the sake of completion, we detail in Table \ref{table:weak_order_conditions_order_3} the 95 additional order conditions for weak order 3 with stochastic frozen flow methods \eqref{equation:def_CGsto}. Recall that these conditions are not independent and satisfy shuffle relations (see Remark \ref{remark:shuffle_identities}).

\begin{longtable}{C|C|C}
\text{Exotic forest } \pi & a(\pi) & e(\pi) \\\hline
\forestAI & z^{0}_{i,k_3} Z^{0}_{i,j,k_2} Z^{0}_{j,k,k_1} & \frac{1}{6}\\
\forestBI & \sum^!_{k_1\geq k_2} z^{0}_{i,k_3} Z^{0}_{i,j,k_2} Z^{0}_{i,k,k_1} & \frac{1}{6}\\
\forestCI & \sum^!_{k_1\geq k_2} z^{0}_{i,k_4} Z^{0}_{i,j,k_3} Z^{d_1}_{j,k_2} Z^{d_1}_{j,k_1} & \frac{1}{6}\\
\forestDI & \sum^!_{k_1\geq k_3} z^{0}_{i,k_4} Z^{0}_{i,j,k_3} Z^{d_1}_{j,k_2} Z^{d_1}_{i,k_1} & \frac{1}{3}\\
\forestEI & \sum^!_{k_3\geq k_1} z^{0}_{i,k_4} Z^{0}_{i,j,k_3} Z^{d_1}_{j,k_2} Z^{d_1}_{i,k_1} & \frac{1}{3}\\
\forestFI & \sum^!_{k_1\geq k_2\geq k_3} z^{0}_{i,k_4} Z^{0}_{i,j,k_3} Z^{d_1}_{i,k_2} Z^{d_1}_{i,k_1} & \frac{1}{6}\\
\forestGI & \sum^!_{k_1\geq k_3\geq k_2} z^{0}_{i,k_4} Z^{0}_{i,j,k_3} Z^{d_1}_{i,k_2} Z^{d_1}_{i,k_1} & 0\\
\forestHI & \sum^!_{k_3\geq k_1\geq k_2} z^{0}_{i,k_4} Z^{0}_{i,j,k_3} Z^{d_1}_{i,k_2} Z^{d_1}_{i,k_1} & \frac{1}{6}\\
\forestII & \sum^!_{k_1\geq k_2\geq k_3\geq k_4} z^{0}_{i,k_5} Z^{d_2}_{i,k_4} Z^{d_2}_{i,k_3} Z^{d_1}_{i,k_2} Z^{d_1}_{i,k_1} & \frac{1}{6}\\
\forestJI & \sum^!_{k_1\geq k_2\geq k_3\geq k_4} z^{0}_{i,k_5} Z^{d_2}_{i,k_4} Z^{d_1}_{i,k_3} Z^{d_2}_{i,k_2} Z^{d_1}_{i,k_1} & 0\\
\forestKI & \sum^!_{k_1\geq k_2\geq k_3\geq k_4} z^{0}_{i,k_5} Z^{d_1}_{i,k_4} Z^{d_2}_{i,k_3} Z^{d_2}_{i,k_2} Z^{d_1}_{i,k_1} & 0\\
\forestLI & \sum^!_{k_1\geq k_2} Z^{0}_{j,k,k_3} z^{0}_{j,k_2} z^{0}_{i,k_1} & \frac{1}{6}\\
\forestMI & \sum^!_{k_2\geq k_1} Z^{0}_{j,k,k_3} z^{0}_{j,k_2} z^{0}_{i,k_1} & \frac{1}{3}\\
\forestNI & \sum^!_{k_1\geq k_2,k_3\geq k_4} Z^{d_1}_{j,k_4} Z^{d_1}_{j,k_3} z^{0}_{j,k_2} z^{0}_{i,k_1} & \frac{1}{6}\\
\forestOI & \sum^!_{k_2\geq k_1,k_3\geq k_4} Z^{d_1}_{j,k_4} Z^{d_1}_{j,k_3} z^{0}_{j,k_2} z^{0}_{i,k_1} & \frac{1}{3}\\
\forestPI & \sum^!_{k_1\geq k_2} Z^{d_1}_{j,k_4} Z^{d_1}_{i,k_3} z^{0}_{j,k_2} z^{0}_{i,k_1} & \frac{1}{6}\\
\forestQI & \sum^!_{k_1\geq k_2} Z^{d_1}_{j,k_4} Z^{0}_{i,j,k_3} z^{0}_{i,k_2} z^{d_1}_{k_1} & 0\\
\forestRI & \sum^!_{k_2\geq k_1} Z^{d_1}_{j,k_4} Z^{0}_{i,j,k_3} z^{0}_{i,k_2} z^{d_1}_{k_1} & \frac{1}{3}\\
\forestSI & \sum^!_{k_1\geq k_2,k_3\geq k_4} Z^{0}_{i,j,k_4} Z^{d_1}_{i,k_3} z^{0}_{i,k_2} z^{d_1}_{k_1} & 0\\
\forestTI & \sum^!_{k_1\geq k_2,k_4\geq k_3} Z^{0}_{i,j,k_4} Z^{d_1}_{i,k_3} z^{0}_{i,k_2} z^{d_1}_{k_1} & 0\\
\forestUI & \sum^!_{k_2\geq k_1,k_3\geq k_4} Z^{0}_{i,j,k_4} Z^{d_1}_{i,k_3} z^{0}_{i,k_2} z^{d_1}_{k_1} & 0\\
\forestVI & \sum^!_{k_2\geq k_1,k_4\geq k_3} Z^{0}_{i,j,k_4} Z^{d_1}_{i,k_3} z^{0}_{i,k_2} z^{d_1}_{k_1} & \frac{1}{3}\\
\forestWI & \sum^!_{k_1\geq k_2,k_3\geq k_4\geq k_5}  Z^{d_2}_{i,k_5} Z^{d_2}_{i,k_4} Z^{d_1}_{i,k_3} z^{0}_{i,k_2} z^{d_1}_{k_1} & 0\\
\forestXI & \sum^!_{k_1\geq k_2,k_3\geq k_4\geq k_5}  Z^{d_2}_{i,k_5} Z^{d_1}_{i,k_4} Z^{d_2}_{i,k_3} z^{0}_{i,k_2} z^{d_1}_{k_1} & 0\\
\forestYI & \sum^!_{k_1\geq k_2,k_3\geq k_4\geq k_5}  Z^{d_1}_{i,k_5} Z^{d_2}_{i,k_4} Z^{d_2}_{i,k_3} z^{0}_{i,k_2} z^{d_1}_{k_1} & 0\\
\forestAJ & \sum^!_{k_2\geq k_1,k_3\geq k_4\geq k_5}  Z^{d_2}_{i,k_5} Z^{d_2}_{i,k_4} Z^{d_1}_{i,k_3} z^{0}_{i,k_2} z^{d_1}_{k_1} & \frac{1}{3}\\
\forestBJ & \sum^!_{k_2\geq k_1,k_3\geq k_4\geq k_5}  Z^{d_2}_{i,k_5} Z^{d_1}_{i,k_4} Z^{d_2}_{i,k_3} z^{0}_{i,k_2} z^{d_1}_{k_1} & 0\\
\forestCJ & \sum^!_{k_2\geq k_1,k_3\geq k_4\geq k_5}  Z^{d_1}_{i,k_5} Z^{d_2}_{i,k_4} Z^{d_2}_{i,k_3} z^{0}_{i,k_2} z^{d_1}_{k_1} & \frac{1}{3}\\
\forestDJ & \sum^!_{k_1\geq k_2\geq k_3}  z^{0}_{k,k_3} z^{0}_{j,k_2} z^{0}_{i,k_1} & \frac{1}{6}\\
\forestEJ & \sum^!_{k_1\geq k_2\geq k_3}  Z^{0}_{i,j,k_4} z^{0}_{i,k_3} z^{d_1}_{k_2} z^{d_1}_{k_1} & \frac{1}{6}\\
\forestFJ & \sum^!_{k_1\geq k_3\geq k_2}  Z^{0}_{i,j,k_4} z^{0}_{i,k_3} z^{d_1}_{k_2} z^{d_1}_{k_1} & 0\\
\forestGJ & \sum^!_{k_3\geq k_1\geq k_2}  Z^{0}_{i,j,k_4} z^{0}_{i,k_3} z^{d_1}_{k_2} z^{d_1}_{k_1} & \frac{1}{6}\\
\forestHJ & \sum^!_{k_1\geq k_2\geq k_3}  Z^{d_1}_{j,k_4} z^{0}_{j,k_3} z^{0}_{i,k_2} z^{d_1}_{k_1} & 0\\
\forestIJ & \sum^!_{k_2\geq k_1\geq k_3}  Z^{d_1}_{j,k_4} z^{0}_{j,k_3} z^{0}_{i,k_2} z^{d_1}_{k_1} & 0\\
\forestJJ & \sum^!_{k_1\geq k_3\geq k_2}  Z^{d_1}_{j,k_4} z^{0}_{j,k_3} z^{0}_{i,k_2} z^{d_1}_{k_1} & 0\\
\forestKJ & \sum^!_{k_3\geq k_1\geq k_2}  Z^{d_1}_{j,k_4} z^{0}_{j,k_3} z^{0}_{i,k_2} z^{d_1}_{k_1} & \frac{1}{3}\\
\forestLJ & \sum^!_{k_2\geq k_3\geq k_1}  Z^{d_1}_{j,k_4} z^{0}_{j,k_3} z^{0}_{i,k_2} z^{d_1}_{k_1} & \frac{1}{3}\\
\forestMJ & \sum^!_{k_3\geq k_2\geq k_1}  Z^{d_1}_{j,k_4} z^{0}_{j,k_3} z^{0}_{i,k_2} z^{d_1}_{k_1} & \frac{1}{3}\\
\forestNJ & \sum^!_{k_1\geq k_2\geq k_3,k_4\geq k_5} Z^{d_2}_{i,k_5} Z^{d_2}_{i,k_4} z^{0}_{i,k_3} z^{d_1}_{k_2} z^{d_1}_{k_1} & \frac{1}{6}\\
\forestOJ & \sum^!_{k_1\geq k_3\geq k_2,k_4\geq k_5} Z^{d_2}_{i,k_5} Z^{d_2}_{i,k_4} z^{0}_{i,k_3} z^{d_1}_{k_2} z^{d_1}_{k_1} & 0\\
\forestPJ & \sum^!_{k_3\geq k_1\geq k_2,k_4\geq k_5} Z^{d_2}_{i,k_5} Z^{d_2}_{i,k_4} z^{0}_{i,k_3} z^{d_1}_{k_2} z^{d_1}_{k_1} & \frac{1}{3}\\
\forestQJ & \sum^!_{k_1\geq k_2\geq k_3,k_4\geq k_5} Z^{d_2}_{i,k_5} Z^{d_1}_{i,k_4} z^{0}_{i,k_3} z^{d_2}_{k_2} z^{d_1}_{k_1} & 0\\
\forestRJ & \sum^!_{k_1\geq k_2\geq k_3,k_4\geq k_5} Z^{d_1}_{i,k_5} Z^{d_2}_{i,k_4} z^{0}_{i,k_3} z^{d_2}_{k_2} z^{d_1}_{k_1} & 0\\
\forestSJ & \sum^!_{k_1\geq k_3\geq k_2,k_4\geq k_5} Z^{d_2}_{i,k_5} Z^{d_1}_{i,k_4} z^{0}_{i,k_3} z^{d_2}_{k_2} z^{d_1}_{k_1} & 0\\
\forestTJ & \sum^!_{k_1\geq k_3\geq k_2,k_4\geq k_5} Z^{d_1}_{i,k_5} Z^{d_2}_{i,k_4} z^{0}_{i,k_3} z^{d_2}_{k_2} z^{d_1}_{k_1} & 0\\
\forestUJ & \sum^!_{k_3\geq k_1\geq k_2,k_4\geq k_5} Z^{d_2}_{i,k_5} Z^{d_1}_{i,k_4} z^{0}_{i,k_3} z^{d_2}_{k_2} z^{d_1}_{k_1} & \frac{2}{3}\\
\forestVJ & \sum^!_{k_3\geq k_1\geq k_2,k_4\geq k_5} Z^{d_1}_{i,k_5} Z^{d_2}_{i,k_4} z^{0}_{i,k_3} z^{d_2}_{k_2} z^{d_1}_{k_1} & 0\\
\forestWJ & \sum^!_{k_1\geq k_2\geq k_3\geq k_4} z^{0}_{j,k_4} z^{0}_{i,k_3} z^{d_1}_{k_2} z^{d_1}_{k_1} & \frac{1}{6}\\
\forestXJ & \sum^!_{k_1\geq k_3\geq k_2\geq k_4} z^{0}_{j,k_4} z^{0}_{i,k_3} z^{d_1}_{k_2} z^{d_1}_{k_1} & 0\\
\forestYJ & \sum^!_{k_3\geq k_1\geq k_2\geq k_4} z^{0}_{j,k_4} z^{0}_{i,k_3} z^{d_1}_{k_2} z^{d_1}_{k_1} & \frac{1}{6}\\
\forestAK & \sum^!_{k_1\geq k_3\geq k_4\geq k_2} z^{0}_{j,k_4} z^{0}_{i,k_3} z^{d_1}_{k_2} z^{d_1}_{k_1} & 0\\
\forestBK & \sum^!_{k_3\geq k_1\geq k_4\geq k_2} z^{0}_{j,k_4} z^{0}_{i,k_3} z^{d_1}_{k_2} z^{d_1}_{k_1} & 0\\
\forestCK & \sum^!_{k_3\geq k_4\geq k_1\geq k_2} z^{0}_{j,k_4} z^{0}_{i,k_3} z^{d_1}_{k_2} z^{d_1}_{k_1} & \frac{1}{6}\\
\forestDK & \sum^!_{k_1\geq k_2\geq k_3\geq k_4} Z^{d_1}_{i,k_5} z^{0}_{j,k_4} z^{d_2}_{k_3} z^{d_2}_{k_2} z^{d_1}_{k_1} & 0\\
\forestEK & \sum^!_{k_1\geq k_2\geq k_4\geq k_3} Z^{d_1}_{i,k_5} z^{0}_{j,k_4} z^{d_2}_{k_3} z^{d_2}_{k_2} z^{d_1}_{k_1} & 0\\
\forestFK & \sum^!_{k_1\geq k_4\geq k_2\geq k_3} Z^{d_1}_{i,k_5} z^{0}_{j,k_4} z^{d_2}_{k_3} z^{d_2}_{k_2} z^{d_1}_{k_1} & 0\\
\forestGK & \sum^!_{k_4\geq k_1\geq k_2\geq k_3} Z^{d_1}_{i,k_5} z^{0}_{j,k_4} z^{d_2}_{k_3} z^{d_2}_{k_2} z^{d_1}_{k_1} & \frac{1}{3}\\
\forestHK & \sum^!_{k_1\geq k_2\geq k_3\geq k_4} Z^{d_1}_{i,k_5} z^{0}_{j,k_4} z^{d_2}_{k_3} z^{d_1}_{k_2} z^{d_2}_{k_1} & 0\\
\forestIK & \sum^!_{k_1\geq k_2\geq k_4\geq k_3} Z^{d_1}_{i,k_5} z^{0}_{j,k_4} z^{d_2}_{k_3} z^{d_1}_{k_2} z^{d_2}_{k_1} & 0\\
\forestJK & \sum^!_{k_1\geq k_4\geq k_2\geq k_3} Z^{d_1}_{i,k_5} z^{0}_{j,k_4} z^{d_2}_{k_3} z^{d_1}_{k_2} z^{d_2}_{k_1} & 0\\
\forestKK & \sum^!_{k_4\geq k_1\geq k_2\geq k_3} Z^{d_1}_{i,k_5} z^{0}_{j,k_4} z^{d_2}_{k_3} z^{d_1}_{k_2} z^{d_2}_{k_1} & 0\\
\forestLK & \sum^!_{k_1\geq k_2\geq k_3\geq k_4} Z^{d_1}_{i,k_5} z^{0}_{j,k_4} z^{d_1}_{k_3} z^{d_2}_{k_2} z^{d_2}_{k_1} & 0\\
\forestMK & \sum^!_{k_1\geq k_2\geq k_4\geq k_3} Z^{d_1}_{i,k_5} z^{0}_{j,k_4} z^{d_1}_{k_3} z^{d_2}_{k_2} z^{d_2}_{k_1} & \frac{1}{3}\\
\forestNK & \sum^!_{k_1\geq k_4\geq k_2\geq k_3} Z^{d_1}_{i,k_5} z^{0}_{j,k_4} z^{d_1}_{k_3} z^{d_2}_{k_2} z^{d_2}_{k_1} & 0\\
\forestOK & \sum^!_{k_4\geq k_1\geq k_2\geq k_3} Z^{d_1}_{i,k_5} z^{0}_{j,k_4} z^{d_1}_{k_3} z^{d_2}_{k_2} z^{d_2}_{k_1} & \frac{1}{3}\\
\forestPK & \sum^!_{k_1\geq k_2\geq k_3\geq k_4\geq k_5} z^{0}_{i,k_5} z^{d_2}_{k_4} z^{d_2}_{k_3} z^{d_1}_{k_2} z^{d_1}_{k_1} & \frac{1}{6}\\
\forestQK & \sum^!_{k_1\geq k_2\geq k_3\geq k_4\geq k_5} z^{0}_{i,k_5} z^{d_2}_{k_4} z^{d_1}_{k_3} z^{d_2}_{k_2} z^{d_1}_{k_1} & 0\\
\forestRK & \sum^!_{k_1\geq k_2\geq k_3\geq k_4\geq k_5} z^{0}_{i,k_5} z^{d_1}_{k_4} z^{d_2}_{k_3} z^{d_2}_{k_2} z^{d_1}_{k_1} & 0\\
\forestSK & \sum^!_{k_1\geq k_2\geq k_3\geq k_5\geq k_4} z^{0}_{i,k_5} z^{d_2}_{k_4} z^{d_2}_{k_3} z^{d_1}_{k_2} z^{d_1}_{k_1} & 0\\
\forestTK & \sum^!_{k_1\geq k_2\geq k_3\geq k_5\geq k_4} z^{0}_{i,k_5} z^{d_2}_{k_4} z^{d_1}_{k_3} z^{d_2}_{k_2} z^{d_1}_{k_1} & 0\\
\forestUK & \sum^!_{k_1\geq k_2\geq k_3\geq k_5\geq k_4} z^{0}_{i,k_5} z^{d_1}_{k_4} z^{d_2}_{k_3} z^{d_2}_{k_2} z^{d_1}_{k_1} & 0\\
\forestVK & \sum^!_{k_1\geq k_2\geq k_5\geq k_3\geq k_4} z^{0}_{i,k_5} z^{d_2}_{k_4} z^{d_2}_{k_3} z^{d_1}_{k_2} z^{d_1}_{k_1} & \frac{1}{6}\\
\forestWK & \sum^!_{k_1\geq k_2\geq k_5\geq k_3\geq k_4} z^{0}_{i,k_5} z^{d_2}_{k_4} z^{d_1}_{k_3} z^{d_2}_{k_2} z^{d_1}_{k_1} & 0\\
\forestXK & \sum^!_{k_1\geq k_2\geq k_5\geq k_3\geq k_4} z^{0}_{i,k_5} z^{d_1}_{k_4} z^{d_2}_{k_3} z^{d_2}_{k_2} z^{d_1}_{k_1} & 0\\
\forestYK & \sum^!_{k_1\geq k_5\geq k_2\geq k_3\geq k_4} z^{0}_{i,k_5} z^{d_2}_{k_4} z^{d_2}_{k_3} z^{d_1}_{k_2} z^{d_1}_{k_1} & 0\\
\forestAL & \sum^!_{k_1\geq k_5\geq k_2\geq k_3\geq k_4} z^{0}_{i,k_5} z^{d_2}_{k_4} z^{d_1}_{k_3} z^{d_2}_{k_2} z^{d_1}_{k_1} & 0\\
\forestBL & \sum^!_{k_1\geq k_5\geq k_2\geq k_3\geq k_4} z^{0}_{i,k_5} z^{d_1}_{k_4} z^{d_2}_{k_3} z^{d_2}_{k_2} z^{d_1}_{k_1} & 0\\
\forestCL & \sum^!_{k_5\geq k_1\geq k_2\geq k_3\geq k_4} z^{0}_{i,k_5} z^{d_2}_{k_4} z^{d_2}_{k_3} z^{d_1}_{k_2} z^{d_1}_{k_1} & \frac{1}{6}\\
\forestDL & \sum^!_{k_5\geq k_1\geq k_2\geq k_3\geq k_4} z^{0}_{i,k_5} z^{d_2}_{k_4} z^{d_1}_{k_3} z^{d_2}_{k_2} z^{d_1}_{k_1} & 0\\
\forestEL & \sum^!_{k_5\geq k_1\geq k_2\geq k_3\geq k_4} z^{0}_{i,k_5} z^{d_1}_{k_4} z^{d_2}_{k_3} z^{d_2}_{k_2} z^{d_1}_{k_1} & 0\\
\forestFL & \sum^!_{k_1\geq k_2\geq k_3\geq k_4\geq k_5\geq k_6} z^{d_3}_{k_6} z^{d_3}_{k_5} z^{d_2}_{k_4} z^{d_2}_{k_3} z^{d_1}_{k_2} z^{d_1}_{k_1} & \frac{1}{6}\\
\forestGL & \sum^!_{k_1\geq k_2\geq k_3\geq k_4\geq k_5\geq k_6} z^{d_3}_{k_6} z^{d_3}_{k_5} z^{d_2}_{k_4} z^{d_1}_{k_3} z^{d_2}_{k_2} z^{d_1}_{k_1} & 0\\
\forestHL & \sum^!_{k_1\geq k_2\geq k_3\geq k_4\geq k_5\geq k_6} z^{d_3}_{k_6} z^{d_3}_{k_5} z^{d_2}_{k_4} z^{d_1}_{k_3} z^{d_1}_{k_2} z^{d_2}_{k_1} & 0\\
\forestIL & \sum^!_{k_1\geq k_2\geq k_3\geq k_4\geq k_5\geq k_6} z^{d_3}_{k_6} z^{d_2}_{k_5} z^{d_3}_{k_4} z^{d_2}_{k_3} z^{d_1}_{k_2} z^{d_1}_{k_1} & 0\\
\forestJL & \sum^!_{k_1\geq k_2\geq k_3\geq k_4\geq k_5\geq k_6} z^{d_3}_{k_6} z^{d_2}_{k_5} z^{d_3}_{k_4} z^{d_1}_{k_3} z^{d_2}_{k_2} z^{d_1}_{k_1} & 0\\
\forestKL & \sum^!_{k_1\geq k_2\geq k_3\geq k_4\geq k_5\geq k_6} z^{d_3}_{k_6} z^{d_2}_{k_5} z^{d_3}_{k_4} z^{d_1}_{k_3} z^{d_1}_{k_2} z^{d_2}_{k_1} & 0\\
\forestLL & \sum^!_{k_1\geq k_2\geq k_3\geq k_4\geq k_5\geq k_6} z^{d_3}_{k_6} z^{d_2}_{k_5} z^{d_2}_{k_4} z^{d_3}_{k_3} z^{d_1}_{k_2} z^{d_1}_{k_1} & 0\\
\forestML & \sum^!_{k_1\geq k_2\geq k_3\geq k_4\geq k_5\geq k_6} z^{d_3}_{k_6} z^{d_2}_{k_5} z^{d_2}_{k_4} z^{d_1}_{k_3} z^{d_3}_{k_2} z^{d_1}_{k_1} & 0\\
\forestNL & \sum^!_{k_1\geq k_2\geq k_3\geq k_4\geq k_5\geq k_6} z^{d_3}_{k_6} z^{d_2}_{k_5} z^{d_2}_{k_4} z^{d_1}_{k_3} z^{d_1}_{k_2} z^{d_3}_{k_1} & 0\\
\forestOL & \sum^!_{k_1\geq k_2\geq k_3\geq k_4\geq k_5\geq k_6} z^{d_3}_{k_6} z^{d_2}_{k_5} z^{d_1}_{k_4} z^{d_3}_{k_3} z^{d_2}_{k_2} z^{d_1}_{k_1} & 0\\
\forestPL & \sum^!_{k_1\geq k_2\geq k_3\geq k_4\geq k_5\geq k_6} z^{d_3}_{k_6} z^{d_2}_{k_5} z^{d_1}_{k_4} z^{d_3}_{k_3} z^{d_1}_{k_2} z^{d_2}_{k_1} & 0\\
\forestQL & \sum^!_{k_1\geq k_2\geq k_3\geq k_4\geq k_5\geq k_6} z^{d_3}_{k_6} z^{d_2}_{k_5} z^{d_1}_{k_4} z^{d_2}_{k_3} z^{d_3}_{k_2} z^{d_1}_{k_1} & 0\\
\forestRL & \sum^!_{k_1\geq k_2\geq k_3\geq k_4\geq k_5\geq k_6} z^{d_3}_{k_6} z^{d_2}_{k_5} z^{d_1}_{k_4} z^{d_2}_{k_3} z^{d_1}_{k_2} z^{d_3}_{k_1} & 0\\
\forestSL & \sum^!_{k_1\geq k_2\geq k_3\geq k_4\geq k_5\geq k_6} z^{d_3}_{k_6} z^{d_2}_{k_5} z^{d_1}_{k_4} z^{d_1}_{k_3} z^{d_3}_{k_2} z^{d_2}_{k_1} & 0\\
\forestTL & \sum^!_{k_1\geq k_2\geq k_3\geq k_4\geq k_5\geq k_6} z^{d_3}_{k_6} z^{d_2}_{k_5} z^{d_1}_{k_4} z^{d_1}_{k_3} z^{d_2}_{k_2} z^{d_3}_{k_1} & 0\\
\caption{Order conditions of frozen flow methods for weak order 3. The expectations and sums that do not satisfy inequalities are omitted for conciseness.}
\label{table:weak_order_conditions_order_3}
\end{longtable}

\end{appendices}

\end{document}